\documentclass[10pt]{article}
\usepackage[a4paper]{geometry}
\usepackage{amsthm}
\usepackage{amsmath,amssymb}
\usepackage{bm,xspace,enumerate,color}
\usepackage{graphicx,verbatim,mathtools}
\usepackage[only,llbracket,rrbracket,llparenthesis,rrparenthesis]{stmaryrd}
\usepackage[numbers]{natbib}
\usepackage{mathrsfs}
\usepackage[normalem]{ulem} 
\usepackage[textsize=tiny,color=blue!10!white]{todonotes}
\usepackage[breaklinks,bookmarks=false]{hyperref}

\hypersetup{colorlinks, linkcolor=blue, citecolor=blue,
urlcolor=blue, plainpages=false, pdfwindowui=false,
pdfstartview={FitH}}

\numberwithin{equation}{section}

\newtheorem{theorem}{Theorem}[section]{\bfseries}{\it}
{\bfseries}{\it}
\newtheorem{lemma}[theorem]{Lemma}{\bfseries}{\it}
\newtheorem{corollary}[theorem]{Corollary}{\bfseries}{\it}
\newtheorem{definition}{Definition}[section]{\bfseries}{\it}
{\bfseries}{\it}
\theoremstyle{definition}
\newtheorem{remark}{Remark}[section]{\bfseries}{\rmfamily}

\renewcommand{\dim}{d}
\newcommand{\p}{\partial}
\newcommand{\abs}[1]{\lvert#1\rvert} 
\newcommand{\tends}{\rightarrow}
\newcommand{\norm}[1]{\lVert#1\rVert}
\newcommand{\R}{\mathbb{R}}
\newcommand{\N}{\mathbb{N}}

\newcommand{\cF}{\mathcal{F}}
\newcommand{\cT}{\mathcal{T}}

\newcommand{\cS}{\mathcal{S}}
\newcommand{\eps}{\epsilon}

\newcommand{\avg}[1]{\left\{#1\right\}}
\newcommand{\jump}[1]{\llbracket #1 \rrbracket}
\newcommand{\Om}{\Omega}
\newcommand{\om}{\omega}
\newcommand{\ab}{{\alpha\beta}}
\newcommand{\sA}{\mathscr{A}}
\newcommand{\sB}{\mathscr{B}}
\newcommand{\al}{\alpha}
\newcommand{\be}{\beta}
\newcommand{\Fop}{F}
\DeclareMathOperator{\Tr}{Tr}
\DeclareMathOperator{\diam}{diam}
\DeclareMathOperator{\Div}{div}
\DeclareMathOperator{\supp}{supp}

\newcommand{\bn}{\bm{n}}
\newcommand{\Fp}{\cF^+}
\newcommand{\Fpi}{\cF^{I+}}
\newcommand{\Fpb}{\cF^{B+}}
\newcommand{\Tp}{\cT^+}
\newcommand{\Tk}{\cT_k}
\newcommand{\Omp}{\Omega^+}
\newcommand{\Omm}{\Omega^-}
\newcommand{\Tkj}{\Tk^{j+}}
\newcommand{\Fkp}{\cF_k^+}
\newcommand{\Fk}{\cF_k}
\newcommand{\Npw}{\nabla}
\newcommand{\Dpw}{\nabla^2}

\newcommand{\Sp}{\cS^+}
\newcommand{\Sk}{\cS_k}
\newcommand{\Skp}{\cS_k^+}
\newcommand{\Pp}{\mathbb{P}_p}
\newcommand{\normk}[1]{\norm{#1}_{k}}

\newcommand{\Vk}{V_k}
\newcommand{\pair}[1]{\langle#1\rangle}
\newcommand{\HD}{H^2_D(\Om;\Tp)}
\newcommand{\Dp}{\Npw}
\newcommand{\Hessp}{\Dpw}
\newcommand{\Tkp}{\Tk^+}
\newcommand{\Vinfty}{V_\infty}
\newcommand{\Tkm}{\Tk^-}
\newcommand{\Fki}{\Fk^I}
\newcommand{\Fkb}{\Fk^B}
\newcommand{\Fkip}{\Fk^{I+}}
\newcommand{\br}{\bm{r}}
\newcommand{\Flp}{\cF^{\dagger}_{\ell}}
\newcommand{\rk}{\br_k(\jump{\Npw v_k})}
\newcommand{\Flpi}{\cF^{I\dagger}_{\ell}}
\newcommand{\Cid}{C^\infty_0\big(\Om;\R^\dim\big)}
\newcommand{\CidR}{C^\infty_0\big(\R^\dim;\R^{\dim}\big)}
\newcommand{\Cidd}{C^\infty_0\big(\Om;\R^{\dim\times\dim}\big)}
\newcommand{\bvphi}{\bm{\varphi}}
\newcommand{\bphi}{\bm{\phi}}
\newcommand{\Ld}{L^2(\Om;\R^\dim)}
\newcommand{\Ldd}{L^2(\Om;\R^{\dim\times\dim})}
\newcommand{\Fks}{\cF^{\dagger}_k}
\newcommand{\Fksi}{\cF^{I\dagger}_k}
\newcommand{\Vq}{V_{k,q}^0}
\newcommand{\Vct}{\mathrm{HCT}(K)}
\newcommand{\Spi}{\mathcal{S}^{I+}}
\newcommand{\Skpi}{\mathcal{S}^{I+}_k}
\newcommand{\Jpen}{J_k^{\sigma,\rho}}
\newcommand{\Fg}{F_{\gamma}}
\newcommand{\absJ}[1]{\abs{#1}_{J,k}}
\newcommand{\ek}{\eta_k}
\newcommand{\Mk}{\mathcal{M}_k}
\newcommand{\Hinf}{\bm{H}_\infty}
\newcommand{\Hk}{\bm{H}_k}

\newcommand{\Skf}{S_k}
\newcommand{\Sinfty}{S_\infty}
\newcommand{\cmon}{C_{\mathrm{mon}}}
\newcommand{\tauk}{\tau_{\p K}}
\newcommand{\Fqed}{}

  \title{Convergence of adaptive discontinuous Galerkin and $C^0$-interior penalty finite element methods for Hamilton--Jacobi--Bellman and Isaacs equations\footnotemark[1]}
  \author{Ellya L.~Kawecki\footnotemark[2]~ and Iain Smears\footnotemark[2]}

\begin{document}

 \maketitle

 \renewcommand{\thefootnote}{\fnsymbol{footnote}}
\footnotetext[1]{This work was supported by an Engineering and Physical Sciences Research Council (EPSRC) Doctoral Prize Fellowship under grant EP/R513143/1.}
\footnotetext[2]{Department of Mathematics, University College London, Gower
Street, WC1E 6BT London, United Kingdom (\texttt{e.kawecki@ucl.ac.uk}, \texttt{i.smears@ucl.ac.uk}).}

\begin{abstract}
We prove the convergence of adaptive discontinuous Galerkin and $C^0$-interior penalty methods for fully nonlinear second-order elliptic Hamilton--Jacobi--Bellman and Isaacs equations with Cordes coefficients.
We consider a broad family of methods on adaptively refined conforming simplicial meshes in two and three space dimensions, with fixed but arbitrary polynomial degrees greater than or equal to two.
A key ingredient of our approach is a novel intrinsic characterization of the limit space that enables us to identify the weak limits of bounded sequences of nonconforming finite element functions.
We provide a detailed theory for the limit space, and also some original auxiliary functions spaces, that is of independent interest to adaptive nonconforming methods for more general problems, including Poincar\'e and trace inequalities, a proof of the density of functions with nonvanishing jumps on only finitely many faces of the limit skeleton, approximation results by finite element functions and weak convergence results.
\end{abstract}

\tableofcontents

\section{Introduction}
We study the convergence of a broad class of adaptive discontinuous Galerkin (DG) and $C^0$-interior penalty (IP) finite element methods (FEM) for second-order fully nonlinear Isaacs equations, with a homogeneous Dirichlet boundary condition, of the form
\begin{equation}\label{eq:isaacs}
\begin{aligned}
F[u]\coloneqq \inf_{\alpha\in\sA}\sup_{\beta\in\sB}\left[L^\ab u-f^\ab\right] & = 0 &&\text{in }\Om,\\ 
u & = 0  && \text{on }\p\Om,
\end{aligned}
\end{equation}
where $\Om$ is a nonempty bounded convex polytopal open set in $\R^\dim$, $\dim\in\{2,3\}$, where $\sA$ and $\sB$ are nonempty compact metric spaces, and where the second-order nondivergence form elliptic operators $L^\ab$, $\al\in\sA,\be\in\sB$, are defined in~\eqref{eq:Lab_operators} below.
It is equally possible to consider Isaacs equations with the reverse order of the infimum and supremum in~\eqref{eq:isaacs}.
Isaacs equations arise in models of two-player stochastic differential games. 
If $\sA$ is a singleton set, then the Isaacs equation~\eqref{eq:isaacs} reduces to a Hamilton--Jacobi--Bellman (HJB) equation for the value function of the associated stochastic optimal control problem~\cite{FlemingSoner06}.
These equations find applications in a wide range of fields, such as engineering, energy, finance and computer science.
HJB and Isaacs equations are important examples of~\emph{fully nonlinear} partial differential equations (PDE), where the nonlinearity includes the second-order partial derivatives of the unknown solution, thereby prohibiting standard approaches via weak formulations that are commonly employed for divergence-form elliptic problems.
Several other important nonlinear PDE can be reformulated as Isaacs or HJB equations, including for instance the Monge--Amp\`ere equation~\cite{FengJensen17,Krylov87}; see also~\cite{Kawecki2018a}. 

There still remain significant challenges in the design and analysis of stable, efficient and accurate numerical methods for fully nonlinear PDE such as~\eqref{eq:isaacs}.
Numerical methods that enjoy a discrete maximum principle can be shown to converge to the exact solution, in the sense of viscosity solutions, under rather general conditions which in particular allow the treatment of possibly degenerate elliptic problems~\cite{Souganidis91,CrandallIshiiLions92,KuoTrudinger1990,KushnerDupuis01}.
However, it is well-known that the need for a discrete maximum principle leads to significant costs in terms of computational efficiency, in terms of the order of accuracy, the flexibility of the grids and the locality of the stencils for strongly anisotropic diffusions~\cite{CrandallLions96,Kocan95,MotzkinWasow53}.
We refer the reader to \cite{DebrabantJakobsen13,FengJensen17,JensenS13,NochettoZhang18,SalgadoZhang19} for recent results and further discussion on this class of numerical methods.

Recently there has been significant interest in the design and analysis of methods that do not require discrete maximum principles for fully nonlinear PDE.
However, designing provably stable and convergent methods  without a discrete maximum principle remains generally challenging.
In the series of papers~\cite{SS13,SS14,SS16}, this obstacle was overcome in the context of fully nonlinear HJB equations that satisfy the Cordes condition~\cite{Cordes1956,MaugeriPalagachevSoftova00}, which is an algebraic condition on the coefficients of the differential operator.
In particular, for fully nonlinear HJB equations on convex domains with Cordes coefficients, existence and uniqueness of the strong solution in $H^2(\Om)\cap H^1_0(\Om)$ was proved in~\cite{SS14} using a variational reformulation in terms of a strongly monotone operator equation.
It was then shown in~\cite{SS13,SS14} that the structure of the continuous problem can be preserved under discretization, forming the basis for a provably stable $hp$-version discontinuous Galerkin (DG) finite element method (FEM), with stability achieved in a mesh-dependent $H^2$-type norm, and with optimal convergence rates with respect to the mesh-size, and only half-order suboptimal rates with respect to the polynomial degree, under suitable regularity assumptions.
Moreover, the method was shown to be stable for general shape-regular simplicial and parallelipipedal meshes in arbitrary dimensions, thus opening the way towards adaptive refinements.
These results were then extended to the parabolic setting in~\cite{SS16}.
This approach has sparked significant recent activity exploring a range of directions, including $H^2$-conforming and mixed methods~\cite{Gallistl17,Gallistl19}, preconditioners~\cite{S18}, $C^0$-IP methods~\cite{Bleschmidt19,Kawecki19c,NeilanWu19}, curved elements~\cite{Kawecki19b}, and other types of boundary conditions~\cite{Gallistl19b,Kawecki19}.
Note that in the context of these problems, DG and $C^0$-IP methods are examples of nonconforming methods, since the appropriate functional setting is in $H^2$-type spaces.
In~\cite{KaweckiSmears20}, we provide a unified analysis of \emph{a posteriori} and \emph{a priori} error bounds for a wide family of DG and $C^0$-IP methods, where we also show that the original method of~\cite{SS13,SS14}, along with many related variants, are quasi-optimal in the sense of  near-best approximations without any additional regularity assumptions, along with convergence in the small mesh-limit for minimal regularity solutions.

We are interested here in \emph{adaptive} methods for Isaacs and HJB equations based on successive mesh refinements driven by computable error estimators.
The first work on adaptivity for these problems is due to Gallistl~\cite{Gallistl17,Gallistl19}, who proved convergence of an adaptive scheme for some $C^1$-conforming and mixed method approximations.
In particular, the analysis there follows the framework of~\cite{MorinSiebertVeeser08}, where the key tool in the proof of convergence is the introduction of a suitable limit problem posed on a limit space of the adaptive approximation spaces, and a proof of convergence of the numerical solutions to the limit problem.
Note that in the case of nested conforming approximations, the limit space is obtained simply by closure of the sequence of approximation spaces with respect to the norm; however many standard $C^1$-conforming elements, such as Argyris or Hsieh--Clough--Tocher (HCT) elements, do not lead to nested spaces in practice. 
More broadly, the analysis of adaptive methods for Isaacs and HJB equations is still in its infancy, and the analysis of rates of convergence of the adaptive algorithms remains open.

Even in the case of linear divergence-form equations, the construction and analysis of the corresponding limit spaces for adaptive nonconforming methods is less obvious than for the conforming methods, and this was only recently addressed by Kreuzer \& Georgoulis in~\cite{KreuzerGeorgoulis18} for DGFEM discretizations of divergence-form second-order elliptic equations.
Their approach has been extended to $C^0$-IP methods for the biharmonic equation in~\cite{DominincusGaspozKreuzer19}; we refer the reader to these references for further discussion of the literature on adaptivity for DGFEM for other PDE.
An advantage of the approach of~\cite{DominincusGaspozKreuzer19,KreuzerGeorgoulis18} is that the analysis encompasses all choices of the penalty parameters that are sufficient for stability of the methods.
Note that a further difficulty for the analysis of adaptive methods for both the biharmonic problem in~\cite{DominincusGaspozKreuzer19} and also for the fully nonlinear HJB and Isaacs equations considered here is the general absence of a sufficiently rich $H^2$-conforming subspace for DG and $C^0$-IP methods, which prevents a range of techniques employed in $H^1$-type settings~\cite{KarakashianPascal07,KreuzerGeorgoulis18}.

In this paper, we analyse in a single framework a broad family of DG and $C^0$-IP methods that are based on the original method of \cite{SS13,SS14} and recent variants.
These methods have significant advantages over $C^1$-conforming elements in terms of practicality, flexibility and computational cost.
They also require fewer unknowns than mixed methods.
We prove the plain convergence of a class of adaptive DG and $C^0$-IP methods on conforming simplicial meshes in two and three space dimensions for fixed but arbitrary polynomial degrees greater than or equal to two, and for all choices of penalty parameters that are sufficient for stability of the discrete problems.
Similar to~\cite{DominincusGaspozKreuzer19,Gallistl17}, the only condition on the marking strategy is that the set of elements marked for refinement at each step must include the element with maximum error estimator; in practice this allows for all relevant marking strategies.

In addition, we make several wider contributions to the general analysis of adaptive nonconforming methods in order to overcome some critical challenges appearing in the analysis, as we now explain.
The bedrock of our strategy for proving convergence of the adaptive methods is in the spirit of monotone operator theory: by showing weak precompactness in a suitable sense for the bounded sequence of numerical solutions, and by showing the asymptotic consistency of the numerical scheme, we use a \emph{strong times weak convergence} argument and the strong monotonicity of the problem to turn weak convergence of subsequences of numerical solutions into strong convergence of the whole sequence to the solution of the limit problem. 
However, this step rests upon a proof that the weak limits of bounded sequences of finite element functions indeed belongs to the correct limit space, which, in the existing approaches of~\cite{DominincusGaspozKreuzer19,KreuzerGeorgoulis18}, requires a proof that the weak limit can also be approximated by a strongly convergent sequence of finite element functions.
Note that this is handled in~\cite{DominincusGaspozKreuzer19} for piecewise quadratic $C^0$-IP methods in two space dimensions using rather specific relations between the degrees of freedom of quadratic $C^0$-Lagrange elements and 4th-order HCT elements.
However, the extension to DG methods represents a significant challenge, which we resolve here in a unified way for both DG and $C^0$-IP methods in both two and three space dimensions.
A key ingredient of our analysis is a novel approach to the construction and analysis of the limit spaces, namely we provide intrinsic characterizations of the limit spaces, without reference to strongly approximating sequences of finite element functions.
This constitutes a foundational change from~\cite{DominincusGaspozKreuzer19,KreuzerGeorgoulis18} in terms of how we approach the analysis.
In particular, starting in Section~\ref{sec:limspace}, we define the limit spaces, along with some related more general first- and second-order spaces, directly via characterizations of the distributional derivatives of the function and its gradient and via appropriate integrability properties, see~Definitions~\ref{def:H1limitspaces},~\ref{def:HD_def} and~\ref{def:limit_space} of Section~\ref{sec:limspace} below.
This is done in the spirit of the definition of Sobolev spaces in terms of weak derivatives.
Some further benefits of this approach are significant simplifications in the theory, especially with regard to completeness of the spaces and weak precompactness of bounded sequences of finite element functions, as well as a broader understanding of the nature of the limit spaces.
We stress that this approach is by no means limited to HJB and Isaacs equations, and it is of general interest to the analysis of nonconforming adaptive methods for more general problems.

Our intrinsic approach to the limit spaces ultimately connects to~\cite{DominincusGaspozKreuzer19,KreuzerGeorgoulis18} since we also prove that the functions in the limit spaces are also limits of strongly converging sequences of finite element functions, see Theorem~\ref{thm:limit_space_characterization}.
This requires addressing a particular fundamental difficulty in the case of DG methods, as we now explain.
For DG methods, the limit space can be seen as a specific subspace of $SBV^2(\Om)$, where $SBV^2(\Om)$ denotes the space of functions of special bounded variation \cite{DeGiorgiAmbrosio88} with gradient density also of special bounded variation, see e.g.\ \cite{FonsecaLeoniParoni05} for a precise definition.
A surprising result due to~\cite{FonsecaLeoniParoni05}, based on an earlier result from~\cite{Alberti1991}, is that in general there exists functions in $SBV^2(\Om)$ with nonsymmetric Hessians, and it is easy to see that such functions cannot be strong limits in the required sense of finite element functions.
One of our key results here is that the intrinsic properties the limit space, in particular the integrability properties and the structure of the jump sets, are sufficient to guarantee the symmetry of the Hessians and thereby rule out such pathological functions.
The key step in the analysis is a crucial approximation result, namely the density of the subspace of functions with only finitely many jumps over the set of faces that are never refined, see~Theorem~\ref{thm:finite_approx} below, which we use to prove the symmetry of the Hessians of these functions in~Corollary~\ref{cor:H2_omm_restriction}. 
These results are obtained without \emph{a priori} knowledge of the existence of strongly convergent sequences of finite element functions, and thus resolves the challenge highlighted above.

The paper is organised as follows. Section~\ref{sec:notation} sets the notation and defines the DG and $C^0$-IP finite element spaces.
In section~\ref{sec:var_fem} we state our main assumptions on the problem~\eqref{eq:isaacs}, and recall some well-posedness results from~\cite{SS14,KaweckiSmears20}.
Section~\ref{sec:var_fem} then introduces the family of adaptive DG and $C^0$-IP methods that are considered, and states our main result on convergence of the adaptive algorithm in Theorem~\ref{thm:main}.
In Section~\ref{sec:limspace} we study the limit spaces as described above, and in Section~\ref{sec:limit_problem_proof} we introduce the limit problem, and prove our main result on the convergence of the adaptive algorithm.

  \section{Notation}\label{sec:notation}
Let $\Om\subset \R^\dim$ be a bounded convex polytopal open set in $\R^\dim$, $\dim\in\{2,3\}$.
For a Lebesgue measurable set $\om \subset \R^\dim$, let $\abs{\om}$ denote its Lebesgue measure, and let $\diam(\om)$ denote its diameter. The $L^2$-norm of functions over $\omega$ is denoted by $\norm{\cdot}_{\om}$.
For two vectors $\bm{v}$ and $\bm{w}\in \R^\dim$, let $\bm{v}\otimes\bm{w}\in \R^{\dim\times\dim}$ be defined by $(\bm{v}\otimes\bm{w})_{ij}=\bm{v}_i \bm{w}_j$.
Let~$\{\Tk\}_{k\in\N}$ be a shape-regular sequence of conforming simplicial meshes on $\Om$. We have in mind sequences of meshes $\{\Tk\}_{k\in\N}$ that are obtained by successive refinements without coarsening from an initial mesh $\cT_1$. More precisely, we assume the framework of~\cite{MorinSiebertVeeser08} of \emph{unique quasi-regular element subdivisions.}  
The adaptive process that determines the mesh refinement is presented in Section~\ref{sec:var_fem} below. For real numbers $a$ and $b$, we write $a\lesssim b$ if there exists a constant $C$ such that $a\leq C b$, where $C$ depends only on the dimension $\dim$, the domain $\Om$, and on the shape-regularity of the meshes and on the polynomial degrees $p$ and $q$ defined below, but is otherwise independent of all other quantities. We write $a\eqsim b$ if and only if $a\lesssim b$ and $b\lesssim a$.
For each $k\in\N$, let $\Fk$ denote the set of $\dim-1$ dimensional faces of the mesh~$\Tk$, and let $\Fk^I$ and~$\Fk^B$ denote the set of internal and boundary faces of $\Tk$ respectively.
Let $\Sk$ denote the \emph{skeleton} of the mesh~$\Tk$, i.e.\ $\Sk\coloneqq\bigcup_{F\in\Fk} F$, and let $\cS^I_k\coloneqq \bigcup_{F\in\Fki}F$ denote the internal skeleton of $\Tk$.
For each $F\in\Fk$, $k\in\N$, let $\bn_F$ be a fixed choice of unit normal vector to $F$, where the choice of unit normal must be independent of $k$ and solely dependent on $F$.
If $F$ is a boundary face then $\bn_F$ is chosen to be the outward normal to~$\Om$.
In a slight abuse of notation, we shall usually drop the subscript and simply write $\bn$ when there is no possibility of confusion.
For each $K\in \Tk$, $k\in\N$, let $h_K \coloneqq \abs{K}^{\frac{1}{\dim}}$; note that shape-regularity of the meshes imply that $h_K \eqsim \diam(K)$.
For each $F\in\Fk$, let $h_F \coloneqq \left(\mathcal{H}^{\dim-1}(F)\right)^{\frac{1}{\dim-1}}$, where $\mathcal{H}^{\dim-1}$ denotes the $(d-1)$-dimensional Hausdorff measure.
Shape-regularity also implies that $h_K\eqsim h_F$ for any element $K\in\Tk$ and any face $F\in\Fk$ contained in $K$.
Similarly, shape-regularity implies that $h_F\eqsim \diam(F)$ for all $F\in\Fk$, $k\in\N$.
For each $k\in\N$, we define the global mesh-size function $h_k\colon \overline{\Om}\tends\R$ by $h_k|_{K^\circ}=h_K$ for each $K\in\Tk$, where $K^\circ$ denotes the interior of $K$, and $h_k|_F=h_F$ for each $F\in\Fk$.
The functions $\{h_k\}_{k\in\N}$ are uniformly bounded in $\Om$ and are only defined up to sets of zero $\mathcal{H}^{d-1}$-measure, which will be sufficient for our purposes.
We say that two elements are neighbours if they have nonempty intersection. 
For each $K\in \Tk$ and $j\in \N_0$, we define the set $N_k^j(K)$ of $j$-th neighbours of $K$ recursively by setting $N_k^0(K) \coloneqq K$, and then defining $N_k^j(K)$ as the set of all elements in $\Tk$ that are neighbours of at least one element in $N_k^{j-1}(K)$.
For the case $j=1$ we drop the superscript and write $N_k^1(K)=N_k(K)$.

It will be frequently convenient to use a shorthand notation for integrals over collections of elements and faces of the meshes.
For collections of elements $\mathcal{E}\subset \bigcup_{k\in\N}\Tk$ that are disjoint up to sets of $\dim$-dimensional Lebesgue measure zero, we write $\int_{\mathcal{E}} \coloneqq \sum_{K\in\mathcal{E}}\int_K$, where the measure of integration is the Lebesgue measure on $\R^\dim$.
Likewise, if $\mathcal{G}\subset \bigcup_{k\in\N} \Fk$ is a collection of faces that are disjoint up to sets of zero $\mathcal{H}^{\dim-1}$-measure, then we write $\int_{\mathcal{G}} \coloneqq \sum_{F\in\mathcal{G}}\int_F$, where the measure of integration is the $(\dim-1)$-dimensional Hausdorff measure on $\R^\dim$.
Note that in the case where $\mathcal{E}$ or $\mathcal{G}$ are countably infinite, the notation $\int_{\mathcal{E}}$ and $\int_{\mathcal{G}}$ represent infinite series whose convergence will be determined as necessary. We do not write the measure of integration as there is no risk of confusion.

\subsection{Derivatives and traces of functions of bounded variation.}\label{sec:BV}
We recall some known results about spaces of functions of bounded variation~\cite{AmbrosioFuscoPallara00,EvansGariepy2015}. For an open set $\om\subset\Om$, let $BV(\om)$ denote the space of real-valued functions of bounded variation on $\omega$.
Recall that $BV(\om)$ is a Banach space equipped with the norm $\norm{v}_{BV(\om)}\coloneqq\norm{v}_{L^1(\om)}+\abs{Dv}(\om)$, where $\abs{Dv}(\om)$ denotes the total variation of its distributional derivative $Dv$ over $\omega$, defined by $\abs{Dv}(\om)\coloneqq \sup\left\{\int_\om v \Div \bphi\colon \bphi\in C^\infty_0(\om;\R^\dim),\|\bphi\|_{C(\overline{\om};\mathbb R^d)}\le1 \right\}$.
To simplify the notation below, we also define $BV(\overline{\om})\coloneqq BV(\om)$ where $\overline{\omega}$ is the closure of $\omega$.
In the following, we shall frequently have to handle functions of bounded variation that are typically only piecewise regular over different and possibly infinite subdivisions of~$\Omega$, and the analysis is greatly simplified by adopting a notation that unifies and generalises various familiar concepts of weak and piecewise derivatives.
In particular we follow the notation of~\cite{FonsecaLeoniParoni05}.
For any $v\in BV(\Omega)$, the distributional derivative $Dv$ can be identified with a Radon measure on $\Omega$ that can be decomposed into the sum of an absolutely continuous part with respect to Lebesgue measure, and a singular part; the density of the absolutely continuous part of $D v$ with respect to Lebesgue measure is denoted by
\begin{equation}\label{eq:nabla_notation}
\Npw v =(\nabla_{x_1} v,\dots \nabla_{x_\dim} v) \in L^1(\Om;\R^\dim).
\end{equation}
Following \cite{FonsecaLeoniParoni05}, for functions $v\in BV(\Om)$ such that $\Npw v \in BV(\Om;\R^\dim)$, we define $\Dpw v$ as the density of the absolutely continuous part of $D(\nabla v)$, the distributional derivative of $\nabla v$; in particular, 
\begin{equation}\label{eq:Hessian_notation}
\begin{aligned}
\nabla^2 v \coloneqq \nabla(\nabla v) \in L^1(\Om;\R^{\dim\times\dim}), &&&
(\nabla^2 v)_{ij} \coloneqq \nabla_{x_j} (\nabla_{x_i} v) \quad \forall i,\, j\in \{1,\dots,\dim\}.
\end{aligned}
\end{equation}
We then define the Laplacian $\Delta v = \Tr \Dpw v$, where $\Tr \bm{M}\coloneqq\sum_{i=1}^\dim \bm{M}_{ii}$ is the matrix trace for $\bm{M}\in\R^{\dim\times\dim}$.
We emphasize that $\Dpw v$ is defined in terms of $D(\nabla v)$ and not $D^2v$, the second distributional derivative of $v$, since in general $D^2 v$ is not necessarily a Radon measure.
Crucially, there is no conflict of notation here when considering Sobolev regular functions, since $\Npw v$ coincides with the weak gradient of $v$ if $v\in W^{1,1}(\Om)$ and $\Dpw v$ coincides with the weak Hessian of $v$ if $v\in W^{2,1}(\Om)$.
Moreover, for functions from the DG and $C^0$-IP finite element spaces defined shortly below, it is easy to see that the gradient and Hessian as defined above coincide with the piecewise gradient and Hessian over elements of the mesh.
Therefore, the above notation \emph{unifies} and \emph{generalises} the above notions of derivatives.
Furthermore, the more general notions of gradients and Hessians defined above play a key role in the formulation of intrinsic definitions of the limit spaces of the sequence of finite element spaces given in Section~\ref{sec:limspace}. 

\paragraph{Jump and average operators.}
We recall some known results concerning one-sided traces of functions of bounded variation.
It follows from \cite[Theorems~5.6 \& 5.7]{EvansGariepy2015} that for each interior face $F\in\Fki$, $k\in\N$, there exist bounded one-sided trace operators $\tau_F^+\colon BV(\Om)\tends L^1(F)$ and $\tau_F^-\colon BV(\Om)\tends L^1(F)$, where the notation $\tau_F^\pm$ is determined by the chosen unit normal $\bn_F$ so that $\tau_F^-$ and $\tau_F^+$ are the traces from the sides of $F$ for which $\bn_F$ is outward pointing and inward pointing, respectively.
If $F$ is a boundary face, we only define its interior trace $\tau_F^-$, where it is recalled that $\bn_F$ is outward pointing to $\Om$.
In particular, \cite[Theorem~5.7]{EvansGariepy2015} shows that, for any $v\in BV(\Om)$, we have $\tau_F^\pm v(x) =\lim_{r\tends 0} \frac{1}{\abs{B_{\pm}(x,r)}} \int_{B_{\pm}(x,r)} v$ for $\mathcal{H}^{\dim-1}$-a.e.\ $x\in F$, where $B_{\pm}(x,r)\coloneqq\{y\in \Om\colon \abs{x-y}<r,\, (y-x)\cdot \bn_F \in\R_{\pm} \}$ are half-balls centred on $x$ of radius $r$, for which $\bn_F$, and where $\R_+$ and $\R_-$ denote the sets of nonnegative and nonpositive real numbers, respectively.
Therefore, the values of the traces do not depend on a choice of surrounding element from any particular mesh.
However, the $L^1$-norm of traces on faces can be bounded in terms of the BV-norm on elements as follows. For each element $K\in \Tk$, $k\in\N$, let $\tauk \colon BV(K)\tends L^1(\p K)$ denote the corresponding trace operator from $K$ to $\p K$.
For instance, if $F$ is a face and if $K$ is an element containing $F$ for which $\bn_F$ is outward pointing, then $\norm{\tau_F^- v}_{L^1(F)}\leq \norm{\tauk v}_{L^1(\p K)}\lesssim \abs{Dv}(K)+\frac{1}{h_K}\norm{v}_{L^1(K)}$ for all $v\in BV(K)$; a similar bound holds for $\tau_F^+$ if $\bn_F$ is inward pointing with respect to $K$.
In other words, the $L^1$-norm of the appropriate one-sided trace is bounded by the BV-norm of a function over the element containing the face.
We now define jump and average operators over faces.
For $v\in BV(\Om)$, we define the jump $\jump{v}_F\in L^1(F)$ and average of $\avg{v}_F\in L^1(F)$ for each $F\in\Fk$ by 
\begin{equation}\label{eq:jumpavg}
\begin{aligned}
  \avg{v}_F&\coloneqq \frac{1}{2}\left(\tau_{F}^+v+\tau_{F}^- v\right), & \jump{v}_F & \coloneqq \tau_{F}^-v -\tau_{F}^+v, &\forall F\in\Fki,\\
  \avg{v}_F &\coloneqq \tau_{F}^- v &  \jump{v}_F & \coloneqq \tau_F^{-} v & \forall F\in\Fkb.
\end{aligned}  
\end{equation}
The jump and average operators are further extended to vector fields in $BV(\Om;\R^\dim)$ component-wise.
Although the sign of $\jump{v}_F$ depends on the choice of $\bn_F$, in subsequent expressions the jumps will appear either under absolute value signs or in products with $\bn_F$, so that the overall resulting expression is uniquely defined and independent of the choice of $\bn_F$.
When no confusion is possible, we will often drop the subscripts and simply write $\avg{\cdot}$ and~$\jump{\cdot}$. 

\paragraph{Tangential derivatives.} For $F\in \Fk$ and a sufficiently regular function $w\colon F\mapsto \R$, let $\nabla_T w$ denote the tangential (surface) gradient of $w$, and let $\Delta_T w$ denote its the tangential Laplacian of $w$. We do not indicate the dependence on $F$ in order to alleviate the notation, as it will be clear from the context. Since all faces considered here are flat, these tangential differential operators commute with the trace operator for sufficiently regular functions, see~\cite{SS13} for further details.

\subsection{Finite element spaces.}\label{sec:fem_spaces_def}
For a nonnegative integer $p$, let $\Pp$ be the space of polynomials of total degree at most $p$.
In the following, let $p\geq 2$ denote a fixed choice of polynomial degree to be used for the finite element approximations.
We then define the finite element spaces $\Vk^s$, $s\in \{0,1\}$, by
\begin{equation}\label{sec:fem_spaces}
\begin{aligned}
  \Vk^ 0 &\coloneqq \{v\in L^2(\Om):v|_K\in \Pp\;\forall K\in\Tk\},
   &\Vk^1&\coloneqq  V_{k}^0 \cap H^1_0(\Om).
\end{aligned}  
\end{equation}
Therefore, the spaces $\Vk^0$ and $\Vk^1$ correspond to DG and $C^0$-IP spaces on $\Tk$, respectively. Clearly $\Vk^1$ is a subspace of $\Vk^0$.
 As mentioned above in section~\ref{sec:BV}, for any $v_k\in \Vk^s$, the piecewise gradient of $v_k$ over $\Tk$ coincides with $\Npw v_k$ the density of the absolutely continuous part of its distributional derivative $Dv_k$. Similarly, the piecewise Hessian of $v_k$ over $\Tk$ coincides with $\Dpw v_k$ the density of the absolutely continuous part of $D(\Npw v_k)$.

\paragraph{Norms.}
We equip the spaces $\Vk^s$ for each $s\in\{0,1\}$ with the same norm $\normk{\cdot}\colon \Vk^s\tends \R$ and jump seminorm $\absJ{\cdot}\colon \Vk^s\tends \R$ defined by
\begin{subequations}\label{eq:norm_def}
\begin{align}
\normk{v}^2\coloneqq \int_{\Om}\left[ \abs{\Dpw v}^2+\abs{\Npw v}^2 + \abs{v}^2\right] + \absJ{v}^2,
\\ \absJ{v}^2\coloneqq \int_{\cF_k^I}h_k^{-1}\abs{\jump{\Npw v}}^2+\int_{\Fk}h_k^{-3}\abs{\jump{v}}^2,
\end{align}
\end{subequations}
for all $v\in\Vk^s$.
Although  $\Vk^0$ and $\Vk^1$ are equipped with the same norm, we remark that for any $v\in\Vk^1$, the terms in~\eqref{eq:norm_def} involving the jumps $\jump{v}$ over mesh faces vanishes identically owing to $H^1_0$-conformity, whilst the terms involving the jumps $\jump{\Npw v}$ of first derivatives over internal mesh faces can be simplified to merely jumps of normal derivatives. However, to give a unified treatment of both cases $s=0$ and $s=1$, we will not make explicit use of these specific simplifications for the case $s=1$.

\paragraph{Lifting operators.} Let $q$ denote a fixed choice of polynomial degree such that $q\geq p-2$, which implies that $q\geq 0$ since $p\geq 2$. Let $\Vq\coloneqq \{w \in L^2(\Om)\colon w|_K \in \mathbb{P}_q \;\forall K \in\Tk\}$ denote the space of piecewise polynomials of degree at most $q$ over $\Tk$.
For each face $F\in\Fk$, the lifting operator $r_k^F\colon L^2(F)\tends \Vq$ is defined by $\int_\Om r_k^F(w) \varphi_k = \int_F w \{\varphi_k\}$ for all $\varphi_k \in \Vq$.
Using inverse inequalities for polynomials, it is easy to see that $\norm{r_k^F(w)}_\Om\lesssim h_F^{-1/2} \norm{w}_F$ for any $w\in L^2(F)$ and any $F\in\Fk$.
Next, for each $F\in \Fk$, we define $\br_k^F\colon L^2(F;\R^\dim)\tends [\Vq]^{\dim\times\dim} $, where $[\Vq]^{\dim\times\dim}$ denotes the space of $\dim\times\dim$-matrix valued functions that are component-wise in $\Vq$, as follows.
For all $\bm{w}\in L^2(F;\R^\dim)$ and all $i,\,j=1,\dots,\dim$, if  $F\in \Fki$ is an interior face, then let $[\br_k^F(\bm{w})]_{ij} \coloneqq r_k^F(\bm{w}_i \bn_j)$ where $\bn=\bn_F$ is the chosen unit normal for $F$.
Otherwise, if $F\in\Fkb$ is a boundary face then let $[\br_k^F(\bm{w})]_{ij} \coloneqq r_k^F((\bm{w}_T)_i \bn_j)$, where $\bm{w}_T=\bm{w}-(\bm{w}\cdot\bn)\bn$ denotes the tangential component of $\bm{w}$ on $F$.
In other words, on boundary faces, only the tangential component of~$\bm{w}$ is considered in the lifting $\br_k^F(\bm{w})$.
 It follows that, for any $\bvphi_k \in [\Vq]^{\dim\times\dim}$,
\begin{equation}\label{eq:lifting_identity}
\int_\Om \br_k^F(\bm{w}):\bvphi_k =
\begin{cases}
 \int_F (\bm{w}\otimes\bn):\{\bvphi_k\} = \int_F \bm{w}\cdot\{\bvphi_k \bn\} &\text{if }F\in\Fki,\\
 \int_F (\bm{w}_T\otimes\bn):\{\bvphi_k\}= \int_F \bm{w}_T\cdot\{\bvphi_k \bn\} &\text{if }F\in\Fkb.
 \end{cases}
\end{equation}
We then define the global lifting operator $\br_k$ and the lifted Hessian operator $\Hk$, which both map $\Vk^s$, $s\in\{0,1\}$, into $L^2(\Om;\R^{\dim\times\dim})$, as well as the lifted Laplacian operator~$\Delta_k$, by
\begin{align}\label{eq:lifted_Hessian}
\br_k \coloneqq \sum_{F\in\Fk}\br_k^F, \qquad \Hk v_k \coloneqq \Dpw v_k - \rk, \qquad \Delta_k v_k \coloneqq \Tr \Hk v_k,
\end{align}
where it is recalled that $\Tr \bm{M}$ is the matrix trace for any $\bm{M}\in\R^{\dim\times\dim}$.
The operators defined above then satisfy the following bounds
\begin{equation}
\begin{aligned}
\norm{\br_k(\jump{\Npw v_k})}_\Om\lesssim \absJ{v_k}, &&& \norm{\bm{H}_k v_k}_\Om + \norm{\Delta_k v_k}_\Om \lesssim \normk{v_k} &&& \!\!\forall v_k\in\Vk^s.
\end{aligned}
\end{equation}
Using~\eqref{eq:lifting_identity}, it is easy to see that $\Tr\br_k^F(\bm{w})=0 $ for any $\bm{w}\in L^2(F)$ and when~$F\in\Fkb$ is a boundary face, since $\Tr (\bm{w}_T\otimes\bn) = \bm{w}_T\cdot\bn =0$ as $\bm{w}_T$ is tangential to~$F$. Thus only interior face liftings contribute to $\Delta_k v_k$.

 \section{Variational formulation of the problem and adaptive finite element approximation}\label{sec:var_fem}

\subsection{Variational formulation of the problem}\label{sec:variational}

In order to focus on the most important aspects of analysis, we shall restrict our attention to Isaacs and HJB equations without lower order terms, although we note that the approach we consider here easily accommodates problems with lower order terms, see~\cite{KaweckiSmears20,SS14,SS16}.
More precisely, let the real valued functions $a_{ij}=a_{ji}$ and $f$ belong to $C(\overline{\Om}\times\sA\times\sB )$ for each $i,j=1,\ldots,\dim$.
For each $(\al,\be)\in\sA\times\sB$, we then define the matrix-valued function $a^\ab \colon \Om\tends \R^{\dim\times\dim}$ by $a^\ab_{ij}(x)=a_{ij}(x,\al,\be)$ for all $x\in \Om$ and $i,\,j=1,\dots,\dim$.
The functions $f^\ab$ are defined similarly for all $\al\in\sA$ and $\be\in\sB$.
Then, for each $\al\in\sA$ and $\be\in\sB$, the operators $L^\ab:H^2(\Om)\to L^2(\Om)$ are defined by 
\begin{equation}\label{eq:Lab_operators}\begin{aligned}
L^\ab v &= a^\ab {\colon}\nabla^2 v && \forall v\in H^2(\Om).
\end{aligned}\end{equation}
The nonlinear operator $\Fop\colon H^2(\Om)\tends L^2(\Om)$ is then defined as in~\eqref{eq:isaacs}. Note that the compactness of $\overline{\Om}\times\sA\times\sB$ and the continuity of the coefficients imply that $F$ is well-defined as a mapping from $H^2(\Om)$ to $L^2(\Om)$.
We consider the problem~\eqref{eq:isaacs} in its strong form, i.e.\ to find a solution $u\in H^2(\Om)\cap H^1_0(\Om)$ such that $F[u]=0$ pointwise a.e.\ in $\Om$.
We assume that the problem is uniformly elliptic, i.e.\ there exists positive constants $\underline{\nu}$ and $\overline{\nu}$ such that $\underline{\nu}\abs{\bm{v}}^2\leq \bm{v}^{\top}a^\ab(x) \bm{v} \leq \overline{\nu}\abs{\bm{v}}^2$ for all $\bm{v}\in \R^\dim$, for all $x\in \overline{\Om}$ and all $(\alpha,\beta)\in\sA\times\sB$, where $\abs{\bm{v}}$ denotes the Euclidean norm of $\bm{v}$.
Furthermore, we assume the Cordes condition: there exists a $\nu\in(0,1]$ such that
\begin{equation}\label{eq:Cordes}
\begin{aligned}
\frac{\abs{a^{\ab}(x)}^2}{\Tr(a^{\ab}(x))^2}\le\frac{1}{d-1+\nu}&&&\forall x \in \overline{\Om},\quad\forall(\alpha,\beta)\in\sA\times\sB,
\end{aligned}
\end{equation}
where $\abs{a^{\ab}}$ denotes the Frobenius norm of the matrix $a^\ab$.
It is well-known that if $\dim=2$, then uniform ellipticity implies the Cordes condition~\eqref{eq:Cordes}, see e.g.~\cite[Example~2]{SS14}.
In~\cite{SS14,SS16} and later in~\cite{KaweckiSmears20} it was shown that fully nonlinear HJB and Isaacs equations can be reformulated in terms of a renormalized nonlinear operator, as follows. 
For each $(\alpha,\beta)\in\sA\times\sB$, let $\gamma^{\ab}\in C(\overline{\Om})$ be defined by $\gamma^{\ab}\coloneqq \frac{\Tr a^{\ab} }{\abs{a^{\ab}}^2}$.
Let the renormalised operator~$\Fg \colon H^2(\Om)\tends L^2(\Om)$ be defined by
\begin{equation}\label{eq:Fg_def}
\begin{aligned}
\Fg[v]\coloneqq \inf_{\alpha\in\sA}\sup_{\beta\in\sB}\left[\gamma^{\ab}\left(L^\ab v - f^\ab\right)\right] &&&\forall v \in H^2(\Om).
\end{aligned}
\end{equation}
It is shown in~\cite{KaweckiSmears20}, see also~\cite{SS14}, that the renormalized operator $\Fg$ is Lipschitz continuous and satisfies the following bounds
\begin{subequations}\label{eq:cordes_ineq}
\begin{align}
\abs{\Fg[w] - \Fg[v] - \Delta (w-v) }&\leq \sqrt{1-\nu}\sqrt{|\nabla^2 z|^2+2\lambda|\nabla z|^2+\lambda^2|z|^2},\label{eq:cordes_ineq1}\\
|\Fg[w]-\Fg[v]|&\leq \big(1+\sqrt{d+1}\big)\sqrt{|\nabla^2z|^2+2\lambda|\nabla z|^2+\lambda^2|z|^2},\label{eq:cordes_ineq2}
\end{align}
\end{subequations}
for all functions $w$ and $v\in H^2(\om)$ for any open subset $\om\subset \Om$, where $z:=w-v$, and with the above bounds holding pointwise a.e.\ in $\om$.
The following Lemma from~\cite{KaweckiSmears20}, which extends earlier results from~\cite{SS14}, states that the equations $F[u]=0$ and $\Fg[u] =0$ have equivalent respective sets of sub- and supersolutions.
\begin{lemma}[\cite{KaweckiSmears20,SS14}]\label{lem:subsupersolutions}
A function $v\in H^2(\Om)$ satisfies $F[v]\leq 0$ pointwise a.e.\ in $\Omega$ if and only if $\Fg[v]\leq 0$ pointwise a.e.\ in~$\Omega$. Furthermore, a function $v\in H^2(\Om)$ satisfies $F[v]\geq 0$ pointwise a.e.\ in $\Omega$ if and only if $\Fg[v]\geq 0$ pointwise a.e.\ in~$\Omega$.
\end{lemma}

A particular consequence of Lemma~\ref{lem:subsupersolutions} is that a solution of $F[u]=0$ is equivalently a solution of $\Fg[u]=0$. Moreover, it is was shown in~\cite{SS14} for fully nonlinear HJB equations, and later for Isaacs equations in~\cite{KaweckiSmears20}, that under the above assumptions, there exists a unique strong solution of~\eqref{eq:isaacs}.

\begin{theorem}[\cite{KaweckiSmears20,SS14}]\label{thm:well_posedness}
There exists a unique $u\in H^2(\Om)\cap H^1_0(\Om)$ that solves $F[u]=0$ pointwise a.e.\ in $\Om$, and, equivalently, that solves $\Fg[u]=0$ pointwise a.e.\ in $\Om$.
\end{theorem}

In particular, the proof, due to~\cite{SS14}, involves reformulating the equation $F[u]=0$ in terms of a strongly monotone nonlinear operator equation of the form $A(u;v)=0$ for all $v\in H^2(\Om)\cap H^1_0(\Om)$, where
\begin{equation}
\begin{aligned}
A(u;v)\coloneqq \int_\Om \Fg[u]\Delta v  &&&\forall v\in H^2(\Om)\cap H^1_0(\Om).
\end{aligned}
\end{equation}
Note that the equivalence of these formulations is a consequence of the bijectivity of the Laplace operator from $H^2(\Om)\cap H^1_0(\Om)$ to $L^2(\Om)$ on the convex domain $\Omega$.
It is then shown in~\cite{SS14,KaweckiSmears20} that $A(\cdot;\cdot)$ is Lipschitz continuous, and also strongly monotone on the space $H^2(\Om)\cap H^1_0(\Om)$, i.e.\
\begin{equation}\label{eq:continuous_strong_monotonicity}
\begin{aligned}
\frac{1}{c_{\star}}\norm{w-v}_{H^2(\Om)}^2 \leq A(w;w-v)-A(v;w-v) &&&\forall w,\, v \in H^2(\Om)\cap H^1_0(\Om),
\end{aligned}
\end{equation}
where $c_\star$ in particular depends only on $\dim$, $\diam\Om$ and $\nu$ from~\eqref{eq:Cordes}.
Therefore, the existence and uniqueness of a strong solution $u$ follows from the Browder--Minty theorem.

\subsection{Numerical discretizations and error estimators}\label{sec:num_schemes}

For each $k\in\N$, let the bilinear form $\Skf\colon \Vk^0\times\Vk^0\tends \R$ be defined by
\begin{equation}\label{eq:B_def}
\begin{split}
\Skf(w_k,v_k)\coloneqq &\int_\Om \left[\Dpw w_k:\Dpw v_k - \Delta w_k \Delta v_k\right]
\\ & + \int_{\Fki} \left[\avg{\Delta_T w_k} \jump{\Npw v_k\cdot \bn} + \avg{\Delta_T v_k} \jump{\Npw w_k\cdot \bn} \right] \\
 &- \int_{\Fk} \left[ \Npw_T\avg{\Npw w_k\cdot \bn} \cdot \jump{\Npw_T v_k} + \Npw_T\avg{\Npw v_k\cdot \bn} \cdot \jump{\Npw_T w_k} \right],
\end{split}
\end{equation}
for all  $w_k,\,v_k\in\Vk^0$.
The bilinear form $\Skf(\cdot,\cdot)$ represents a stabilization term in the numerical schemes defined below.
For two positive constant parameters $\sigma$ and $\rho$ to be chosen sufficiently large, let the jump penalisation bilinear form $\Jpen\colon \Vk^0\times\Vk^0\tends \R$ be defined by
\begin{equation}\label{eq:jump_pen_def}
\begin{aligned}
\Jpen(w_k,v_k)\coloneqq \int_{\Fki} \sigma h_k^{-1} \jump{\Npw w_k}\cdot\jump{\Npw v_k} + \int_{\Fk} \rho h_k^{-3}\jump{w_k}\jump{v_k},
\end{aligned}
\end{equation}
for all $w_k,\,v_k\in\Vk^0$.
For a parameter $\theta\in[0,1]$, let the nonlinear form $A_k\colon \Vk^0\times\Vk^0\tends \R$ be defined by
\begin{equation}\label{eq:nonlinear_form}
\begin{aligned}
A_k(w_k;v_k)\coloneqq \int_\Om \Fg[w_k]\Delta_k v_k + \theta \Skf(w_k,v_k) + \Jpen(w_k,v_k),
\end{aligned}
\end{equation}
for all functions $w_k,\,v_k \in \Vk^0$, where we recall that the lifted Laplacian $\Delta_k v_k$ appearing in the first integral on the right-hand side of \eqref{eq:nonlinear_form} is defined in~\eqref{eq:lifted_Hessian}.
The nonlinear form $A_k$ is nonlinear in its first argument, but linear in its second argument.
For a fixed choice of $s\in\{0,1\}$, the numerical scheme is then to find $u_k\in\Vk^s$ such that
\begin{equation}\label{eq:num_scheme}
\begin{aligned}
A_k(u_k;v_k) = 0 &&&\forall v_k\in\Vk^s.
\end{aligned}
\end{equation}
Since $s\in\{0,1\}$ is fixed, we omit the dependence of $u_k$ on $s$ in the notation, as there is no risk of confusion.
The choice $\theta=1/2$ is based on the method of~\cite{SS13,SS14,SS16}, with the modification that the nonlinear operator is tested against the lifted Laplacian rather than the piecewise Laplacian of test functions. The choice $\theta=0$ and $s=1$ is similar to the method of \cite{NeilanWu19}, again modulo the introduction of the lifted Laplacians for the first integral term. 
The lifted Laplacians will play a role later on in the proof of asymptotic consistency of the nonlinear forms $A_k(\cdot;\cdot)$.

\begin{remark}[Simplifications for $C^0$-IP methods]
Note that when considering the restriction of $\Jpen(\cdot,\cdot)$ to $\Vk^1\times \Vk^1$, the last term on the right-hand side of~\eqref{eq:jump_pen_def} vanishes identically, and we can take $\rho=0$. Furthermore, since the jumps of gradients of functions in $\Vk^1$ have vanishing tangential components over the faces of the mesh, the first term in the right-hand side of \eqref{eq:jump_pen_def} can be further simplified to just the jumps in the normal components of the gradient. These simplifications can be useful in computational practice but we retain the general form above in order to present a unified analysis for both DG and $C^0$-IP methods.
\end{remark}

 We recall now some basic properties of the numerical scheme that have been shown in previous works, see in particular \cite{KaweckiSmears20} for a complete treatment.
Building on the analysis in~\cite{SS13,SS14}, it was shown in~\cite{KaweckiSmears20} that the parameters $\sigma$ and $\rho$ can be chosen sufficiently large such that $A_k$ is strongly monotone with respect to~$\normk{\cdot}$, i.e.\ such that there is a fixed constant $\cmon>0$ independent of $k$, such that
\begin{equation}\label{eq:strong_monotonicity}
\begin{aligned}
\frac{1}{\cmon}\normk{w_k - v_k}^2 \leq  A_k(w_k;w_k-v_k)-A_k(v_k,w_k-v_k) &&&\forall w_k,\,v_k\in\Vk^s, \;\forall k\in\N.
\end{aligned}
\end{equation}
It is also straightforward to show from standard techniques along with~\eqref{eq:cordes_ineq2} that the nonlinear form $A_k$ is Lipschitz continuous, i.e.\ there exists a positive constant $C_{\mathrm{Lip}}$, independent of $k$, such that
\begin{equation}\label{eq:Ak_lipschitz}
\begin{aligned}
\abs{A_k(w_k;v_k)-A_k(z_k;v_k)}\leq C_{\mathrm{Lip}} \normk{w_k-z_k}\normk{v_k}&&&\forall w_k,\,z_k,\,v_k\in\Vk^0,\;\forall k\in\N.
\end{aligned}
\end{equation}
It then follows from the Browder--Minty theorem that there exists a unique solution $u_k\in\Vk^s$ of~\eqref{eq:num_scheme} for each $k\in\N$. We refer the reader to~\cite{KaweckiSmears20} for a detailed discussion of the dependencies of the constants.
The strong monotonicity bound~\eqref{eq:strong_monotonicity}, the boundedness of the data, and the Lipschitz continuity~\eqref{eq:Ak_lipschitz} also imply the boundedness of the sequence of numerical solutions, i.e.\ 
\begin{equation}\label{eq:num_sol_bounded}
\sup_{k\in\N}\normk{u_k} < \infty.
\end{equation}
Furthermore, it follows from~\cite[Theorem~4.3]{KaweckiSmears20} that the numerical solution $u_k$ is a quasi-optimal approximation of $u$, i.e.\, up to a constant, the error attained by $u_k$ is equivalent to the best approximation error of $u$ from the space $\Vk^s$.

\paragraph{Analysis of stabilization terms.}
We collect here two results that will be used later in the analysis.
First, we note that the bilinear form $\Skf(\cdot,\cdot)$ defined in~\eqref{eq:B_def} constitutes a stabilization term, and is consistent with the original problem, see~\cite[Lemma~5]{SS13}. 
We will also use the following theorem from~\cite[Theorem~5.3]{KaweckiSmears20}, which improves on~\cite{SS13}, provides a quantitative bound for possibly nonsmooth functions in $\Vk^s$.

\begin{theorem}[\cite{KaweckiSmears20}]\label{thm:b_k_jump_bound}
The bilinear form $\Skf(\cdot,\cdot)$ satisfies
\begin{equation}
\begin{aligned}
\abs{\Skf(w_k,v_k)} \lesssim \absJ{w_k}\absJ{v_k} &&& \forall w_k,\,v_k\in \Vk^s, \;\forall s\in\{0,1\}.
\end{aligned}
\end{equation}
\end{theorem}
When it comes to the analysis of asymptotic consistency of the numerical schemes, it is advantageous to write the face terms in bilinear form $\Skf(\cdot,\cdot)$ via the lifting operators defined in~Section~\ref{sec:fem_spaces_def}.
\begin{lemma}\label{lem:lifted_Bk_form}
For all $v_k,w_k\in V_k^s$, $s\in\{0,1\}$, there holds
\begin{equation}\label{eq:lifted_Bk_form}
\begin{aligned}
\Skf(w_k,v_k) &= \int_\Om \left[\Hk w_k {:} \Hk v_k - \Delta_k w_k \Delta_k v_k \right]
\\ &  + \int_\Om \left[ \Tr \br_k(\jump{\Npw w_k}) \Tr \br_k(\jump{\Npw v_k}) - \br_k(\jump{\Npw w_k}){:}\br_k(\jump{\Npw v_k}) \right].
\end{aligned}
\end{equation}
\end{lemma}
\begin{proof} 
Using the identity~\eqref{eq:lifting_identity}, simple algebraic manipulations show that, for any $w_k$ and $v_k \in\Vk^s$,
\begin{multline}\label{eq:lifted_Bk_form_1}
\int_\Om \Dpw v_k : \br_k(\jump{\Dp w_k}) -\Delta v_k \Tr\br_k(\jump{\Npw w_k})\\  = \int_{\Fki} \left[\avg{\Dpw v_k}:(\jump{\Npw w_k}\otimes\bn) - \avg{\Delta v_k}\jump{\Npw w_k \cdot \bn}\right] + \int_{\Fkb} \{\Dpw v_k\}:(\jump{ \Npw_T w_k}\otimes\bn)
\\ = \int_{\Fk} \Npw_T\avg{\Npw v_k\cdot \bn} \cdot \jump{\Npw_T w_k} -\int_{\Fki} \avg{\Delta_T v_k} \jump{\Npw w_k\cdot \bn},
\end{multline}
where the second identity is obtained by cancelling terms exactly as in the proof of~\cite[Lemma~5]{SS13}. Note that it is possible to interchange $w_k$ and $v_k$ in~\eqref{eq:lifted_Bk_form_1}. The identity~\eqref{eq:lifted_Bk_form} is then obtained by expanding all terms in its right-hand side and simplifying with the help of~\eqref{eq:lifted_Bk_form_1}.
\Fqed\end{proof}

Theorem~\ref{thm:b_k_jump_bound} will be used later in the proof of convergence of the adaptive methods.

\paragraph{Reliable and efficient a posteriori error estimator.}
For each $k\in\N$ and any $v_k\in\Vk^s$, we define the element-wise error estimators $\ek(v_k,K)$ for each $K\in\Tk$, and total error estimator $\ek(v_k)$, by 
\begin{subequations}\label{eq:estimator_def}
\begin{align}
\left[\ek(v_k,K)\right]^2 &\coloneqq \int_K \abs{\Fg[v_k]}^2 + \sum_{F\in\Fki;F\subset \p K} \int_F \delta_F h_k^{-1}\abs{\jump{\Npw v_k}}^2 +\sum_{F\in\Fk; F\subset \p K}\int_{F} \delta_F h_k^{-3} \abs{\jump{v_k}}^2 ,
\\ [\ek(v_k)]^2&\coloneqq\sum_{K\in\Tk}[\ek(v_k,K)]^2,
\end{align}
\end{subequations}
where the weight $\delta_F=1/2$ if $F\in\Fk^I$ and otherwise $\delta_F=1$ for $F\in\Fk^B$.
The reliability and local efficiency of the above error estimators is shown in~\cite{KaweckiSmears20}, see also related results in~\cite{Kawecki19c,Bleschmidt19}. 
In particular,~\cite[Theorem~4.2]{KaweckiSmears20} shows that there exists a constant $C_{\mathrm{rel}}>0$, independent of $k\in\N$, such that
\begin{equation}
\normk{u-v_k} \leq C_{\mathrm{rel}} \ek(v_k)\quad \forall v_k\in\Vk^s,\;\forall s\in\{0,1\},\;\forall k\in\N.
\end{equation}
Note that the reliability bound indeed holds for all functions from the approximation space and not only the numerical solution $u_k\in\Vk^s$; this results primarily from the fact that $u$ is a strong solution of the problem.  
Furthermore, the estimators are locally efficent, in particular, there is a constant $C_{\mathrm{eff}}>0$ independent of $k$, such that 
\begin{multline}
\frac{1}{C^2_{\mathrm{eff}}}[\ek(v_k;K)]^2 \\ \leq  \int_K \abs{\Dpw(u-v_k)}^2 + \sum_{\substack{F\in\Fki\\ F\subset \p K}} \int_F \delta_F h_k^{-1}\abs{\jump{\Npw v_k}}^2  + \sum_{\substack{F\in\Fk\\ F\subset \p K}}\int_{F}\delta_F h_k^{-3} \abs{\jump{v_k}}^2,
\end{multline}
for all $v_k\in\Vk^s$.
This implies the global efficiency bound
\begin{equation}\label{eq:global_efficiency}
\ek(v_k) \leq C_{\mathrm{eff}}\normk{u-v_k} \quad \forall v_k\in\Vk^s,\;\forall s\in\{0,1\}.
\end{equation}
For further analysis of the dependencies of the constants $C_{\mathrm{rel}}$ and $C_{\mathrm{eff}}$ we refer the reader to~\cite{KaweckiSmears20}. Note that the error estimators do not feature any positive power weight of the mesh-size in the residual terms, which is an issue for the reduction property typically used in the analysis of convergence rates of adaptive algorithms.

\subsection{Adaptive algorithm and main result}\label{sec:main_result}
We now state precisely the adaptive algorithm. Consider a fixed choice of $s\in\{0,1\}$, with $s=0$ corresponding to the DG method, and $s=1$ corresponding to the $C^0$-IP method, and consider fixed integers $p$ and $q$ such that $p\geq 2$ and $q\geq p-2$.
Given an initial mesh $\cT_1$, the algorithm produces the sequence of meshes $\{\Tk\}_{k\in\N} $ and numerical solutions $u_k\in\Vk^s$ by looping over the following steps for each $k\in\N$.
\begin{enumerate}
\item \emph{Solve.} Solve the discrete problem~\eqref{eq:num_scheme} to obtain the discrete solution $u_k\in\Vk^s$.
\item \emph{Estimate.} Compute the estimators $\{\ek(u_k,K)\}_{K\in\Tk}$ defined by~\eqref{eq:estimator_def}.
\item \emph{Mark.} Choose a subset of elements $\mathcal{M}_k\subset \Tk$ marked for refinement, such that 
\begin{equation}\label{eq:max_marking}
\max_{K\in\Tk}\ek(u_k,K) = \max_{K\in \Mk}\ek(u_k,K).
\end{equation}
\item \emph{Refine.} Construct a conforming simplicial refinement $\cT_{k+1}$ from $\Tk$ such that every element of $\Mk$ is refined, i.e.\  $K \in \Tk\setminus \cT_{k+1}$ for all $K\in\Mk$.
\end{enumerate}

The marking condition~\eqref{eq:max_marking} is rather general and can be combined with additional conditions on the marked set such as those used in maximum and bulk-chasing strategies.
Since~\eqref{eq:max_marking} is sufficient for the proof of convergence of the adaptive method, we do not specify further conditions on the marking strategy and instead allow for any marking strategy that satisfies~\eqref{eq:max_marking}.
Recall also that the refinement routine is assumed to satisfy the conditions of \emph{quasi-regular subdivisions} of~\cite{MorinSiebertVeeser08}.

\paragraph{Main result.}
The main result of this work states that the sequence of numerical approximations generated by the adaptive algorithm converges to the solution of \eqref{eq:isaacs} and that the estimators vanish in the limit.

\begin{theorem}\label{thm:main}
The sequence of numerical solutions $\{u_k\}_{k\in\N}$ converges to the solution $u$ of~\eqref{eq:isaacs} with
\begin{equation}\label{eq:main}
\begin{aligned}
\lim_{k\tends\infty}\normk{u-u_k} =0, &&& \lim_{k\tends\infty} \ek(u_k)= 0.
\end{aligned}
\end{equation}
\end{theorem}

Theorem~\ref{thm:main} establishes plain convergence of the numerical solutions to the exact solution, without requiring any additional regularity assumptions on the problem.

\section{Analysis of the limit spaces}\label{sec:limspace}

In this section we introduce appropriate limit spaces for the sequence of the finite element spaces $\{\Vk^s\}_{k\in\N}$. 
We give here an intrinsic approach to the construction of the limit spaces, which is designed to overcome some key difficulties in the analysis of weak limits of bounded sequences of finite element functions. In particular, we construct the limit spaces in terms of some original function spaces that are of independent interest for adaptive nonconforming methods for more general problems.

\subsection{Sets of never-refined elements and faces}\label{sec:refinement_sets}
We start by considering some elementary properties of the sets of elements and faces that are never-refined, following~e.g.~\cite{DominincusGaspozKreuzer19,KreuzerGeorgoulis18,MorinSiebertVeeser08}.  
Let $\Tp$ be the set of elements of the sequence of meshes $\{\Tk\}_{k\in\N}$ that are never refined once created, i.e.\
\[
\Tp\coloneqq \bigcup_{m\geq 0}\bigcap_{k\geq m}\Tk,  
\]
and let $ \Omp \coloneqq \bigcup_{K\in \Tp}K$ be its associated subdomain.
Let the complement $\Omm\coloneqq\overline{\Omega}\setminus\Omp$, which represents the region of $\overline{\Om}$ where the mesh-sizes become vanishingly small in the limit, as shown by Lemma~\ref{lem:hjvanishes} below.
For $k\in \N$, let $\Tkp$ denote the set of never-refined elements in $\Tk$, and let $\Tk^{-}$ denote its complement in $\Tk$, given by
\[
\begin{aligned}
\Tkp \coloneqq \Tk\cap \Tp, &&& \Tk^{-} \coloneqq \Tk\setminus \Tkp.
\end{aligned}
\]
For integers $k\geq 1$ and $j\geq 0$, we also define the set $\Tkj \coloneqq \{K\in\Tk:N_k^j(K)\subset\Tk^+\}$ and its complement $\Tk^{j-} \coloneqq \Tk\setminus\Tk^{j+}$, where we recall that $N_k^j(K)$ denotes the set of all elements in $\Tk$ that are at most $j$-th neighbours of $K$. 
Recalling that $N_k^0(K)=K$, we have $\Tk^{0+}=\Tkp$ and $\Tk^{0-}=\Tk^{-}$. 
For the corresponding domains, we define $\Om_k^{j+} \coloneqq \bigcup_{K\in\Tk^{j+}} K$ and $\Om_k^{j-} \coloneqq \bigcup_{K\in\Tk^{j-}} K$. It follows that the intersection $\Om_k^{j+}\cap\Om_k^{j-}$ is a set of Lebesgue measure zero. 
Similar to $\Tkp$ and $\Tk^{-}$, we use the shorthand notation $\Omega_k^+\coloneqq\Omega_k^{0+}$, and $\Omega_k^-\coloneqq\Omega_k^{0-}$.
Furthermore, it is also easy to see that the sets $\Tkp$ and $\Tkj $ are ascending with respect to $k$, i.e.\ $\Tk^{j+} \subset \cT_{k+1}^{j+}$ for all $k\in \N$ and all $j\in\N_0$, whereas the $\Tkj$ are descending with respect to $j$, i.e.\ $\cT_k^{j+}\subset \cT_k^{(j-1)+}$ for all $j\in \N$.
The following two Lemmas are from \cite{DominincusGaspozKreuzer19}, see also \cite{MorinSiebertVeeser08}. The first Lemma  states that neighbours of never-refined elements are also eventually never-refined, and the second Lemma shows that the mesh-size functions converge uniformly to zero on the refinement sets $\Om_k^{j-}$ as $k\tends \infty$, for any fixed $j\in\N_0$. 

\begin{lemma}[\cite{DominincusGaspozKreuzer19,MorinSiebertVeeser08}]\label{lem:eventuallyinT+}
For every $K\in \Tp$ there exists an integer $m=m(K) \in \N$ such that $K\in\Tk^+$ for all $k\geq m$ and $N_k(K)=N_m(K) \subset \Tp$ for all $k\geq m$.
\end{lemma}

\begin{lemma}[\cite{DominincusGaspozKreuzer19,MorinSiebertVeeser08}]\label{lem:hjvanishes}
For any $j\in \N_0$, we have $\norm{h_k\chi_{\Om_k^{j-}}}_{L^\infty(\Om)}\tends 0$ as $k\tends \infty$, where $\chi_{\Om_k^{j-}}$ denotes the characteristic function of $\Om_k^{j-}$. Moreover, $|\Om_k^{j-}\setminus\Om^{-}|=|\Om^+\setminus\Om_k^{j+}|\to0$ as $k\to\infty$.
\end{lemma}

For each $K\in \Tp$, let $N_+(K)$ denote the neighbourhood of $K$ in $\Tp$, i.e.\ $N_+(K)=\{K^\prime \in \Tp,\; K^\prime\cap K \neq \emptyset\}$. Lemma~\ref{lem:eventuallyinT+} implies that for each $K\in \Tp$, there exists $m=m(K)\in\N$ such that $N_+(K)= N_k(K)$ for all $k\geq m$. 

\paragraph{Never-refined faces.}
Let $\Fp$ denote the set of all faces of elements from $\Tp$, i.e. $F\in \Fp$ if and only if there exists $K\in \Tp$ such that $F$ is a face of $K$. 
The set $\Fp$ is at most a countably infinite subcollection of $\bigcup_{k\in\N} \Fk$. We also consider $\Fpi$ and $\Fpb$ the set of interior and boundary faces of $\Fp$, respectively.
For each $k\in\N$, let $\Fkp \coloneqq \Fk\cap \Fp$ denote the set of never-refined faces in $\Fk$. 
It holds trivially that $\Fp=\bigcup_{k\in\N} \Fkp$ and that the sets $\Fkp$ are ascending, with $\Fkp\subset \cF_{k+1}^+$ for all $k\in \N$.
We also consider the set $\Fks$, $k\in\N$, of faces of only elements in $\Tkp$, defined by
\begin{equation}\label{eq:Fks_def}
\Fks \coloneqq \{F\in\Fp \colon \exists \{K,K^\prime\}\subset \Tk^+, \text{ s.t. } F = K\cap K'\text{ or } F=K\cap\partial\Om\}.
\end{equation}
Additionally, let $\Fksi\coloneqq \Fpi \cap \Fks$ denote the subset of interior faces of $\Fks$.
The definition implies that $\Fks \subset \Fkp$ and $\Fksi \subset \Fkip$, however in general $\Fks\neq \Fkp$ since it is possible to refine pairs of neighbouring elements without refining their common face.
Note also that $\Fks \subset \cF_{k+1}^{\dagger}$ for all $k\in\N$ and thus $\{\Fks\}_{k\in\N}$ also forms ascending sequence of sets with respect to $k$.
Moreover, since neighbours of elements in $\Tp$ are eventually also in $\Tp$, as shown by Lemma~\ref{lem:eventuallyinT+}, and since the meshes $\Tk$ are conforming, we also have $\Fp = \bigcup_{k\in\N}\Fks$.
We also consider the skeletons formed by sets of never refined faces. In particular, let $\Sp$ denote the skeleton of $\Fp$, defined by $\Sp = \bigcup_{F\in\Fp}F$.
Additionally, let $\Skp \coloneqq \Sk\cap \Sp$.
It follows that $\Sp$ is a measurable set with respect to the $\dim-1$ dimensional Hausdorff measure with $\mathcal{H}^{\dim-1}(\Sp)\in[0,\infty]$, i.e.\ $\mathcal{H}^{\dim-1}(\Sp)$ is not necessarily finite.

The next Lemma shows that the set of never-refined faces of any particular mesh is fully determined after at most finitely many refinements.

\begin{lemma}\label{lem:face_refinements}
For each $k\in \N$ there exists $M=M(k)$ such that
\begin{equation}\label{eq:Fkp_equivalence}
\Fkp = \cF_k\cap \cF_m \quad  \forall m\geq M.
\end{equation}
\end{lemma}

\begin{proof}
The inclusion $\Fkp \subset \cF_k\cap \cF_m$ for all $m$ large enough is clear and follows easily from the definitions. 
The converse inclusion $\Fk\cap \cF_m \subset \Fkp$ for all $m$ large enough is shown by contradiction. Since $\Fk$ is finite, if the claim were false, there would exist $F\in (\cF_{m_j}\cap \Fk)\setminus \Fkp$ for a sequence of indices $m_j\tends \infty $ as $j\tends \infty$. Then, by definition, there exists a sequence of elements $K_j\in \cT_{m_j}$ such that $F$ is a face of $K_j$ for each $j\in \N$. The shape-regularity of the meshes implies that $h_{m_j}|_{K_j^\circ}=\abs{K_j}^{1/\dim}\gtrsim \diam(F)$ for all $j\in\N$ and hence $\epsilon\coloneqq \inf_{j\in\N} h_{m_j}|_{K_j^\circ}$ is strictly positive. Lemma~\ref{lem:hjvanishes} then implies that there exists $J$ such that $h_{m_j}|_{K^\circ} <\epsilon$ for all $K\in \cT_{m_j}^{-}$ and all $j\geq J$, which implies that $K_j\in\cT_{m_j}^+$ for all $j\geq J$ and thus $F\in \Fp$. This implies that $F\in \Fkp$, thereby giving a contradiction and completing the proof.
\Fqed\end{proof}

\paragraph{Mesh-size function on never-refined elements and faces.}
Recalling the notation of Section~\ref{sec:notation}, for each $K\in \Tp$, let $h_+|_{K^\circ} \coloneqq h_K $ and for each face $F\in\Fp$, let $h_+|_{F} \coloneqq h_F $.
Thus $h_+$ is defined on $\Omp$ up to sets of $\mathcal{H}^{\dim-1}$-measure zero. The function $h_+$ is $\mathcal{H}^{\dim-1}$-a.e.\ positive on $\Omp$, although it is generally not uniformly bounded away from zero. 
Due to Lemma~\ref{lem:eventuallyinT+}, it follows that for each  $K\in\Tp$, there exists an $m=m(K)\in\N$ such that $h_+|_{K}=h_k|_{K}$ for all $k\ge m$, see also \cite[Lemma~4.3]{MorinSiebertVeeser08} which implies that $\norm{h_k-h_+}_{L^\infty(\Omp)}\tends 0$ as $k\tends \infty$.

\subsection{First-order spaces, Poincar\'e and trace inequalities.}

The construction of the limit spaces for the sequence of finite element spaces is broken down into several steps.
In a first step, we introduce particular subspaces of functions of (special) bounded variation with possible jumps only on the set of never-refined faces of the meshes, and that have sufficiently integrable gradients and jumps. 
We then show that these spaces are Hilbert spaces, and that they enjoy a Poincar\'e inequality and $L^2$-trace inequalities on all elements from all of the meshes.
Recall the notation of Section~\ref{sec:notation}, in particular for a function $v\in BV(\Om)$, the gradient $\nabla v$ denotes the density of the absolutely continuous part of the distributional derivative $Dv$.

\begin{definition}\label{def:H1limitspaces}
Let $H^1_D(\Om;\Tp)$ denote the space of functions $v \in L^2(\Om)$ such that the zero-extension of $v$ to $\R^\dim$, also denoted by $v$, belongs to $BV(\R^\dim)$, such that 
\begin{equation}\label{eq:distderiv_H100Tp}
\langle D v,\bphi \rangle_{\R^\dim} \coloneqq - \int_{\R^\dim} v \Div \bphi = \int_{\Om} \Npw v \cdot \bphi - \int_{\Fp} \jump{v} (\bphi{\cdot}\bn) \quad \forall \bphi\in \CidR,
\end{equation}
and such that
\begin{equation}\label{eq:H100Tp_norm}
\norm{v}_{H^1_D(\Om;\Tp)}^2\coloneqq\int_{\Om} \left[\abs{\Npw v}^2 +\abs{v}^2\right] + \int_{\Fp }h_+^{-1}\abs{\jump{v}}^2 <\infty.
\end{equation}
Let $H^1(\Om;\Tp)$ denote the space of functions $v\in L^2(\Om)\cap BV(\Om)$ such that 
\begin{equation}\label{eq:distderiv_H1Tp}
\langle D v,\bphi \rangle_{\Om} \coloneqq - \int_\Om v \Div \bphi = \int_{\Om} \Npw v \cdot \bphi - \int_{\Fpi} \jump{v} (\bphi{\cdot}\bn) \quad \forall \bphi\in \Cid,
\end{equation}
and such that
\begin{equation}\label{eq:H1Tp_norm}
\norm{v}_{H^1(\Om;\Tp)}^2 \coloneqq\int_{\Om} \left[\abs{\Npw v}^2 +\abs{v}^2\right] + \int_{\Fpi}h_+^{-1}\abs{\jump{v}}^2 <\infty.
\end{equation}
\end{definition}

\begin{remark}[Piecewise $H^1$-regularity over $\Tp$]\label{rem:notation}
For any $K\in \Tp$, by simply considering test functions $\bphi \in C^\infty_0(K;\R^\dim)$ in \eqref{eq:distderiv_H1Tp}, it is seen that any function $v\in H^1(\Om;\Tp)$ is $H^1$-regular over $K$, i.e.\ $v|_K \in H^1(K)$, and that the weak derivative $\nabla (v|_K)$ coincides with $(\nabla v)|_K$ the restriction of $\nabla v$ to $K$.
\end{remark}

\begin{remark}
The space $H^1_D(\Om;\Tp)$ consists of functions with a weakly imposed Dirichlet boundary condition on $\p \Omega$ through a Nitsche-type penalty term.
The definition of the space $H^1_D(\Om;\Tp)$ is motivated by the characterization of $H^1_0(\Om)$ as the space of measurable functions on $\Om$ whose zero-extension to $\R^\dim$ belongs to $H^1(\R^\dim)$, see \cite[Theorem~5.29]{AdamsFournier03}.
In particular, it follows that $H^1_0(\Om)$ is a closed subspace of $H^1_D(\Om;\Tp)$.
In general, functions in $H^1_D(\Om;\Tp)$ do not have vanishing interior traces on $\p\Omega$, which is why we avoid the notation $H^1_0(\Om;\Tp)$.
\end{remark}

We now show that the spaces in Definition~\ref{def:H1limitspaces} are continuously embedded into the corresponding spaces of functions of bounded variation.

\begin{lemma}\label{lem:BV_embedding_H1TP}
The space $H^1(\Om;\Tp)$ is continuously embedded in $BV(\Om)$. The space $H^1_D(\Om;\Tp)$ is continuously embedded in $BV(\R^\dim)$, where functions in $H^1_D(\Om;\Tp)$ are considered to be extended by zero to $\R^\dim$.
\end{lemma}
\begin{proof}
Consider first the case of $H^1(\Om;\Tp)$ and let $v\in H^1(\Om;\Tp)$ be arbitrary.
Recall that $\pair{Dv,\bphi}_\Om$ is given by \eqref{eq:distderiv_H1Tp} for any $\bphi\in \Cid$.
Thus $\abs{\pair{Dv,\bphi}_{\Om}}\leq \int_\Om \abs{\Npw v}+\int_{\Fp} \abs{\jump{v}}$ for any $\bphi\in \Cid$ such that $\norm{\bphi}_{C(\overline{\Om},\R^\dim)}\leq 1$.
Since $\Om$ is bounded, we get $\norm{\nabla v}_{L^1(\Om)}\lesssim \norm{v}_{H^1(\Om;\Tp)}$.
For the term involving jumps, the Cauchy--Schwarz inequality gives $\int_{\Fp} \abs{\jump{v}}\leq \left(\int_{\Fpi}h_+^{-1}\abs{\jump{v}}^2\right)^{\frac{1}{2}} \left(\int_{\Fpi}h_+\right)^\frac{1}{2}$.
To bound $\int_{\Fpi}h_+$, consider any face $F\in\Fp$, and let $K\in\Tp$ be an element that contains $F$.
Then, by shape-regularity of the meshes, we have $\int_F h_+ = \left(\mathcal{H}^{\dim-1}(F)\right)^{\frac{d}{d-1}} \lesssim \abs{K}$, and thus, after a counting argument, we get $\int_{\Fpi}h_+\lesssim \abs{\Omp}\leq \abs{\Om}<\infty$ since $\Om$ is bounded.
These bounds then imply that $\abs{Dv}(\Om)\lesssim \norm{v}_{H^1(\Om;\Tp)}$ and thus $H^1(\Om;\Tp)$ is continuously embedded in $BV(\Om)$.
The proof of the corresponding claim for $H^1_D(\Om;\Tp)$ is similar to the one given above, where we only need to additionally use the fact that functions in $H^1_D(\Om;\Tp)$, once extended by zero to $\R^\dim$, remain compactly supported.
\Fqed\end{proof}

\begin{theorem}\label{thm:completeness_H1tp}
The space $H^1(\Om;\Tp)$ is a Hilbert space with the inner-product
\begin{equation*}
\begin{aligned}
\pair{w,v}_{H^1(\Om;\Tp)}\coloneqq\int_\Om \left[\Npw w{\cdot}\Npw v + wv \right]+\int_{\Fpi} h_+^
{-1}\jump{w}\jump{v} &&&\forall w,\, v\in H^1(\Om;\Tp).
\end{aligned}
\end{equation*}
The space $H^1_D(\Om;\Tp)$ is a Hilbert space with the inner-product 
\begin{equation*}
\begin{aligned}
\pair{w,v}_{H^1_D(\Om;\Tp)}\coloneqq  \int_\Om \left[\Npw w{\cdot}\Npw v + wv \right]+\int_{\Fp}h_+^{-1}\jump{w}\jump{v} &&&\forall w,\,v \in H^1_D(\Om;\Tp).
\end{aligned}
\end{equation*}
\end{theorem}
\begin{proof}
It is clear that the spaces $H^1(\Om;\Tp)$ and $H^1_D(\Om;\Tp)$ are inner-product spaces when equipped with their respective inner-products, so it is enough to show that they are complete.
We give the proof in the case of $H^1_D(\Om;\Tp)$ as it is similar for $H^1(\Om;\Tp)$. Consider a Cauchy sequence $\{v_k\}_{k\in\N} \subset H^1_D(\Om;\Tp)$. Then, the continuous embedding of $H^1_D(\Om;\Tp)$ into $BV(\R^\dim)$ implies the existence of a $v\in BV(\R^\dim)$ such that $v_k\tends v$ in $BV(\R^\dim)$. Since convergence in $BV(\R^\dim)$ implies convergence in $L^1(\R^\dim)$, and the $v_k$ form a Cauchy sequence in $L^2(\R^\dim)$, by uniqueness of limits we then deduce that $v\in L^2(\R^\dim)$ and that $v_k \tends v$ in $L^2(\R^\dim)$. In particular, $v=0$ a.e.\ on $\R^\dim\setminus\Om$.
Furthermore, continuity of the trace operator from $BV(K)$ to $L^1(\p K)$ for each $K\in\Tp$ implies that $\jump{v_k}_F \tends \jump{v}_F \in L^1(F)$ for each $F\in\Fp$, and again the functions $\jump{v_k}_F$ form a Cauchy sequence in $L^2(F)$, so we deduce similarly that $\jump{v}_F\in L^2(F)$ for all $F\in\Fp$.
Additionally, using a diagonal argument over the countable set $\Fp$, we may extract a subsequence $\{v_{k_j} \}_{j\in\N}$ such that $\jump{v_{k_j}}\tends \jump{v}$ pointwise $\mathcal{H}^{\dim-1}$-a.e.\ on $\Sp$, recalling that $\Sp\coloneqq\bigcup_{F\in\Fp} F$.
Therefore, Fatou's Lemma implies that $\int_{\Fp}h_+^{-1}\abs{\jump{v}}^2=\int_{\Sp}h_+^{-1}\abs{\jump{v}}^2<\infty$ and that $\int_{\Fp} h_{+}^{-1} \abs{\jump{v-v_k}}^2_F=\int_{\Sp}  h_{+}^{-1} \abs{\jump{v-v_k}}^2_F \leq \liminf_{j\tends\infty}\int_{\Sp}h_+^{-1}\abs{\jump{v_{k_j}-v_k}}^2  \tends 0$ as $k\tends \infty$.
Then, using the fact that $\Npw v_k$ is a Cauchy sequence in $\Ld$, it is easy to show that the distributional derivative $Dv$  is also of the form in~\eqref{eq:distderiv_H100Tp} and that $\Npw v_k \tends \Npw v$ in $\Ld$. This implies that $v\in H^1_D(\Om;\Tp)$ and that $v_k\tends v $ as $k\tends\infty$.
\Fqed\end{proof}

The following Theorem shows that functions in $H^1(\Om;\Tp)$ and $H^1_D(\Om;\Tp)$ can be approximated by functions from the same space that have at most finitely many nonvanishing jumps.
\begin{theorem}\label{thm:H1tp_finite_jumps}
For every $v\in H^1(\Om;\Tp)$, respectively $v\in H^1_D(\Om;\Tp)$, there exists a sequence of functions $v_k \in H^1(\Om;\Tp)$ for all $k\in\N$, respectively $v_k\in H^1_D(\Om;\Tp)$ for all $k\in\N$, such that $\lim_{k\tends \infty }\norm{v-v_k}_{H^1(\Om;\Tp)} =0$, respectively $\lim_{k\tends \infty }\norm{v-v_k}_{H^1_D(\Om;\Tp)} =0$ and such that, for each $k\in \N$, there are only finitely many faces $F\in\Fpi$, respectively $\Fp$, such that $\jump{v_k}_F\neq 0$.
Moreover, $v_k=v$ and $\Npw v_k=\Npw v$ a.e.\ on $\Omp_k\cup\Omm$ for each $k\in\N$, and $\int_{\Tp}h_+^{-2}\abs{v-v_k}^2\tends 0$ as $k\tends \infty$.
\end{theorem}

We postpone the proof of Theorem~\ref{thm:H1tp_finite_jumps} until after the proof of Theorem~\ref{thm:finite_approx} below, owing to the similar nature of the two results and the similarities in their proofs.

\begin{corollary}\label{cor:H1_omm_restriction}
For every $v\in H^1_D(\Om;\Tp)$, there exists a $w\in H^1_0(\Om)$ such that $v=w$ and $\nabla v= \nabla w$ a.e.\ on~$\Omm$.
\end{corollary}

\begin{proof}
If $\Omm$ is empty then there is nothing to show, so we need only consider the case where $\Omm$ is nonempty.
Choose $k\in \N$ and let $v_k \in H^1_D(\Om;\Tp)$ be given by Theorem~\ref{thm:H1tp_finite_jumps}.
We infer from $\Fp=\bigcup_{\ell\in\N}\Flp$ with ascending sets $\Flp$, c.f.\ Section~\ref{sec:refinement_sets}, that there exists $m=m(k)$ such that $v_k$ has nonzero jumps only on $\cF^{\dagger}_m$, i.e.\ $\jump{v_k}_F=0$ for every face $F\in \Fp\setminus \cF^{\dagger}_m$.
Since any element in $\Tp$ is by definition closed, it follows that $\Omp_m$ is a finite union of closed sets, and moreover it follows from Lemma~\ref{lem:eventuallyinT+} that $\Omp_m$ is disjoint from $\overline{\Omm}$. Therefore, $\Omp_m$ and $\overline{\Omm}$ are two disjoint compact sets in $\R^\dim$, so there exists a $\eta \in C^\infty_0(\R^\dim)$ such that $\eta|_{\Omm}=1$ and $\eta|_{\Omp_m}=0$.
Then, define $w(x) \coloneqq \eta(x) v_k(x) $ for all $x\in \R^\dim$, where we recall that $v_k$ is extended by zero outside of $\Om$.
We see that $w=v$ a.e.\ on $\Omm$ immediately from the facts that $v_k=v$ on $\Omm$ and $\eta=1$ on $\Omm$. 
It remains only to show that $w\in H^1_0(\Om)$.
Note that $v_k=0$ on $\R^\dim\setminus \overline{\Om}$ by definition, therefore $w=0$ on $\R^\dim\setminus \overline{\Om}$.
Then, for any test function $\bphi \in \CidR$, we have
\begin{equation}\label{eq:H1_omm_restriction_1}
\langle Dw , \bphi\rangle_{\R^\dim} = \int_{\R^\dim} \left[ -v_k \Div (\eta \bphi) + v_k \nabla \eta \cdot \bphi\right] = \langle D v_k, \eta\bphi\rangle_{\R^\dim}+ \int_{\R^\dim}v_k \nabla \eta \cdot \bphi.
\end{equation}
Since $v_k\in H^1_D(\Om;\Tp)$ has a distributional derivative satisfying \eqref{eq:distderiv_H100Tp}, and since $\eta$ vanishes identically on every face $F\in \cF^{\dagger}_m \subset \Omp_m$, whereas $\jump{v_k}_F=0$ for every face $F\in \Fp\setminus \cF^{\dagger}_m$; we then see from~\eqref{eq:H1_omm_restriction_1} that
\begin{equation}
\langle D w,\bphi\rangle_{\R^\dim} =  \int_{\Om} (\eta \Npw v_k + v_k \nabla \eta){\cdot} \bphi - \int_{\Fp}\jump{v_k}(\eta \bphi\cdot \bn)=
\int_{\Om} (\eta \Npw v_k + v_k\nabla \eta)\cdot \bphi
\end{equation}
for all $\bphi \in \CidR$, which implies that $w\in H^1(\R^\dim)$ and $\nabla w = \eta \Npw v_k + v_k\nabla \eta$. Since $w=0$ outside $\overline{\Om}$, we conclude that $w\in H^1_0(\Om)$ by \cite[Theorem~5.29]{AdamsFournier03}, and since $\eta=1$ on $\Omm$, we find that $\nabla w = \nabla v_k = \nabla v$ a.e.\ on $\Omm$, which completes the proof.
\Fqed\end{proof}

We now turn to some key properties of the space $H^1(\Om;\Tp)$, namely that it enjoys a Poincar\'e and $L^2$-trace inequalities. For an element $K\in \Tk$ for some $k\in\N$, let $\mathcal{F}^+_{\circ}(K)$ denote the set of faces in $\Fpi$ that are contained in $K$ but do not lie entirely on the boundary of $K$, i.e.
\begin{equation}\label{eq:fp_circ}
\mathcal{F}^+_{\circ}(K) \coloneqq \{ F\in\Fpi\colon F\subset K,\; F\not\subset\p K \}.
\end{equation}
Note that by definition no boundary face of $\Fp$ can intersect the interior of any element of any mesh.
The following Theorem shows that functions in $H^1(\Om;\Tp)$ enjoy a Poincar\'e inequality over elements of the meshes $\Tk$, with optimal scaling with respect to element sizes. Recall that $h_K=\abs{K}^{\frac{1}{\dim}}\eqsim \diam K$ owing to shape-regularity of the sequence of meshes. 

\begin{theorem}[Poincar\'e inequality]\label{thm:poincare}
For every $k\in\N$ and any $K\in \Tk$, we have 
\begin{equation}\label{eq:poincare}
h_K^{-2}\int_K \abs{v-\overline{v_K}}^2 \lesssim \int_K \abs{\Npw v }^2 + \int_{\Fp_{\circ}(K)} h_+^{-1} \abs{\jump{v}}^2 \quad \forall v\in H^1(\Om;\Tp),
\end{equation}
where $\overline{v_K}$ denotes the mean-value of $v$ over $K$ and $\Fp_{\circ}(K)$ is defined in~\eqref{eq:fp_circ}.
\end{theorem}
\begin{proof}
Let $v\in H^1(\Om;\Tp)$ be arbitrary. 
Since $v\in L^2(\Om)$ it is clear that the restriction of the distributional derivative $Dv$ to $K$ is in $H^{-1}(K;\R^\dim)$.
We start by showing that
\begin{equation}\label{eq:distderiv_H10}
\norm{D v}_{H^{-1}(K;\R^\dim)} \lesssim h_K \left( \int_K \abs{\Npw v }^2 + \int_{\Fp_{\circ}(K)} h_+^{-1} \abs{\jump{v}}^2\right)^{\frac{1}{2}},
\end{equation}
where $\norm{D v}_{H^{-1}(K)}\coloneqq \sup \{\abs{ \pair{D v,\bphi}_{K} }\colon \bphi \in H^1_0(K;\R^\dim), \norm{\nabla \bphi}_K =1\}$.
By density of smooth compactly supported functions in $H^1_0(K;\R^\dim)$, it is enough to show~\eqref{eq:distderiv_H10} for $\bphi\in C^\infty_0(K;\R^\dim)$.
 Consider now an arbitrary $\bphi\in C^\infty_0(K;\R^\dim)$, and extend it by zero to $\Omega$. Then $\pair{Dv,\bphi}_K=\pair{Dv,\bphi}_\Om$ is given by \eqref{eq:distderiv_H1Tp}. Since $\bphi$ is compactly supported in $K$ and vanishes on faces in $\Fp\setminus \Fp_\circ(K)$, the Cauchy--Schwarz inequality gives 
 \begin{equation}
 \abs{\pair{Dv,\bphi}_K}\leq \norm{\Npw v}_K \norm{\bphi}_K+  \left(\int_{\Fp_\circ(K)}h_+^{-1}\abs{\jump{v}}^2\right)^{\frac{1}{2}} \left(\int_{\Fp_\circ(K)} h_+ \norm{\bphi}_F^2\right)^{\frac{1}{2}}.
 \end{equation}
Then, the multiplicative trace inequality, applied to the parent elements from $\Tp$ of each face $F\in \Fp_\circ(K)$, and the Cauchy--Schwarz inequality imply that
\begin{equation}
 \begin{split}
 \int_{\Fp_\circ(K)} h_+ \norm{\bphi}_F^2 &\lesssim \sum_{K^\prime\in\Tp(K)}\mkern-18mu \left[ h_{K^\prime} \norm{\nabla \bphi}_{K^\prime}\norm{\bphi}_{K^\prime } + \norm{\bphi}_{K^\prime}^2\right]
 \\ &\leq h_K \norm{\nabla \bphi}_K \norm{\bphi}_K + \norm{\bphi}_K^2 \lesssim h_K^2 \norm{\nabla \bphi}^2_K,
 \end{split}
 \end{equation}
where $\Tp(K)\coloneqq \{K^\prime\in \Tp\colon K^\prime \subset K\}$ is the set of elements of $\Tp$ contained in $K$, and where we have used the Poincar\'e--Friedrichs inequality $\norm{\bphi}_K\lesssim h_K \norm{\nabla \bphi}_K$ for $\bphi \in C^\infty_0(K;\R^\dim)$.
This implies that $\pair{Dv,\bphi}_K$ is bounded by the right-hand side of~\eqref{eq:distderiv_H10} for all $\bphi \in C^\infty_0(K;\R^\dim)$ such that $\norm{\nabla \bphi}_K=1$, and thus $Dv$ extends to a distribution in $H^{-1}(K)$ satisfying~\eqref{eq:distderiv_H10}.
Next, we use the fact that for any $v \in L^2(K)$, there exists a vector field $\bphi \in  H^1_0(K;\R^\dim) $ such that $\Div \bphi = v - \overline{v_K}$ in $K$ and such that $\norm{\nabla\bphi}_{K} \lesssim \norm{v-\overline{v_K}}_{K}$, see~\cite{Bogovskii79}.
In particular we may take the constant to depend only on the shape-regularity of the meshes and on the spatial dimension, since $\bphi$ can be obtained by mapping back to a reference element through the Piola transformation, see e.g. the textbook~\cite[p.~59]{BoffiBrezziFortin2013}.
Then, noting that $\int_K \Div \bphi \overline{v_K} =0$, we obtain
\[
\norm{v-\overline{v_K}}_K^2= \int_K v \Div \bphi = - \langle Dv, \bphi \rangle_K \lesssim \norm{Dv}_{H^{-1}(K;\R^\dim)} \norm{\nabla\bphi}_{K },
\]
and then we use~\eqref{eq:distderiv_H10} and $\norm{\nabla\bphi}_{K}\lesssim \norm{v-\overline{v_K}}_{K}$ to obtain~\eqref{eq:poincare}.
\Fqed\end{proof}

For each $K\in\Tk$, recall that $\tauk  \colon BV(K)\tends L^1(\partial K)$ denotes the trace operator. We now show that functions in $H^1(\Om;\Tp)$ have traces in $L^2$ over all element boundaries. Recall again that $h_K=\abs{K}^{\frac{1}{\dim}} \eqsim \diam K$ owing to the shape-regularity of the meshes.

\begin{theorem}[Trace inequality on element boundaries]\label{thm:trace}
For every $k\in\N$ and every $K\in\Tk$, the trace operator $\tauk $ is a bounded operator from $H^1(\Om;\Tp)$  to $L^2(\p K)$ and satisfies
\begin{equation}\label{eq:trace_inequality}
h_K^{-1}\int_{\p K} \abs{\tauk v}^2 \lesssim  \int_K \left[\abs{\Npw v}^2 + h_K^{-2}\abs{v}^2\right] + \int_{\mathcal{F}^+_{\circ}(K)}h_+^{-1} \abs{\jump{v}}^2 \quad\forall v \in H^1(\Om;\Tp),
\end{equation}
where $\mathcal{F}^+_{\circ}(K)$ is defined by~\eqref{eq:fp_circ}.
\end{theorem}
\begin{proof}
We start by showing \eqref{eq:trace_inequality} for functions $v\in H^1(\Om;\Tp)$ that have non-vanishing jumps on only finitely many faces of $\Fp$, and we will extend the result to all of $H^1(\Om;\Tp)$ with the density result of Theorem~\ref{thm:H1tp_finite_jumps}. 
First, suppose that there is a $\ell\in\N$ such that $\jump{v}_F=0$ for all $F\in\Fpi\setminus \cF_{\ell}^+$. It is then easy to see that $v|_{K^\prime}\in H^1(K^\prime)$ for any $K^\prime \in \cT_{\ell}$, because the interior of any element $K^\prime$ of $\cT_{\ell}$ is disjoint from all faces in $\cF_\ell^+$.
Now, if $\ell< k$, then there is nothing to show as $v\in H^1(K)$ and the inequality \eqref{eq:trace_inequality} is then simply the scaled trace inequality for functions in $H^1(K)$, and the jump terms in the right-hand side of~\eqref{eq:trace_inequality} would then vanish as a sum over an empty set.
If $\ell\geq k$, then let $\cT_{\ell}(K)\coloneqq\{K^\prime\in \cT_\ell\colon K^\prime \subset K\}$ denote the set of children of $K$ in the mesh $\cT_{\ell}$, and note that $\cT_{\ell}(K)$ forms a conforming shape-regular triangulation of $K$ by nestedness of the meshes. Moreover, the function $v$ is piecewise $H^1$-regular over $\cT_{\ell}(K)$.
Then, inequality~\eqref{eq:trace_inequality} holds owing to~\cite[Lemma~3.1]{FengKarakashian01}, which proves the trace inequality~\eqref{eq:trace_inequality} for piecewise $H^1$-regular functions with respect to finite subdivisions of an element. 
To generalise the result to all functions in $H^1(\Om;\Tp)$, consider now an arbitrary $v\in H^1(\Om;\Tp)$ and let $\{v_{\ell}\}_{\ell\in\N}\subset H^1(\Om;\Tp)$ denote the sequence given by Theorem~\ref{thm:H1tp_finite_jumps} (indexed now by $\ell$).
The continuous embedding of $H^1(\Om;\Tp)$ into $BV(\Om)$, given by Lemma~\ref{lem:BV_embedding_H1TP}, shows that $v_{\ell}\tends v$ in $BV(\Om)$ as $\ell\tends\infty$, so the traces $\tauk v_{\ell} \tends \tauk v$ in $L^1(\p K)$ as $\ell\tends\infty$.
But then, after extracting a subsequence (without change of notation), we can assume that $\tauk v_{\ell}\tends \tauk v$ pointwise $\mathcal{H}^{\dim-1}$-a.e.\ on $\p K$ as $\ell\tends \infty$. Fatou's lemma then allows us to conclude that $\tauk v \in L^2(\p K)$ and that~\eqref{eq:trace_inequality} holds for general $v\in H^1(\Om;\Tp)$.
\Fqed\end{proof}

\subsection{Second-order space, symmetry of Hessians and approximation by quadratic polynomials.}

We now turn towards the second key step in constructing a suitable limit space for the sequence of finite element spaces. In Definition~\ref{def:HD_def} below, we introduce a space of functions with suitably regular gradients and Hessians and sufficiently integrable jumps in values and gradients over never-refined faces. 
Recall that we consider here the notion of Hessian defined in~\eqref{eq:Hessian_notation} for functions of bounded variation with gradients of bounded variation.

\begin{definition}\label{def:HD_def}
Let $\HD$ denote the space of functions $v \in H^1_D(\Om;\Tp)$ such that $\Npw_{x_i} v \in H^1(\Om;\Tp)$ for all $i=1,\dots,\dim$, where $\Npw v=(\Npw_{x_1} v,\dots,\Npw_{x_\dim} v)$, and such that
\begin{equation}\label{eq:HD_norm_def}
\norm{v}_{\HD}^2 \coloneqq \int_\Om \left[\abs{\Dpw v}^2 + \abs{\Npw v}^2 + \abs{v}^2 \right] + \int_{\Fpi} h_{+}^{-1}\abs{\jump{\nabla v}}^2 + \int_{\Fp} h_{+}^{-3}\abs{\jump{v}}^2  <\infty.
\end{equation}
\end{definition}

Note that each component $\Npw_{x_i} v$ has a distributional derivative of the form~\eqref{eq:distderiv_H1Tp} if and only if
\begin{equation}\label{eq:distderiv_h2}
\pair{D(\Npw v),\bvphi}_{\Om} \coloneqq - \int_\Om \Npw v\cdot \Div \bvphi = \int_\Om \Dpw v : \bvphi - \int_{\Fpi} \jump{\Npw v}\cdot(\bvphi \bn),
\end{equation}
for all $\bvphi \in \Cidd$, where the divergence $\Div\bvphi$ is defined by $(\Div \bvphi)_i \coloneqq \sum_{j=1}^\dim \nabla_{x_j} \bvphi_{ij}$ for all $i\in\{1,\dots,\dim\}$.
Therefore, a function $v\colon \Om\tends \R$ belongs to $\HD$ if and only if $v\in H^1_D(\Om;\Tp)$, if $D(\nabla v)$ is of the form given in~\eqref{eq:distderiv_h2}, and if $\norm{v}_{\HD}<\infty$.
The space $\HD$ is clearly non-empty and contains $H^2(\Om)\cap H^1_0(\Om)$ as a closed subspace.

\begin{theorem}[Completeness]\label{thm:completeness_H2}
The space $\HD$ is a Hilbert space under the inner-product
\begin{multline}\label{eq:HD_innerprod_def}
\pair{w,v}_{\HD}\coloneqq \int_\Om\left[\Dpw w:\Dpw v+\Npw w\cdot \Npw v+wv\right]\\ +\int_{\Fpi}h_+^{-1}\jump{\Npw w}\cdot\jump{\Npw v}+\int_{\Fp}h_+^{-3}\jump{w}\jump{v},
\end{multline}
for all $w$, $v\in\HD$.
\end{theorem}
\begin{proof}
It is clear that $\HD$ is an inner-product space when equipped with the inner-product defined above, so it is enough to show that it is complete. Considering a Cauchy sequence $\{v_k\}_{k\in\N}$, it follows from Theorem~\ref{thm:completeness_H1tp} that there exists a $v\in H^1_D(\Om;\Tp)$ such that $v_k\tends v$ in $H^1_D(\Om;\Tp)$; moreover Theorem~\ref{thm:completeness_H1tp} also shows that $\Npw_{x_i} v\in H^1(\Om;\Tp)$ with $\Npw_{x_i} v_k \tends \Npw_{x_i} v$ in $H^1(\Om;\Tp)$ for each $i\in\{1,\dots,\dim\}$. This implies in particular that $\Dpw v_k \tends \Dpw v $ in $\Ldd$.
Then, using a pointwise a.e.\ convergent subsequence for the jumps over faces, similar to the one in the proof of Theorem~\ref{thm:completeness_H1tp}, we find also that $\int_{\Fp}h_+^{-3}\abs{\jump{v}}^2<\infty$ and $\int_{\Fp}h_+^{-3}\abs{\jump{v-v_k}}^2\tends 0$ as $k\tends \infty$. This proves that $v\in\HD$ and that $v_k\tends v$ in~$\HD$ as $k\tends \infty$. Therefore $\HD$ is complete.
 \Fqed\end{proof}

\begin{remark}[Piecewise $H^2$-regularity on $\Tp$]\label{rem:pw_H2_regularity}
As explained already in Remark~\ref{rem:notation}, a function $v\in \HD\subset H^1_D(\Om;\Tp)$ is piecewise $H^1$-regular over $\Tp$, i.e.\ $v|_K\in H^1(K)$ for all $K\in\Tp$, and $(\Npw v)|_K $ is equal to the weak gradient of $v|_K$ over $K$.
By definition, $\Npw_{x_i} v \in H^1(\Om;\Tp)$ so likewise $\Npw_{x_i} v|_K \in H^1(K)$ for all $i=1,\dots,\dim$, and hence $v|_K \in H^2(K)$ for all $K\in\Tp$ and $\Npw^2 v|_K$ equals the weak Hessian of $v|_K$ over $K$, for each $K\in\Tp$.
\end{remark}

\begin{remark}[Symmetry of the Hessians]\label{rem:symmetry}
The space $\HD$ is continuously embedded in the space $SBV^2(\Om)$, which is defined as the space of functions $v\in SBV(\Om)$ such that $\Npw v \in SBV(\Omega;\R^\dim)$ \cite{AmbrosioFuscoPallara00,FonsecaLeoniParoni05}.
There generally exists functions $v \in SBV^2(\Om)$ such that $\Dpw v \coloneqq \Npw (\Npw v)$ fails to be symmetric, see~\cite{FonsecaLeoniParoni05}.
It is thus not \emph{a priori} obvious that $\Dpw v$ should be symmetric for a general function $v\in \HD$, yet the symmetry of the Hessian is essential for the approximation theory required to construct a suitable limit space for the sequence of finite element spaces.
One of the principal contributions of our work below is a proof that $\Dpw v$ is indeed symmetric a.e.\ on $\Omega$ for all $v\in\HD$, see Corollary~\ref{cor:H2_omm_restriction} below.
We immediately note that symmetry of $\Dpw v$ over the subset $\Omp$ is a consequence of piecewise $H^2$-regularity over $\Tp$ as explained in Remark~\ref{rem:pw_H2_regularity}, so the difficulty is to show the symmetry of $\Dpw v $ on $\Omm$.
\end{remark}

\begin{figure}
\begin{center}
\includegraphics[height=2.5cm]{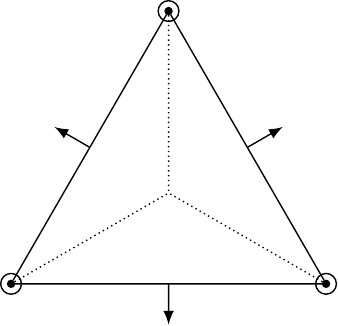}
\hspace{1cm}
\includegraphics[height=2.5cm]{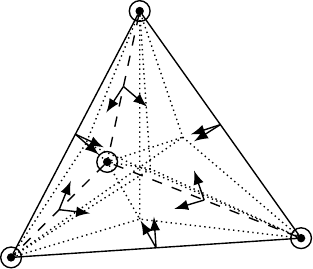}
\caption{Degrees of freedom of the cubic Hsieh--Clough--Tocher (HCT) macro-element in two (left) and three (right) space dimensions on a reference element. The basis functions are $C^1$-regular and piecewise cubic with respect to subdivisions of the element into subsimplices \cite{CloughTocher66,DouglasDupontPercellScott79,WorseyFarin87}.
Solid dots represent degrees of freedom associated to point values, the circles represent gradient values, and the arrows represent directional derivative values.}
\label{fig:HCT_elements}
\end{center}
\end{figure}

The next Theorem shows that the subspace of functions in $\HD$ that have nonvanishing jumps in the values and gradients on at most finitely many faces of $\Fp$ forms a dense subspace of $\HD$. 
This result is the key to proving the symmetry of the Hessians of functions in $\HD$.

\begin{theorem}\label{thm:finite_approx}
For each $v\in\HD$, there exists a sequence of functions $v_k\in \HD$ for all $k\in\N$ such that
\begin{equation}\label{eq:finite_approx_1}
\lim_{k\tends \infty }\norm{v-v_k}_{\HD} =0,
\end{equation}
and such that, for each $k\in\N$, there exists only finitely many faces $F\in\Fp$, respectively $F\in\Fpi$, for which $\jump{v_k}_F\neq 0$, respectively $\jump{\Npw v_k}_F\neq 0$.
Moreover, $v_k=v$, $\Npw v_k=\Npw v$ and $\Dpw v_k=\Dpw v$ a.e.\ on $\Omp_k\cup\Omm$ for each $k\in\N$, and additionally
\begin{equation}\label{eq:finite_approx2}
\lim_{k\tends \infty}\int_{\Tp} \left[h_+^{-4}\abs{v-v_k}^2 + h_{+}^{-2} \abs{\Npw(v-v_k)}^2 \right] =0.
\end{equation}
\end{theorem}

\begin{proof}
The proof is composed of four key steps.

\emph{Step~1. Construction of $v_k$.}
For each $k\in\N$, the function $v_k$ is defined as follows. First, let $v_k=v$ on $\Omm$. Then, for each $K \in \Tp$, if $K\in \Tk^+$ then let $v_k|_K=v|_K$.
Otherwise, if $K\in \Tp\setminus \Tk^+$, we define $v_k|_K$ in terms of a quasi-interpolant into the cubic HCT space, by first taking element-wise $L^2$-orthogonal projections in the neighbourhood of $K$ and then applying a local averaging of the degrees of freedom of the projections.
 We shall define $v_k$ in this manner with respect to the possibly countably infinite set of elements in $\Tp$, yet we note that the construction is entirely local to each element and its neighbours. As explained above, the neighbourhood of any element is the same as that from a finite mesh, and thus the standard techniques of analysis on finite meshes extend to the present setting. 
The analysis of local averaging operators is rather standard by now, see e.g.~\cite{GeorgoulisHoustonVirtanen11,HoustonSchotzauWihler07,KarakashianPascal03,NeilanWu19}.
For simplicity, we give the details only for $\dim=2$, whereas for $\dim=3$ we outline the main ingredients in~Remark~\ref{rem:3dHCT} below, which is handled in a similar manner using the three-dimensional cubic HCT element depicted in Figure~\ref{fig:HCT_elements}.

Let $\Vct$ denote the cubic HCT macro-element space over $K$, which consists of all $C^1(\overline{K})$-regular functions over $K$ that are piecewise cubic with respect to the barycentric refinement of $K$, see~\cite{CloughTocher66,DouglasDupontPercellScott79} and the textbook~\cite{Ciarlet02} for a full definition.
The degrees of freedom of $\Vct$ are depicted in~Figure~\ref{fig:HCT_elements} above. 
In particular $\Vct$ contains all cubic polynomials over $K$. 
For each $K^\prime\in\Tp$, let $\pi_2 v|_{K^\prime} \in \mathbb{P}_2$ denote the $L^2$-orthogonal projection of $v$ over $K^\prime$ into the space of quadratic polynomials. Thus $\pi_2 v$ is a piecewise quadratic function over $\Tp$.
Then, for each $K\in\Tp\setminus\Tkp$, we define $v_k|_K\in \Vct$ by local averaging of the degrees of freedom of $\pi_2 v$ as follows. 
Let $\mathcal{V}_K$ denote the set of vertices of $K$ and let $\mathcal{M}_K$ denote the set of mid-points of the faces of $\p K$. We call $\mathcal{V}_K\cup\mathcal{M}_K$ the set of nodes.
For a node $z \in \mathcal{V}_K\cup\mathcal{M}_K$, let $N_+(z)\coloneqq \{K^\prime\in\Tp\colon z\in K^\prime\}$ denote the set of neighbouring elements of $K$ that contain $z$, and let $\abs{N_+(z)}$ denote its cardinality.
Note that $N_+(z)\subset N_+(K)$ for any $z\in\mathcal{V}_K\cup \mathcal{M}_K$, where we recall that $N_+(K) $ is the set of neighbouring elements of $K$ in $\Tp$.
Let $\mathcal{V}_K^I$ and $\mathcal{M}_K^I$ denote the set of interior vertices and interior face-midpoints, respectively. 
We separate boundary vertices into two categories: if a vertex $z\in\mathcal{V}_K$ is on the boundary, and if all boundary faces containing $z$ are coplanar, then we say that $z$ is a flat vertex and we write $z\in\mathcal{V}_K^{\flat}$, otherwise we say that $z$ is a sharp vertex and we write $z\in\mathcal{V}_K^{\sharp}$.
We then define $v_k|_K$ for all $K\in\Tp\setminus \Tkp$ in terms of the degrees of freedom by
\begin{equation}\label{eq:CT_dofs}
\begin{aligned}
(v_k|_K)(z) &\coloneqq
\begin{cases}
\frac{1}{\abs{N_+(z)}}\sum_{K^\prime\in\N_+(z)}   (\pi_2 v|_{K^\prime})(z) &\hspace{2cm}\text{if } z\in \mathcal{V}_K^I, \\
0 &\hspace{2cm}\text{if } z \in \mathcal{V}_K^{\flat}\cup\mathcal{V}_K^{\sharp},
\end{cases}
\\
 (\nabla v_k|_K)(z) &\coloneqq \begin{cases}
\frac{1}{\abs{N_+(z)}}\sum_{K^\prime\in\N_+(z)}  \nabla(\pi_2 v|_{K^\prime})(z) &\text{if }z\in \mathcal{V}_K^{I},\\
\frac{1}{\abs{N_+(z)}}\sum_{K^\prime\in\N_+(z)} (\nabla(\pi_2 v|_{K^\prime})(z) \cdot \bn_{\p\Omega} ) \bn_{\p\Omega} &\text{if } z\in\mathcal{V}_K^{\flat}, \\
0 &\text{if }z\in\mathcal{V}_K^{\sharp},
 \end{cases}
 \\
(\nabla v_k|_K)(z)\cdot \bn_F &\coloneqq
\begin{cases}
\frac{1}{\abs{N_+(z)}}\sum_{K^\prime\in\N_+(z)}  (\nabla(\pi_2 v|_{K^\prime})(z)\cdot \bn_F) &\hspace{0.75cm}\text{if }z\in\mathcal{M}_K^I,\\
\nabla (\pi_2 v|_{K})(z)\cdot \bn_F &\hspace{0.75cm}\text{if }z\in \mathcal{M}_K\setminus\mathcal{M}_K^I,
\end{cases}
\end{aligned}
\end{equation}
where, in the notation above, $\bn_F$ is the chosen unit normal for the face $F$ containing the edge-midpoint $z\in\mathcal{M}_K$, and $\bn_{\p\Omega}$ denotes the unit outward normal to $\Omega$ at $z$ if $z\in\mathcal{V}_K^{\flat}$.

Still considering $K\in\Tp\setminus\Tkp$, it follows that $v_k|_K \in C^1(\overline{K})\cap H^2(K)$, and that, for any boundary face $F \subset \p K$, $\jump{v_k}_F=v_k|_F=0$ owing to the vanishing values and vanishing first tangential derivatives at both boundary vertices on $F$. Moreover, if $K^\prime \in N_+(K)$ is a neighbouring element such that $K^\prime \in \Tp\setminus\Tkp$, then we note that all degrees freedom of $v_k|_{K^\prime}$ and $v_k|_{K}$ that belong to their common face $F$ coincide by definition, which implies that $\jump{v_k}_F =0 $ and $\jump{\Npw v_k}_F=0$.
Furthermore, following standard techniques involving inverse inequalities, see e.g.\ \cite{KarakashianPascal03}, we obtain the bound
\begin{equation}\label{eq:CT_quasi_bound_jump}
 \sum_{m=0}^2 \int_K h_+^{2m-4}\abs{\nabla^m( \pi_2v - v_k  ) }^2 \lesssim \int_{\Fpi_K}h_+^{-1}\abs{\jump{\Npw \pi_2 v}}^2 + \int_{\Fp_K} h_+^{-3}\abs{\jump{\pi_2 v}}^2 ,
\end{equation}
where $\Fp_K=\{F\in\Fp\colon F\cap K\neq \emptyset\}$ and $\Fpi_K\coloneqq \Fp_K\cap \Fpi$ are sets of faces adjacent to $K$. 
Note that~\eqref{eq:CT_quasi_bound_jump} corresponds to the generalisation of \cite[Lemma~3]{NeilanWu19} to fully discontinuous polynomials, and the only additional step required to obtain \eqref{eq:CT_quasi_bound_jump} beyond what is shown already in \cite{NeilanWu19} is the application of inverse inequalities on boundary faces to handle non-vanishing tangential derivatives of $\pi_2 v$.
Recalling that $v\in \HD$ is $H^2$-regular on each element of $\Tp$, we infer from the application of the triangle inequality that
\begin{multline}\label{eq:CT_quasi_interpolant_bound}
\sum_{m=0}^2\int_K h_+^{2m-4} \abs{\nabla^m(v-v_k)}^2 
\lesssim  \sum_{m=0}^2\int_K h_+^{2m-4} \left[\abs{\nabla^m(v-\pi_2v)}^2+\abs{\nabla^m(\pi_2 v - v_k)}^2\right]\\
\lesssim \sum_{m=0}^2\int_K h_+^{2m-4} \abs{\nabla^m(v-\pi_2v)}^2 + \int_{\Fpi_K}h_+^{-1}\abs{\jump{\Npw \pi_2 v}}^2 + \int_{\Fp_K} h_+^{-3}\abs{\jump{\pi_2 v}}^2 
 \\ \lesssim \int_{N_+(K)} \abs{\Dpw v}^2 + \int_{\Fpi_K} h_+^{-1}\abs{\jump{\Npw v}}^2 + \int_{\Fp_K} h_+^{-3}\abs{\jump{v}}^2,
\end{multline}
where in passing from the first to the second line we have applied~\eqref{eq:CT_quasi_bound_jump}, and in passing from the second to the third line we have applied a further triangle inequality $\pi_2 v = \pi_2 v -v + v$ along with trace inequalities and the Bramble--Hilbert Lemma applied to $v-\pi_2 v$.

\emph{Step~2. Proof that $v_k$ has at most finitely many nonvanishing jumps.}
We now show that $v_k$ has nonvanishing jumps on at most finitely many faces of $\Fp$ and $\Npw v_k$ has nonvanishing jumps on at most finitely many interior faces in $\Fpi$.
It is clear that $\jump{v_k}_F=0$ for all boundary faces $F\in \Fp \setminus\mathcal{F}_k^+$, because any boundary face $ F\in \Fp\setminus\Fkp $ must be a face of an element of $K\in \Tp\setminus \Tk^+$.
So there are only finitely many boundary faces where $\jump{v_k}$ does not vanish. To study interior faces, Lemma~\ref{lem:eventuallyinT+} implies that there is $\ell\in \N$, $\ell=\ell(k)\geq k$ such that $\Tk^+\subset \cT_{\ell}^{1+}$.
Then, consider a face  $F\in \Fpi \setminus \mathcal{F}_{\ell}^{\dagger}$, recalling the notation in~\eqref{eq:Fks_def}, and consider the elements $K$, $K^\prime \in \Tp$ forming $F$, i.e.\ $F=K\cap K^\prime$.
If either of $K$ or $K^\prime$ is in $\cT_k^+\subset \cT^{1+}_\ell$ then both must be in $\mathcal{T}_{\ell}^+$ and thus $F$ would have to be a face of $\mathcal{F}_\ell^{\dagger}$ by~\eqref{eq:Fks_def}, which would be a contradiction.
Therefore we have both $K$, $K^\prime \in \Tp\setminus\mathcal{T}_k^+$.
Then the definition of $v_k$ on $K$ and $K^\prime$ above implies that the degrees of freedom of $v_k$ coincide on $F$, so $\jump{v_k}_F=0$ and $\jump{\Npw v_k}_F=0$ for all $F\in \Fpi\setminus \cF_{\ell}^{\dagger}$.
Since there are at most only finitely many faces in $\mathcal{F}_{\ell}^{\dagger}$ we conclude that $\jump{v_k}=0$ and $\jump{\Npw v_k}=0$ except for at most finitely many faces of $\Fp$ and $\Fpi$, respectively.

\emph{Step~3. Proof of \eqref{eq:finite_approx2} and of convergence of jumps.} We now consider the convergence of the $v_k$ to $v$ over $\Omp$.  
Recall that $v=v_k$ on $\Om_k^{+}\cup \Omm$ by definition.
Furthermore, if $K\in\Tp\setminus \Tk^+$ then $N_+(K) \subset \Tp\setminus\Tk^{1+}$ because if $K$ has a neighbour in $\Tk^{1+}$ then $K$ itself must be in $\Tk^+$.
 Therefore, it follows from~\eqref{eq:CT_quasi_interpolant_bound} that
\begin{multline}\label{eq:finite_approx_volume_terms}
\sum_{m=0}^{2}\int_{\Tp} h_+^{2m-4} \abs{\Npw^{m}(v-v_k)}^2 = \sum_{m=0}^2 \int_{\Tp\setminus\Tk^+} h_+^{2m-4} \abs{\Npw^m (v-v_k)}^2
\\ \lesssim \int_{\Tp\setminus \Tk^{1+}} \abs{\Dpw v}^2 + \int_{\Fpi\setminus\Fk^{1\dagger}} h_+^{-1}\abs{\jump{\Npw v}}^2 + \int_{\Fp\setminus \Fk^{1\dagger}} h_+^{-3}\abs{\jump{v}}^2 ,
\end{multline}
where $\Fk^{1\dagger} $ denotes the set of all faces whose parent elements are in $\Tk^{1+}$.
Since $\Fp = \bigcup_{k\in \N}\Fk^{1\dagger}$ and since $\Tp = \bigcup_{k\in\N} \Tk^{1+}$ as a consequence of Lemma~\ref{lem:eventuallyinT+}, we see that the right-hand side in~\eqref{eq:finite_approx_volume_terms} tends to zero as $k\tends\infty $ as it is the tail of a convergent series. In particular, this proves \eqref{eq:finite_approx2}.

We now prove that
\begin{equation}\label{eq:finite_approx_jump_convergence}
\lim_{k\tends \infty}\left(\int_{\Fpi}h_+^{-1}\abs{\jump{\Npw(v-v_k)}}^2 + \int_{\Fp}h_+^{-3}\abs{\jump{v-v_k}}^2\right) =0.
\end{equation}
Recalling that $v_k=v$ on $\Tp$, we see $\jump{v-v_k}_F=0$ for all $F\in \Fks$.
Moreover, if $F\in \Fp\setminus \Fks$, then $F$ must be a face of at least one element of $\Tp\setminus \Tkp$. Also, if $F=K\cap K^\prime$ for some $K\in \Tp\setminus \Tkp$ and some $K^\prime \in \Tkp$, then the trace contribution to the jump from $K^\prime$ must vanish.
Therefore, after a counting argument, we can apply the trace inequality, which is applicable since $(v-v_k)|_K\in H^2(K)$ for all $K\in\Tp$, and the bound~\eqref{eq:CT_quasi_interpolant_bound} to obtain
\[
\begin{aligned}
\int_{\Fp}h_+^{-3}\abs{\jump{v-v_k}}^2 &= \int_{\Fp \setminus \Fks} h_+^{-3} \jump{v-v_k}^2  \lesssim \sum_{K\in\Tp\setminus \Tk^+} \int_{\p K} h_+^{-3} \abs{v-v_k}^2 \\
&\lesssim \int_{\Tp\setminus \Tk^+ }h_+^{-2}\abs{\Npw (v-v_k)}^2 + h_+^{-4}\abs{v-v_k}^2 \tends 0 \quad\text{as }k\tends \infty
\end{aligned}
\]
where convergence follows from~\eqref{eq:finite_approx2} which was already shown above. A similar argument restricted to interior faces can be applied to the jumps of gradients, thus yielding~\eqref{eq:finite_approx_jump_convergence}.

\emph{Step~4. Proof of $v_k\in \HD$ and of \eqref{eq:finite_approx_1}.}
Since the piecewise gradient and Hessian coincide with the classical gradient and Hessian on each $K\in\Tp$, the bound~\eqref{eq:CT_quasi_interpolant_bound} and a counting argument implies that $\int_{\Tp}\abs{v_k}^2+\abs{\Dp{v_k}}^2+\abs{\Hessp{v_k}}^2\lesssim \norm{v}_{\HD}^2<\infty$. Furthermore, we get $\int_{\Fpi}h_+^{-1}\jump{\Dp{v_k}}^2+\int_{\Fp}h_+^{-3}\jump{v_k}^2 < \infty$ since $\jump{v_k}=0$ and $\jump{\Dp{v_k}}=0$ except for at most finitely many faces.
After extending $v_k$ by zero to $\R^\dim$, the distributional derivative of $v_k$ satisfies $
\pair{ Dv_k,\bphi}_{\R^\dim} = \langle Dv,\bphi \rangle_{\R^\dim} + \langle D(v_k-v),\bphi \rangle_{\R^\dim}$ for all $\bphi\in \CidR$.
Note that $v-v_k$ is nonvanishing only on $\Tp\setminus \Tkp\subset \Tp$ and $v$ and $v_k$ are both in $H^2(K)$ for each $K\in\Tp$.
For each $\ell\in\N$, let $\cF^{\star}_\ell$ denote the set of faces of all elements in $\cT_{\ell}^+$ that are not in $\Flp$; note that any element of $\cT_{\ell}^{+}$ containing a face in $\cF^{\star}_\ell$ is necessarily in $\cT_{\ell}^{+}\setminus \cT_{\ell}^{1+}$.
For shorthand, for each $F\in \cF^{\star}_{\ell}$, let $\tau^{\ell}_F$ denote the trace operator from the side of $\Om_{\ell}^+$, and note that $\tau^{\ell}_F=\tau_F^{\pm}$ depending on the orientation of $\bn_F$.
Then, using elementwise integration by parts, we find that
\begin{equation}\label{eq:finite_approx_4}
\begin{split}
&\langle D(v_k-v),\bphi \rangle_{\R^\dim} = - \int_{\Omp}(v_k-v)\Div \bphi = - \lim_{\ell\tends \infty} \int_{\Omp_{\ell}} (v_k-v)\Div \bphi 
 \\ &= \lim_{\ell\tends \infty} \left(\int_{\Omp_{\ell}} \Npw(v_k-v) \cdot \bphi - \int_{\Flp} \jump{v_k-v}(\bphi\cdot \bn) - \int_{\cF^{\star}_\ell} \tau^\ell_F(v_k-v)(\bphi\cdot \bn) \right) 
 \\ & = \int_{\Omp} \Npw(v_k-v) \cdot \bphi - \int_{\Fp} \jump{v_k-v}(\bphi\cdot \bn),
 \end{split}
\end{equation}
where in passing from the second to the third lines, we have used the convergence as $\ell\tends \infty$ of the first two terms in the second line, which follows from finiteness of $\int_{\Omp}\abs{\nabla (v_k-v)}^2+\int_{\Fp}h_+^{-1}\abs{\jump{v_k-v}}^2<\infty$, and we have used the fact that the remainder term $\int_{\cF^{\star}_\ell}\tau_F^{\ell} (v_k-v)(\bphi\cdot \bn)\tends 0$ as $\ell\tends \infty$ as a result of the Cauchy--Schwarz inequality, the trace inequality and the bound
\[
\lim_{\ell\tends \infty}\int_{\cF^{\star}_\ell}\abs{\tau^{\ell}_F(v_k-v)}\lesssim \lim_{\ell\tends \infty} \left(\int_{\Tp\setminus \cT_{\ell}^{1+}} \left[\abs{\nabla(v-v_k)}^2+h_+^{-2}\abs{v-v_k}^2 \right]  \right)^{\frac{1}{2}} =0,
\]
which crucially uses the finiteness $\int_{\Tp}h_+^{-2}\abs{v-v_k}^2<\infty$ as a result of~\eqref{eq:finite_approx2}. 
Hence, by addition and subtraction, we use~\eqref{eq:distderiv_H100Tp} for $\pair{Dv,\bphi}_{\R^\dim}$ and~\eqref{eq:finite_approx_4} to obtain
\[
\begin{aligned}
\langle Dv_k,\bphi\rangle_{\R^\dim} = \int_{\Omp} \Npw v_k \cdot \bphi + \int_{\Omm}\Npw v \cdot \bphi - \int_{\Fp} \jump{v_k}\bphi\cdot \mathbf{n} &&& \forall \bphi \in \CidR,
\end{aligned}
\]
which shows that $v_k$ satisfies~\eqref{eq:distderiv_H100Tp} and also that $\Npw v_k = \Npw v$ on $\Omm$. Therefore $v_k\in H^1_D(\Om;\Tp)$ for each $k\in \N$.
The same argument as above can now be applied to each of the components of $\Npw v_k$, since $\Npw v_k = \Npw v$ on $\Omp_k\cup \Omm$ and since $\int_{\Tp}h_+^{-2}\abs{\Npw(v_k-v)}^2<\infty$ for all $k\in\N$ by~\eqref{eq:finite_approx2}. This yields
\begin{equation}\label{eq:finite_approx_5}
\langle D(\Dp{v_k}),\bvphi\rangle_{\Om} = \int_{\Omp} \Hessp v_k:\bvphi +  \int_{\Omm} \Hessp v:\bvphi- \int_{\Fpi} \jump{\Dp{v_k}}\cdot(\bvphi\mathbf{n}),
\end{equation}
for all $\bvphi\in \Cidd$, thus showing that $v_k$ satisfies~\eqref{eq:distderiv_h2}, that $\Dpw v_k$ equals the piecewise Hessian of $v_k$ over the elements $\Tp$ and that $\Dpw v_k = \Dpw v$ on $\Omp_k \cup \Omm$.
These identities along with the bounds in~\eqref{eq:finite_approx_1},~\eqref{eq:finite_approx2}, \eqref{eq:finite_approx_volume_terms}, and \eqref{eq:finite_approx_jump_convergence} show that $\norm{v_k}_{\HD}<\infty$ and thus $v_k \in \HD$ for each $k\in\N$, and that $\norm{v-v_k}_{\HD}\tends 0$ as $k\tends \infty$.
\Fqed\end{proof}

\begin{remark}[HCT element for $\dim=3$]\label{rem:3dHCT}
In the case $\dim=3$, we consider the generalization of the HCT element due to Worsey and Farin~\cite{WorseyFarin87}, which is based on the subdivision of each tetrahedra into twelve sub-tetrahedra following the incentre splitting algorithm detailed in \cite[p.~108]{WorseyFarin87}. Possible degrees of freedom on a reference element are depicted in Figure~\ref{fig:HCT_elements}, although we note that these elements are not affine equivalent.
A set of degrees of freedom on each physical element that extends the two-dimensional case includes the function value and gradient value at every vertex, and, for each edge, the orthogonal projection of the gradient values at edge midpoints into the plane orthogonal to the edge.
The generalization of~\eqref{eq:CT_dofs} on the averaging of the degrees of freedom is then similar to that in~\cite{NeilanWu19}.
\end{remark}

\emph{Proof of Theorem~\ref{thm:H1tp_finite_jumps}.} The proof of Theorem~\ref{thm:H1tp_finite_jumps} is similar to the proof just given for Theorem~\ref{thm:finite_approx}, where the only main difference is that the quasi-interpolation operator used in Theorem~\ref{thm:finite_approx} is replaced by a nodal quasi-interpolant into piecewise polynomials over $\Tp\setminus\Tkp$ that enforces $C^0$-continuity on all but finitely many faces of $\Fp$ (for instance, it is enough to consider piecewise affine approximations).
If $v\in H^1_D(\Om;\Tp)$ then the quasi-interpolant also enforces a homogeneous Dirichlet boundary condition on $\p\Om$, whereas this is not needed for functions $v\in H^1(\Om;\Tp)$.
We leave the remaining details of the proof to the reader. \qed

\begin{corollary}[Symmetry of the Hessian]\label{cor:H2_omm_restriction}
For every $v\in\HD$ there exists a $w\in H^2(\Om)\cap H^1_0(\Om)$ such that $v=w$, $\Npw v=\nabla w$, and $\Dpw v=\Dpw  w$ a.e.\ on~$\Omm$.
Furthermore,  $\Dpw v$ is symmetric a.e.\ on $\Om$ for all $v\in\HD$.
\end{corollary}
\begin{proof}
If $\Omm$ is empty then there is nothing to show, since $\Dpw v$ is symmetric a.e. on $\Omp$ as shown in Remark~\ref{rem:symmetry}. Therefore, we consider the case where $\Omm$ is nonempty. The proof follows the same path as the proof of~Corollary~\ref{cor:H1_omm_restriction}: choose $k\in \N$ and let $v_k \in \HD$ be given by Theorem~\ref{thm:finite_approx}. Then, by Lemma~\ref{lem:eventuallyinT+}, there exists $m=m(k)$ such that $v_k$ has possible nonzero jumps only on $\cF^{\dagger}_m$, i.e.\ $\jump{v_k}_F=0$ for every face $F\in \Fp\setminus \cF^{\dagger}_m$ and $\jump{\Npw v_k}_F=0$ for every $F\in \Fpi \setminus \cF^{I\dagger}_m$. Then, as shown in the proof of Corollary~\ref{cor:H1_omm_restriction}, there exists $\eta \in C^\infty_0(\R^\dim)$ such that $\eta|_{\Omm}=1$ and $\eta|_{\Omp_m}=0$. Then, define $w(x) \coloneqq \eta(x) v_k(x) $ for all $x\in \R^\dim$, where we recall that $v_k$ is extended by zero outside of $\overline{\Om}$. The same arguments in the proof of Corollary~\ref{cor:H1_omm_restriction} imply that $w\in H^1_0(\Om)$ and that $\nabla w = \eta \Npw v_k + v_k \nabla \eta$, and moreover that $\nabla w=\Npw v$ a.e.\ on $\Omm$.
We now show that also $w\in H^2(\Om)$ so that $w\in H^2(\Om)\cap H^1_0(\Om)$. Considering an arbitrary $\bvphi\in \Cidd$, a straightforward calculation using the known distributional derivatives of $v_k$ and $\Npw v_k$ shows that
\[
\begin{aligned}
&\langle D(\nabla w),\bvphi\rangle = -\int_\Om\nabla w\cdot(\Div\bvphi) = -\int_\Om(\eta\Dp{v_k}+v_k\nabla\eta)\cdot(\Div\bvphi)\\
& = -\int_\Om \left[\Npw v_k {\cdot} \Div(\eta\bvphi) - \left(\Npw v_k{\otimes}\Npw \eta\right) {:} \bvphi  + v_k \Div(\nabla \eta^\top \bvphi) - v_k \nabla^2\eta {: }\bvphi \right]\\
& = \int_{\Om} \left[ \eta \nabla^2 v_k + \Npw v_k \otimes \nabla \eta + \Npw \eta \otimes \Npw v_k + v_k \Dpw \eta \right]: \bvphi
\end{aligned}
\]
where the last equality above follows from~\eqref{eq:distderiv_h2} and~\eqref{eq:distderiv_H1Tp}, where it is noted that all terms involving jumps vanish owing to the facts that $\bvphi$ vanishes on $\partial\Om$, that $\eta$ vanishes on every face $F\in \cF^{\dagger}_m$, and the fact that $v_k$ and $\Npw v_k$ have possible nonzero jumps only on $\cF^{\dagger}_m$ as explained above.
Thus, $\Dpw w= \eta \nabla^2 v_k + \Npw v_k \otimes \nabla \eta + \Npw \eta \otimes \Npw v_k + v_k \Dpw \eta$ and $w\in H^2(\Om) \cap H^1_0(\Om)$.
 Furthermore, we see that $\nabla^2 w = \Dpw v_k = \Dpw v $ a.e.\ in $\Omm$. Since $\Dpw w$ is symmetric owing to $w\in H^2(\Om)$, we see that $\Dpw v$ is symmetric a.e.\ in $\Omm$. Since $\Dpw v$ is also symmetric a.e.\ on $\Omp$, as shown in Remark~\ref{rem:symmetry}, we conclude that $\Dpw v$ is symmetric a.e.\ on~$\Om$. 
\Fqed\end{proof}

Since Corollary~\ref{cor:H2_omm_restriction} shows that functions $v\in \HD$ have symmetric Hessians, we may now write $\Dpw_{x_ix_j}v\coloneqq (\Dpw v)_{ij}$, with symmetry giving $\Dpw_{x_ix_j}v=\Dpw_{x_jx_i} v$ for all $i,\,j\in\{1,\dots,\dim\}$.  
The symmetry of the Hessians of functions in $\HD$ shown in Corollary~\ref{cor:H2_omm_restriction} crucially allows for the construction of good polynomial approximations over the meshes, including over elements that are eventually refined. Recall that the set $\Fp_\circ(K)$, for any element $K$, is defined in~\eqref{eq:fp_circ}.
\begin{lemma}[Approximation by quadratic polynomials]\label{lem:p2_element_approx}
For every function $v\in \HD$, and every $K\in\Tk$, $k\in\N$, we have
\begin{multline}\label{eq:p2_element_approx}
\inf_{\hat{v}\in \mathbb{P}_2}\sum_{m=0}^2 \int_K h_K^{2m-4} \abs{\nabla^m(v-\hat{v})}^2 \\ \lesssim \int_K \abs{\Dpw v-\overline{\Dpw v|_K}}^2 + \int_{\Fp_\circ(K)} \left[ h_+^{-1} \abs{\jump{\Npw v}}^2+h_K^{-2}h_+^{-1}\abs{\jump{v}}^2\right],
\end{multline}
where $\mathbb{P}_2$ denotes the space of quadratic polynomials, and where $\overline{\Dpw v|_K} \in \R^{\dim\times\dim}$ denotes the component-wise mean-value of $\Dpw v$ over $K$, i.e.\ $\bigr[\overline{\Dpw v|_K}\bigr]_{ij}=\frac{1}{\abs{K}}\int_K \Dpw_{x_i x_j} v$ for all $i,\,j\in\{1,\dots,\dim\}$. 
\end{lemma}
\begin{proof}
We construct a polynomial $\hat{v}\in \mathbb{P}_2(K)$ such that 
\begin{equation}\label{eq:p2_element_approx_1}
\int_K (v-\hat{v}) = \int_K \Npw_{x_i} (v-\hat{v}) = \int_K  \Dpw_{x_ix_j} (v-\hat{v}) =0, \quad \forall i,\,j \in \{1,\dots,\dim\},
\end{equation}
which implies that $\Dpw \hat{v} = \overline{\Dpw v|_K}$ since $\hat{v}$ is a quadratic polynomial.
For shorthand, let $\bm{H}\coloneqq \overline{\Dpw v|_K} \in \R^{\dim\times\dim}$, and note that $\bm{H}$ is symmetric owing to the symmetry of $\Dpw v$ as shown by Corollary~\ref{cor:H2_omm_restriction}.
Then, define the vector $\bm{d}\in \R^\dim$ by $\bm{d}=\frac{1}{\abs{K}}\int_K\left[\Npw v - \bm{H} x\right]\mathrm{d} x$, where the integral is taken component-wise, and let the constant $a$ be defined by $a \coloneqq \frac{1}{\abs{K}}\int_K \left[v- \bm{d}\cdot x - \frac{1}{2} x^\top \bm{H} x \right]\mathrm{d}x  $. We claim that $\hat{v}(x)\coloneqq a + \bm{d}\cdot x + \frac{1}{2} x^\top \bm{H} x$ satisfies~\eqref{eq:p2_element_approx_1}. 
First, it is clear that $\int_K (v-\hat{v})=0$ owing to the definition of the constant $a$. Next, the symmetry of $\bm{H}$ implies that $\Npw \hat{v}(x) = \bm{d}+\frac{1}{2}(\bm{H}+\bm{H}^\top)x= \bm{d}+ \bm{H} x$ for all $x\in K$, so by definition of the vector $\bm{d}$ we get $\int_K\Npw_{x_i}(v-\hat{v})=0$. 
Finally, we have $\Dpw \hat{v}=\bm{H}=\overline{\Dpw v|_K}$, so~\eqref{eq:p2_element_approx_1} is verified.
To obtain~\eqref{eq:p2_element_approx}, it remains only to apply the Poincar\'e inequality of Theorem~\ref{thm:poincare} to $v-\hat{v}$ and each component of its gradient. First, the application of the Poincar\'e inequality to each component $\Npw_{x_i}(v-\hat{v}) \in H^1(\Om;\Tp)$, for each $i\in\{1,\dots,\dim\}$, followed by a summation over the components, gives
\begin{equation}\label{eq:p2_element_approx_2}
h_K^{-2}\int_{K} \abs{\Npw (v-\hat{v})}^2 \lesssim \int_{K} \abs{\Dpw(v-\hat{v})}^2 + \int_{\Fp_{\circ}(K)} h_{+}^{-1}\abs{\jump{\Npw v}}^2, 
\end{equation}
where we have used the fact that each component of $\Npw(v-\hat{v})$ has zero mean-value on $K$ from~\eqref{eq:p2_element_approx_1}, and where we have simplified the jumps $\jump{\Npw(v-\hat{v})}=\jump{\Npw v}$ since $\hat{v}$ is a polynomial.
Next, the Poincar\'e inequality applied to $v-\hat{v}$, which also has vanishing mean-value over $K$ by~\eqref{eq:p2_element_approx_1}, also implies
\begin{equation}\label{eq:p2_element_approx_3}
h_K^{-4}\int_{K} \abs{v-\hat{v}}^2 \lesssim h_K^{-2}\int_{K} \abs{\Npw (v-\hat{v})}^2 + \int_{\Fp_{\circ}(K)} h_K^{-2} h_+^{-1}\abs{\jump{v}}^2.
\end{equation}
We then obtain \eqref{eq:p2_element_approx} from the combinations of~\eqref{eq:p2_element_approx_2} with \eqref{eq:p2_element_approx_3}.
\Fqed\end{proof}

\subsection{Limit spaces of finite element functions}

We now introduce the limit spaces of the finite element spaces, which consist of functions in $\HD$ that are piecewise polynomials of degree at most $p$ over~$\Tp$.
Recall that the norm and inner-product of the space $\HD$ are defined in~\eqref{eq:HD_norm_def} and~\eqref{eq:HD_innerprod_def} respectively.

\begin{definition}[Limit spaces]\label{def:limit_space}
Let $\Vinfty^0$ and $\Vinfty^1$ be defined by
\begin{equation}\label{eq:limit_space_def}
\Vinfty^0 \coloneqq \{ v\in \HD\colon v|_K \in \mathbb{P}_p \; \forall K \in \Tp\}, \quad \Vinfty^1 \coloneqq \Vinfty^0 \cap H^1_0(\Om). 
\end{equation}
The spaces $\Vinfty^0$ and $\Vinfty^1$ are equipped with the same inner-product and norm as $\HD$.
\end{definition}

It follows that $\Vinfty^1$ is a closed subspace of $\Vinfty^0$ and that $\Vinfty^0$ is a closed subspace of $\HD$. Therefore the spaces $\Vinfty^s$, $s\in\{0,1\}$, are Hilbert spaces under the same inner-product as $\HD$, see Theorem~\ref{thm:completeness_H2}.
The following Theorem shows that functions in the spaces $\Vinfty^s$ can be approximated by sequences of functions from the corresponding finite element spaces, thereby justifying the choice of notation.

\begin{remark}[Extension of $\norm{\cdot}_k$ to $\HD+\Vk^s$]
The trace inequality of Lemma~\ref{thm:trace} implies that  any function $v\in\HD$ has square-integrable jumps $\jump{v}$ and $\jump{\Npw v}$ over $\Fk$ for each $k\in\N$. 
Hence, the  norm $\normk{v}$ is finite for any $v\in\HD$ and any $k\in\N$.
We may thus extend the norms $\normk{\cdot}$ to the sum space $\HD + \Vk^s$ for all $s\in\{0,1\}$ and all $k\in\N$.
\end{remark}

\begin{theorem}[Approximation by finite element functions]\label{thm:limit_space_characterization}
Let $s\in\{0,1\}$. Then, for any $v\in \Vinfty^s$, there exists a sequence of finite element functions $v_k\in \Vk^s$ for each $k\in\N$, such that
\begin{equation}\label{eq:limit_space_characterization}
\lim_{k\tends \infty} \norm{v-v_k}_k =0, \quad \sup_{k\in\N} \norm{v_k}_k<\infty.
\end{equation}
Moreover, the sequence $\{v_k\}_{k\in\N}$ above can be chosen such that 
\begin{equation}
\lim_{k\tends\infty}\int_\Om \left[h_k^{-2}\abs{\Npw(v-v_k)}^2+h_k^{-4}\abs{v-v_k}^2\right]=0 .
\end{equation}
\end{theorem}
\begin{proof}
\emph{Step 1. Proof for $s=0$.} 
For each $k\in\N$, let $v_k \in \Vk^0$ denote the $L^2$-orthogonal projection of $v$ into $\Vk^0$. Then, since $v|_K \in \mathbb{P}_p $ for each $K\in \Tp$, it follows immediately that $(v-v_k)|_K = 0$ for each $K\in \Tkp$. This implies also that the jumps $\jump{v-v_k}$ and $\jump{\Npw(v-v_k)}$ are only possibly nonvanishing on faces with at least one parent element in $\Tkm$.
Therefore, a counting argument gives
\begin{multline}\label{eq:limit_space_1}
\norm{v-v_k}_k^2 \lesssim  \sum_{m=0}^2 \int_{\Tkm}\abs{\Npw^m(v-v_k)}^2 \\ + \sum_{K\in\Tkm} \int_{\p K}\left[ h_{k}^{-1}\abs{\tauk \nabla (v-v_k)}^2 +  h_{k}^{-3}\abs{\tauk (v-v_k)}^2\right],
\end{multline}
where it is recalled that the trace operator $\tauk$ is bounded from $H^1(\Om;\Tp)$ to $L^2(\p K)$, as shown by Theorem~\ref{thm:trace}.
Recall also that by definition $h_k|_{K^\circ}=h_K=\abs{K}^{\frac{1}{\dim}}$.
Since $v_k$ is the $L^2$-orthogonal projection of $v$ into $\Vk^0$, the stability of the $L^2$-orthogonal projection and inverse inequalities imply that $\sum_{m=0}^2\int_K h_k^{2m-4}\abs{\nabla^m(v-v_k)}^2\lesssim\inf_{\hat{v}\in\mathbb{P}_p}\sum_{m=0}^2 \int_K h_K^{2m-4}\abs{\nabla^m(v-\hat{v})}^2$ for every $K\in\Tk$, where we recall that $p\geq 2$.
Therefore the trace inequality of Theorem~\ref{thm:trace} and the approximation bound of Lemma~\ref{lem:p2_element_approx} imply that, for each $K\in \Tkm$,
\begin{multline}
\sum_{m=0}^2\int_K h_k^{2m-4}\abs{\Npw^m(v-v_k)}^2 + \int_{\p K}\left[ h_{k}^{-1}\abs{\tauk \nabla (v-v_k)}^2 +  h_{k}^{-3}\abs{\tauk (v-v_k)}^2\right]
\\ \lesssim \int_K \abs{\Dpw v - \overline{\Dpw v|_K}}^2 +  \int_{\Fp_\circ(K)} \left[ h_+^{-1} \abs{\jump{\Npw v}}^2+h_+^{-3}\abs{\jump{v}}^2\right],
\end{multline}
where we have used the inequality $h_k^{-2}h_+^{-1}\leq h_+^{-3}$ in the term for the jumps.
We now define $\pi_k^0(\Dpw v)$ the piecewise constant $L^2$-orthogonal projection of $\Dpw v$ over $\Tk$; in particular, $\pi_k^0(\Dpw v)|_K = \overline{\Dpw v|_K}$ for each $K\in\Tk$.
Next, recall that a face $F\in \Fp_\circ(K)$ if and only if $F\in \Fpi$, and that $F\subset K$ but $F\not\subset\p K$; thus $F$ cannot be in $\Fk$, and thus $F\in \Fp\setminus \Fkp$. Therefore, it follows that
\begin{multline}\label{eq:limit_space_2}
\norm{v-v_k}_k^2 + \sum_{m=0}^1 \int_{\Om}h_k^{2m-4}\abs{\Npw^m(v-v_k)}^2 \\ \lesssim \int_{\Tkm} \abs{\Dpw v -\pi_k^0(\Dpw v)}^2 + \int_{\Fpi\setminus\Fkp} \left[ h_+^{-1} \abs{\jump{\Npw v}}^2+h_+^{-3}\abs{\jump{v}}^2\right].
\end{multline}
We now show that the right-hand side of~\eqref{eq:limit_space_2} tends to $0$ as $k\tends \infty$. The terms involving the jumps above consist of the tail of a convergent series bounded by $\norm{v}_{\HD}^2$, and thus
\[
\lim_{k\tends\infty}\int_{\Fpi\setminus\Fkp} \left[ h_+^{-1} \abs{\jump{\Npw v}}^2+h_+^{-3}\abs{\jump{v}}^2\right] =0.
\]
To handle the volume terms, let $\eps>0$ be arbitrary; then, there exist smooth functions $\bvphi_{ij} \in C^\infty_0(\Om)$  such that $\norm{\Dpw_{ij} v - \bvphi_{ij}}_{\Om}<\eps$ for all $i,\,j \in \{1,\dots,\dim\}$. Therefore, recalling Lemma~\ref{lem:hjvanishes}, we get
\begin{equation*}
\begin{split}
\lim_{k\tends \infty}\int_{\Tkm}\abs{\Dpw v-\pi_k^0(\Dpw v)}^2
& \lesssim \sum_{i,j=1}^\dim \left[\norm{\Npw_{ij}^2 v- \bvphi_{ij}}_{\Om}^2+ \lim_{k\tends \infty}\norm{\bvphi_{ij}-\pi_k^0(\bvphi_{ij})}_{\Om_k^{-}}^2 \right] \\
& \lesssim d^2 \eps^2 +  \lim_{k\tends \infty}\sum_{i,j=1}^\dim     \norm{h_k \nabla \bvphi_{ij}}_{\Om_k^{-}}^2  \leq d^2 \eps^2,
\end{split}
\end{equation*}
where, in the first inequality, we have used the stability of the $L^2$-orthogonal projection to bound $\norm{w_{ij}-\pi_k^0(w_{ij})}_\Om\leq \norm{w_{ij}}_\Om$, with $w_{ij}=\Npw_{ij} v-\bvphi_{ij}$, and where, in the second inequality, we note that $\norm{h_k \nabla \bvphi_{ij}}_{\Om_k^{-}}\tends 0$ in the limit owing to Lemma~\ref{lem:hjvanishes}.
Since $\eps$ is arbitrary, we conclude that $\int_{\Tkm}\abs{\Dpw v-\pi_k^0(\Dpw v)}^2\tends 0$ as $k\tends \infty$, which completes the proof that the right-hand side of \eqref{eq:limit_space_2} vanishes in the limit;  from this we then infer that
\[
\begin{aligned}
\lim_{k\tends\infty}\norm{v-v_k}_k=0, &&&
\lim_{k\tends\infty}\int_\Om \left[h_k^{-2}\abs{\Npw(v-v_k)}^2+h_k^{-4}\abs{v-v_k}^2\right]=0.
\end{aligned}
\] 
Combining the triangle inequality with the bounds obtained above then shows that $\sup_{k\in\N}\norm{v_k}_k\lesssim \norm{v}_{\HD}$ and thus completes the proof of~\eqref{eq:limit_space_characterization} for $s=0$.

\emph{Step 2. Proof for $s=1$.} Now let $s=1$ and consider $v\in \Vinfty^1$. Let $w_k \in \Vk^0$ be defined as the $L^2$-orthogonal projections of $v$ into $\Vk$ for each $k\in\N$. Note that we are now relabelling the sequence of approximations used in \emph{Step~1} above.
Since $\Vinfty^1\subset \Vinfty^0$, it follows from the arguments of \emph{Step~1} that $\norm{v-w_k}_{k}\tends 0$ and $\int_\Om h_k^{2m-4}\abs{\Npw^m(v-w_k)}^2\tends 0$ as $k\tends \infty$, for each $m\in\{0,1\}$.
Now let $v_k\coloneqq E_k^1 w_k$ where $E_k^1\colon \Vk^0\tends\Vk^1$ is the $H^1_0$-conforming enrichment operator based on local averaging of degrees of freedom as in~\cite{KarakashianPascal03}. Adapting the analysis therein to the present setting, we obtain the bounds
\begin{equation}\label{eq:limit_space_3}
\sum_{m=0}^2 \int_\Om h_k^{2m-4}\abs{\Npw^m(w_k-v_k)}^2 \lesssim \int_{\Fk} h_k^{-3}\abs{\jump{w_k}}^2=\int_{\Fk} h_k^{-3}\abs{\jump{v-w_k}}^2\leq \norm{v-w_k}_k^2,
\end{equation}
where we have used the fact that now $v\in \Vinfty^1 \subset H^1_0(\Om)$ and hence $\jump{w_k}=\jump{w_k-v}$ for all faces of $\Fk$. Furthermore, the bound~\eqref{eq:limit_space_3}, the triangle inequality and the trace inequality imply that
\begin{multline}\label{eq:limit_space_4}
\int_{\Fki}h_k^{-1}\abs{\jump{\Npw (v-v_k)}}^2 \lesssim \int_{\Fki}h_k^{-1}\abs{\jump{\Npw(v-w_k)}}^2 +\int_{\Fki}h_k^{-1}\abs{\jump{\Npw(w_k-v_k)}}^2
\\ \lesssim \norm{v-w_k}_k^2 +\sum_{m=1}^2 \int_\Om h_k^{2m-4}\abs{\Npw^m(w_k-v_k)}^2 \lesssim \normk{v-w_k}^2.
\end{multline}
So, after applying the triangle inequality and combining the bounds~\eqref{eq:limit_space_3} and \eqref{eq:limit_space_4}, we get
\begin{equation}
\norm{v-v_k}_k^2 + \sum_{m=0}^1 \int_\Om h_k^{2m-4}\abs{\Npw^m(v-v_k)}^2 \lesssim \norm{v-w_k}_k^2 + \sum_{m=0}^1 \int_\Om h_k^{2m-4}\abs{\Npw^m(v-w_k)}^2,
\end{equation}
and we note that the right-hand side above tends to $0$ as $k\tends \infty$.
Hence if $v\in \Vinfty^1$, then the claim of the Theorem is also satisfied for a sequence of functions $v_k\in\Vk^1$ for all $k\in\N$.
\Fqed\end{proof}

\begin{remark}
Theorem~\ref{thm:limit_space_characterization} shows that functions in $\Vinfty^s$ are limits in the sense of \eqref{eq:limit_space_characterization} of functions from the finite element spaces $\Vk^s$, thereby justifying the choice of notation for the limit spaces. 
Furthermore,~Theorem~\ref{thm:limit_space_characterization} establishes the connection between our approach and the approach in \cite{DominincusGaspozKreuzer19,KreuzerGeorgoulis18} where the limit spaces are defined in terms of the existence of an approximating sequence from the finite element spaces. 
\end{remark}

\begin{corollary}[Limits of norms and jumps]\label{cor:jump_term_limits}
For any $v\in\Vinfty^s$, $s\in\{0,1\}$, the sequence $\{\normk{v}\}_{k\in\N}$ is a monotone increasing sequence that converges to $\norm{v}_{\HD}$ as $k \tends \infty$, and
\begin{equation}\label{eq:jump_term_vanish}
\lim_{k\tends \infty}\int_{\Fki\setminus\Fksi}h_k^{-1}\abs{\jump{\Npw v}}^2+\int_{\Fk\setminus\Fks}h_k^{-3}\abs{\jump{v}}^2=0.
\end{equation}
The limit in \eqref{eq:jump_term_vanish} also holds with the sets $\Fks$ and $\Fksi$ replaced by $\Fkp$ and $\Fkip$, respectively.
\end{corollary}
\begin{proof}
The proof follows the same lines as \cite{DominincusGaspozKreuzer19,KreuzerGeorgoulis18}, and we include the proof only for completeness. For $v\in\Vinfty^s$, let $v_k\in\Vk^s$ denote the sequence of functions given by Theorem~\ref{thm:limit_space_characterization}.  We infer the uniformly boundedness of $\{\normk{v}\}_{k\in\N}$ from the convergence $\normk{v-v_k}\tends 0$ as $k\tends \infty$ and the uniform boundedness $\sup_{k\in\N}\normk{v_k}<\infty$. Moreover the sequence $\normk{v}$ is monotone increasing since $h_k^{-1}\leq h_m^{-1}$ for all $m\geq k$, and thus convergences to a limit.
We claim that $\int_{\Fki}h_k^{-1}\abs{\jump{\Npw v}}^2 \tends \int_{\Fpi}h_+^{-1}\abs{\jump{v}}^2$ and that $\int_{\Fk}h_k^{-3}\abs{\jump{ v}}^2\tends \int_{\Fp}h_+^{-3}\abs{\jump{ v}}^2$.
For any $\eps>0$, there is an $\ell\in\N$ such that $\abs{\norm{v}^2_m-\norm{v}^2_k}<\eps$ for all $m,k\geq \ell$.
Moreover, Lemma~\ref{lem:face_refinements} shows that there is an $M=M(k)$ such that for all $m\geq M$, then $\Fkp=\Fk\cap \cF_m$ which implies also that $\Fkip=\Fki\cap\cF^I_m$, hence
\begin{multline*}
\eps > \int_{\cF^I_m\setminus(\Fkip)} h_{m}^{-1} \abs{\jump{\Npw v}}^2 - \int_{\Fki\setminus(\Fkip)}h_{k}^{-1}\abs{\jump{\Npw v}}^2 
\\ + \int_{\cF_m\setminus(\Fkp)} h_{m}^{-3} \abs{\jump{v}}^2 - \int_{\Fk\setminus(\Fkp)} h_{k}^{-3} \abs{\jump{v}}^2
\\ \gtrsim \int_{\Fki\setminus(\Fkip)}h_{k}^{-1}\abs{\jump{\Npw v}}^2 + \int_{\Fk\setminus(\Fkp)} h_{k}^{-3} \abs{\jump{v}}^2,
\end{multline*}
where in the second inequality we use the fact that when face is refined, the $(d-1)$-dimensional Hausdorff measure of that face decreases at least by a fixed factor strictly less than one.
Thus, we obtain $\int_{\Fki\setminus(\Fkip)}h_{k}^{-1}\abs{\jump{\Npw v}}^2 + \int_{\Fk\setminus(\Fkp)} h_{k}^{-3} \abs{\jump{v}}^2 \tends 0$ as $k\tends \infty$.
Note that $h_+|_F=h_k|_F$ for any $F\in\Fkp$, so we use $\int_{\Fpi\setminus\Fkip} h_+^{-1}\abs{\jump{\Npw v}}^2 \tends0 $ to obtain $\int_{\Fki}h_k^{-1}\abs{\jump{\Npw v}}^2 \tends \int_{\Fpi}h_+^{-1}\abs{\jump{\Npw v}}^2$.
Similarly, we find that $\int_{\Fk}h_k^{-3}\abs{\jump{ v}}^2\tends \int_{\Fp}h_+^{-3}\abs{\jump{ v}}^2$.
The above limits imply that $\int_{\Fki\setminus\Fkip} h_+^{-1}\abs{\jump{\Npw v}}^2 \tends 0$ and that $\int_{\Fk\setminus\Fkp} h_+^{-3}\abs{\jump{\Npw v}}^2\tends 0$ as $k\tends \infty$.
 Then, we obtain~\eqref{eq:jump_term_vanish} from the limits $\int_{\Fpi\setminus\Fksi} h_+^{-1}\abs{\jump{\Npw v}}^2 \tends 0$ and from $\int_{\Fp\setminus\Fks} h_+^{-3}\abs{\jump{ v}}^2 \tends 0$ as $k\tends \infty$. We finally conclude that $\normk{v}\tends \norm{v}_{\HD}$ as $k\tends\infty$ from the above limits.
\Fqed\end{proof}

\subsection{Limit lifting operators and weak compactness of bounded sequences of finite element functions}

In order to study the weak convergence properties of bounded sequences of functions from the finite element spaces, we now introduce a lifting operator defined on the limit space $\Vinfty^s$ along with corresponding lifted differential operators. Recall that for each $F\in\Fp$, there exists $\ell\in\N$ such that $F\in \Fks$ for each $k\geq \ell$, thereby implying that the operators $\br_k^F=\br_\ell^F$ for all $k\geq \ell$. We then define $\br_\infty^F \coloneqq \br_\ell^F$, and note that this is well-defined as it is independent of $\ell$. It follows that $\br_\infty^F$ maps $L^2(F;\R^\dim)$ into $\Ldd$, and moreover that the support of $\br_\infty^F(\bm{w})$ is contained in the union of all parent elements in $\Tp$ of $F$, and is thus a subset of $\Omp$. Then, for any $v\in\Vinfty^s$, define the lifted Hessian and Laplacian as
\begin{align}\label{eq:H_infty_def}
\Hinf v \coloneqq \Dpw v - \br_\infty(\jump{\Npw v}), \quad \Delta_\infty v \coloneqq \Tr(\Hinf v), \quad \br_\infty(\jump{\Npw v}) \coloneqq \sum_{F\in\Fp}\br_\infty^F(\jump{\Npw v}),
\end{align}
where we note that series defining $\br_\infty(\jump{\Npw v})$ in \eqref{eq:H_infty_def} is understood as a convergent series of functions in $\Ldd$, owing to the finite overlap of the supports of the lifting operators which implies that
\begin{equation}\label{eq:infty_lifting_boundedness}
\norm{\br_\infty(\jump{\Npw v})}_{\Om}^2\lesssim \sum_{F\in\Fp}\norm{\br_\infty^F(\jump{\Npw v})}^2_{\Om} \lesssim \int_{\Fpi} h_+^{-1}\abs{\jump{\Npw v}}^2 + \int_{\Fp}h_+^{-3}\abs{\jump{v}}^2 < \infty,
\end{equation}
for all $v\in\Vinfty^s$, where we have used an inverse inequality to bound $\jump{\Npw_T v}$, the tangential component of the jumps, on boundary faces, which is possible since $v\in \Vinfty^s$ is piecewise polynomial on $\Tp$.
Moreover, the lifting $\br_\infty(\jump{\Npw v})$ is essentially supported on $\Omp$ and its restriction $\br_\infty(\jump{\Npw v})|_K$ is a piecewise $\dim\times\dim$-matrix valued polynomial of degree at most $q$ for each $K\in\Tp$.
It is then easy to see that the operators $\Hinf$ and $\Delta_\infty$ defined in~\eqref{eq:H_infty_def} are bounded on the space $\Vinfty^s$, i.e.\
\begin{equation}\label{eq:infty_lifting_boundedness_2}
\begin{aligned}
\norm{\Hinf v}_{\Om} + \norm{\Delta_\infty v}_\Om &\lesssim \norm{v}_{\HD} &&&\forall v \in \Vinfty^s.
\end{aligned}
\end{equation}
The next lemma shows that the lifting operators defined in~\eqref{eq:H_infty_def} are the appropriate limits of the corresponding operators from~\eqref{eq:lifted_Hessian} applied to strongly convergent sequences of finite element functions.

\begin{lemma}[Convergence of lifting operators]\label{lem:lifted_Hess_convergence}
Let $\{v_k\}_{k\in\N}$ be a sequence of functions such that $v_k\in\Vk^s$ for each $k\in\N$, and suppose that there is a $v\in \Vinfty^s$ such that $\norm{v-v_k}_k\tends 0$ as $k\tends \infty$. Then~$\bm{H}_k v_k \tends \bm{H}_{\infty} v$ and $\br_k(\jump{\Npw v_k})\tends \br_{\infty}(\jump{\Npw v})$ in $\Ldd$ as $k\tends\infty$.  
\end{lemma}
\begin{proof}
Since the hypothesis of convergence in norms implies that $\Dpw v_k \tends \Dpw v$ in $\Ldd$, it is enough to show that $\br_k(\jump{\Npw v_k}) \tends \br_\infty(\jump{\Npw v})$ in $\Ldd$ as $k\tends\infty$, as convergence of $\Hk v_k$ to $\Hinf v$ then follows immediately.
By definition, the face lifting operator $\br_\infty^F=\br_k^F$ for every $F\in\Fks$.
So, the triangle inequality and the finite overlap of the supports of the lifting operators yield
\begin{multline}\label{eq:lifting_terms_bound}
\norm{\br_\infty(\jump{\Dp v})-\br_k(\jump{\Dp{ v_k}})}_\Om^2
\lesssim \int_{\Fki}h_k^{-1}|\jump{\Dp(v-v_k)}|^2+\int_{\Fpi\setminus\Fksi}h_+^{-1}|\jump{\Dp v}|^2  
\\ +\int_{\Fki\setminus\Fksi}h_k^{-1}\abs{\jump{\Dp{v}}}^2
 + \int_{\Fkb} h_k^{-3}\norm{\jump{v-v_k}}^2
 \\ +\int_{\cF^{B+}\setminus\Fkp}h_+^{-3}\abs{\jump{v}}^2+\int_{\Fkb\setminus\Fks}h_k^{-3}\abs{\jump{v}}^2.
\end{multline}
The right-hand side of \eqref{eq:lifting_terms_bound} tends to zero owing to \eqref{eq:jump_term_vanish}, to the convergence of $\norm{v-v_k}_k\tends 0$, and the vanishing tails $\int_{\Fpi\setminus\Fksi}h_+^{-1}|\jump{\Dp v}|^2+\int_{\cF^{B+}\setminus\Fkp}h_+^{-3}\abs{\jump{v}}^2\tends 0$ as $k\tends \infty$. Then $\Hk  v_k\tends \Hinf v$ follows from the convergence of $\rk\tends \br_\infty(\jump{\Npw v}) $ and $\Dpw v_k \tends \Dpw v$.
\Fqed\end{proof}

We now prove that bounded sequences of functions from the finite element spaces have appropriate weak compactness properties, and have weak limits in the limit spaces. 
Let $\chi_{\Omp}$ denote the indicator function of the set~$\Omp$.

\begin{theorem}[Weak convergence]\label{thm:weak_convergence}
Let $\{v_k\}_{k\in\N}$ be a sequence of functions such that $v_k\in \Vk^s$ for each $k\in\N$, and such that $\sup_{k\in\N}\norm{v_k}_k<\infty$.
Then, there exist a $v\in \Vinfty^s$ and a $\br\in \Ldd$ such that $\br \chi_{\Omp}=\br_\infty(\jump{\Npw v})$ a.e.\ in $\Om$, and there exists a subsequence $\{v_{k_j}\}_{j\in\N}$ such that $v_{k_j}\tends v$ in $L^2(\Om)$, $\Npw v_{k_j}\tends \Npw v \in L^2(\Om;\R^\dim)$, $\bm{H}_{k_j}v_{k_j}\rightharpoonup \bm{H}_{\infty} v$ and $\br_{k_j}(\jump{\Npw v_{k_j}}) \rightharpoonup \br$ in $\Ldd$ as $j\tends \infty$.
\end{theorem} 
\begin{proof}
Since $\Vk^1\subset \Vk^0$ for all $k\in\N$ and $\Vinfty^1\subset \Vinfty^0$, we consider the general case $s=0$, and handle the special case $s=1$ only where it is needed. 
We will also frequently use the fact that for any integer $k\geq \ell$, if a face $F\in\Fk \setminus \Flp$, then $h_k|_F \lesssim \norm{h_\ell \chi_{\Omega_{\ell}^{1-}}}_{L^\infty(\Om)}$. This is due to the fact that any element $K\in \Tk$ that contains $F$ must be included in $\Omega_{\ell}^{1-}$, for otherwise $F\in\Flp$ and there would be a contradiction.

\emph{Step 1. Compactness of values and gradients.}
The discrete Rellich--Kondrachov theorem for DG finite element spaces, see  \cite[Theorem~5.6]{DiPietroErn12}, shows that the sequence $\{v_k\}_{k\in\N}$ is relatively compact in $L^2(\Om)$, and thus, there exists a $v\in L^2(\Om)$ and a subsequence, to which we pass without change of notation, such that $v_k\tends v$ in $L^2(\Om)$ as $k\tends \infty$. Furthermore, after extending the functions $v_k$ and $v$ by zero, we further have $v_k \tends v$ in $L^2(\R^\dim)$ as $k\tends \infty$. The uniform boundedness of the sequence $\{v_k\}_{k\in\N}$ in $BV(\R^\dim)$, as shown by \cite[Lemma~5.2]{DiPietroErn12}, further implies that $v\in BV(\R^\dim)$.
Furthermore, for each $i\in\{1,\dots,\dim\}$, the sequence $\{\Npw_{x_i} v_k \}_{n\in\N}$ is uniformly bounded in both $BV(\Om)$ and in $L^{r}(\Om)$ for some $r>2$, see \cite[Lemma~2 \&~Theorem~4.1]{BuffaOrtner}, and thus, by compactness of the embedding of $BV(\Om)$ into $L^1(\Om)$, after passing to a further subsequence without change of notation, there is a $\bm{\sigma} \in L^2(\Om;\R^\dim)$ such that $\Npw v_k \tends \bm{\sigma}$ in $L^2(\Om)^{\dim}$ as $k\tends \infty$.
We also infer that the restriction $v|_K$ is a polynomial of degree at most $p$ for each $K\in\Tp$, since it is the limit of the sequence of polynomials $ \{v_k|_K\}_{k\in\N}$. Furthermore, the equivalence of norms in finite dimensional spaces and the fact that $\Npw v_k \tends \bm{\sigma}$ imply that $\bm{\sigma}|_K = \Npw v|_K$ for each $K\in\Tp$. In addition, this implies that $\jump{v_k}_F \tends \jump{v}_F$ for all $F \in \Fp$,  and $\jump{\Npw v_k}_F\tends \jump{\Npw v}_F$ for all $F\in \Fpi$, in any norm as $k\tends \infty$.

\emph{Step~2. Bounds on the jumps.}
We now prove that $\int_{\Fp} h_+^{-3}\abs{\jump{v}}^2 <\infty $ and $\int_{\Fpi} h_+^{-1} \abs{\jump{\Npw v}}^2<\infty $.
Recall that $\Sk$ and $\Sp$ denote the skeletons of the sets of faces $\Fk$ and $\Fp$ respectively.
Consider now the function $h_k^{-3}\abs{\jump{v_k}}^2\colon \Sk \tends \R$, and extend it by zero to $\Sp\setminus \Sk$.
Then, since $h_k|_F= h_+|_F$ whenever $k$ is sufficiently large for each $F\in \Fp$, we deduce that $h_k^{-3}\abs{\jump{v_k}}^2$ converges pointwise $\mathcal{H}^{d-1}$-a.e.\ to $h_+^{-3}\abs{\jump{v}}^2$ on $\Sp$.
Therefore, Fatou's Lemma implies that
\begin{equation}{\label{eq:weak_convergence_1}}
\int_{\Fp} h_+^{-3}\abs{\jump{v}}^2 = \int_{\Sp}h_+^{-3}\abs{\jump{v}}^2 \leq \liminf_{k\tends \infty}\int_{\Skp} h_k^{-3}\abs{\jump{v_k}}^2 \leq \liminf_{k\tends\infty}\norm{v_k}_k^2 < \infty .
 \end{equation}
 Similarly, $h_k^{-1}\abs{\jump{\Npw v_k}}^2$ converges $\mathcal{H}^{d-1}$-a.e.\ to $h_+^{-1}\abs{\jump{\Npw v}}^2$ on $\Spi$ and Fatou's Lemma shows that
 \begin{equation}\label{eq:weak_convergence_2}
\int_{\Fpi}h_+^{-1}\abs{\jump{\Npw v}}^2 = \int_{\Spi}h_+^{-1}\abs{\jump{\Npw v}}^2 \leq \liminf_{k\tends\infty}\int_{\Skpi}h_k^{-1}\abs{\jump{\Npw v_k}}^2 \leq \liminf_{k\tends\infty}\norm{v_k}_k^2 <\infty.
 \end{equation}

\emph{Step~3. Proof that $v\in H^1_D(\Om;\Tp)$.}
Next, we claim that $\int_{\Fk}\jump{v_k}(\bphi\cdot \bn)\tends \int_{\Fp}\jump{v}(\bphi\cdot \bn)$ as $k\tends \infty$ for any $\bphi \in \CidR$. Assuming this claim for the moment, we verify that the function~$v$ has a distributional derivative of the form~\eqref{eq:distderiv_H100Tp} where $\Npw v = \bm{\sigma}$ in $\Om$. Indeed, the convergence $v_k\tends v$ in $L^2(\R^\dim)$ implies that $\pair{Dv,\bphi}_{\R^\dim}=\lim_{k\tends\infty}\pair{Dv_k,\bphi}_{\R^\dim}$, and the convergence of the jumps and of $\Npw v_k \tends \bm{\sigma}$ in $L^2(\R^\dim;\R^\dim)$ also imply that
\begin{equation}
\begin{aligned}
\pair{Dv,\bphi}_{\R^\dim} &=\lim_{k\tends\infty} \left(\int_{\Om} \Npw v_k \cdot \bphi - \int_{\Fk} \jump{v_k}(\bphi\cdot\bn) \right)
\\ & = \int_{\Om} \bm{\sigma} \cdot \bphi-\int_{\Fp}\jump{v}(\bphi\cdot \bn) \quad \forall \bphi\in \CidR,
\end{aligned}
\end{equation}
which shows that $\Npw v = \bm{\sigma}$ and that $v\in H^1_D(\Om;\Tp)$.

Returning to the claim that $\int_{\Fk}\jump{v_k}(\bphi\cdot \bn)\tends \int_{\Fp}\jump{v}(\bphi\cdot \bn)$ as $k\tends \infty$, we choose an $\ell \in \N$ to be specified below, and for any $k\geq \ell$ we split the series according to
\begin{multline}\label{eq:weak_convergence_3}
\int_{\Fp}\jump{v}(\bphi\cdot\bn) - \int_{\Fk} \jump{v_k}(\bphi\cdot \bn) 
\\ = \int_{\Flp}\jump{v-v_k}(\bphi\cdot\bn) + \int_{\Fp\setminus\Flp}\jump{v}(\bphi\cdot\bn) - \int_{\Fk\setminus \Flp}\jump{v_k}(\bphi\cdot \bn).
\end{multline}
Note that, for any $\eps>0$, we may choose $\ell$ sufficiently large such that the second and third terms on the right-hand side of \eqref{eq:weak_convergence_3} are both bounded in absolute value by $\eps$ for any $k\geq \ell$. Indeed, for the second term this results from the fact that this represents the tail of a convergent series by~\eqref{eq:weak_convergence_1}, whereas for the third term, this follows from Lemma~\ref{lem:hjvanishes} and the bound $\left\lvert\int_{\Fk\setminus \Flp}\jump{v_k}(\bphi\cdot \bn) \right\rvert \lesssim
M \norm{h_\ell \chi_{\Omega_{\ell}^{1-}}}_{L^\infty(\Om)}\norm{\bphi}_{C(\overline{\Om};\R^\dim)} $ with $M\coloneqq \sup_{k\in\N}\norm{v_k}_k<\infty$.
Then, for any fixed $\ell\in \N$, the strong convergence of the jumps $\jump{v-v_k}$ over the finite set of faces $\Flp$ shows that the first term on the right-hand side vanishes of~\eqref{eq:weak_convergence_3} also vanishes as $k\tends \infty$. We can then choose $k$ large enough such that the left-hand side of \eqref{eq:weak_convergence_3} is bounded by, e.g., $3\eps$, and since $\eps$ is arbitrary, we see that $\int_{\Fk}\jump{v_k}(\bphi\cdot \bn)\tends \int_{\Fp}\jump{v}(\bphi\cdot \bn)$ as claimed.

\emph{Step~4. Weak convergence of Hessians and proof that $v\in \Vinfty^s$.}
The sequences of functions $\{\Dpw v_k\}_{k\in\N}$ and $\{\rk\}_{k\in\N}$ are uniformly bounded in $\Ldd$ owing to the uniform boundedness of $\{\norm{v_k}_k\}_{k\in\N}$. Therefore, there exist $\bm{M}$ and $\bm{r}$ in $\Ldd$ such that $\Dpw v_k \rightharpoonup \bm{M}$ and $\rk\rightharpoonup \bm{r}$ in $\Ldd$ as $k\tends\infty$.
Furthermore, it is easy to see that $\bm{r}|_{\Omp} = \bm{r}_\infty(\jump{\Npw v})|_{\Omp}$ since the restrictions $\rk|_K\tends \bm{r}_\infty(\jump{\Npw v})|_K$ (in any norm) for all $K\in\Tp$, owing to the strong convergence $\jump{\Npw v_k}_F\tends \jump{\Npw v}_F $ (in any norm) for all $F\in\Fpi$ shown in \emph{Step~1} above.

We now claim that the distributional derivative $D(\Npw v)$ is of the form given in~\eqref{eq:distderiv_h2} and in particular that
\begin{equation}\label{eq:weak_convergence_4}
\begin{aligned}
\pair {D(\Npw v),\bvphi}_\Om &=\lim_{k\to\infty}\left(\int_\Om \Dpw{v_k}:\bvphi-\int_{\Fki}\jump{\Dp{v_k}}{\cdot}(\bvphi\mathbf{n})\right)
\\ &=\int_\Om(\bm{M}-\bm{r}\chi_{\Om^-}):\bvphi-\int_{\Fpi}\jump{\Npw{v}}{\cdot}(\bvphi\mathbf{n}),
\end{aligned}
\end{equation}
for all $\bvphi\in \Cidd$, where $\chi_{\Omm}$ denotes the indicator function of $\Omm$.
Supposing momentarily that~\eqref{eq:weak_convergence_4} is given, by definition we get $\Dpw v= \bm{M}-\bm{r}\chi_{\Om^-} \in \Ldd$. Since $\bm{r}_\infty(\jump{\Npw v})$ vanishes on $\Omm$ and equals $\bm{r}$ on $\Omp$, we see that $\Hinf v = \Dpw v - \bm{r}_\infty(\jump{\Npw v}) = \bm{M}-\bm{r} $ is then the weak limit of the sequence $\Hk  v_k = \Dpw v_k - \rk$ in $\Ldd$.
 Furthermore, the bounds~\eqref{eq:weak_convergence_1} and \eqref{eq:weak_convergence_2} above, and the fact that $v\in L^2(\Om)$, $\Npw v \in L^2(\Om;\R^\dim)$ and $\Dpw v \in \Ldd$ together imply that $\norm{v}_{\HD}<\infty$, thus showing that $v\in \HD$. Since~$v$ is piecewise polynomial over $\Tp$, it follows that $v\in \Vinfty^0$. For the special case $s=1$, we additionally have $v\in H^1_0(\Om)$ owing to the fact that the functions $v_k$ are then uniformly bounded in $H^1_0(\Om)$, which additionally implies that $v\in \Vinfty^1$.
 
It remains only to show~\eqref{eq:weak_convergence_4}. 
Consider a fixed but arbitrary $\bvphi\in \Cidd$, and let $\bvphi_k $ be its piecewise mean-value projection on $\Tk$, i.e.\ $\bvphi_k|_K\coloneqq \overline{\bvphi|_K}$ for each $K\in\Tk$, where the mean-value is taken component-wise.
The first equality in~\eqref{eq:weak_convergence_4} follows directly from $\pair{D(\Npw v_k),\bvphi}_\Om\tends \pair{D(\Npw v),\bvphi}_\Om$ owing to the convergence $\Npw v_k \tends \Npw v$ in $\Ld$.
The limit of the jump terms in~\eqref{eq:weak_convergence_4} is determined as follows. The triangle inequality gives
\begin{multline}\label{eq:weak_convergence_7}
\left\lvert\int_{\Fpi}\jump{\Npw v}\cdot(\bvphi\bn)+\int_{\Om} \chi_{\Omm}\br:\bvphi - \int_{\Fki}\jump{\Npw v_k}\cdot(\bvphi\bn) \right\rvert
\\ \leq \left\lvert \int_{\Fpi\setminus\Flpi} \jump{\Npw v}\cdot(\bvphi \bn)\right\rvert + \left\lvert \int_{\Flpi}\jump{\Npw (v-v_k)}\cdot(\bvphi\bn) \right\rvert 
\\ +\left\lvert \int_{\Fki\setminus \Flpi}\jump{\Npw v_k}:(\bvphi\bn) - \int_\Om \chi_{\Omm}\br :\bvphi \right\rvert.
\end{multline}
We show that the terms on the right-hand side of \eqref{eq:weak_convergence_7} become vanishingly small for $k$ and~$\ell$ sufficiently large, and hence the left-hand side vanishes in the limit as $k\tends\infty$.
Let $\eps>0$ be arbitrary; it follows from~\eqref{eq:weak_convergence_2} that we can start by initially choosing $\ell$ large enough such that the first term $\abs{\int_{\Fpi\setminus\Flpi} \jump{\Npw v}\cdot(\bvphi \bn)}<\eps$.
Turning to the last term on the right-hand side of \eqref{eq:weak_convergence_7}, for each $k\geq \ell$, consider the splitting of the lifting operator $\bm{r}_k$ into contributions from faces in $\Flp$ and $\Fk\setminus\Flp$, i.e.\
 \begin{equation}\label{eq:weak_convergence_5}
 \begin{aligned}
\bm{r}_k = \bm{r}_{k,\ell}^+ + \bm{r}_{k,\ell}^{-},
&&& \bm{r}_{k,\ell}^+\coloneqq \sum_{F\in\Flp} \bm{r}_k^F, &&& \bm{r}_{k,\ell}^{-}\coloneqq \sum_{F\in\Fk\setminus\Flp} \br_k^{F}.
\end{aligned}
 \end{equation}
By definition, any face $F\in\Flp$ is a face of only elements that belong to $\cT_{\ell}^+$ and thus $\supp \br_k^{F}(\jump{\Npw v_k}) \subset \Om_{\ell}^{+}$ for all $k\geq \ell$ and all $F\in\Flp$, and thus $\br_{k,\ell}^+(\jump{\Npw v_k})$ vanishes a.e.\ on $\Om_{\ell}^{-}$.
Furthermore, since any element of $\Tk$ that contains a face belonging to $\Fk\setminus \Flp$ must be a subset of $\Om_{\ell}^{1-}$, we see that $\supp \br_{k,\ell}^{-}(\jump{\Npw v_k}) \subset \Om_{\ell}^{1-}$ for all $k\geq \ell$. Additionally, we have the uniform bounds $\norm{\rk}_{\Om}+\norm{\br_{k,\ell}^{+}(\jump{\Npw v_k})}_{\Om}\lesssim\normk{v_k}\leq M$ for all $k,\ell\in\N$, where $M\coloneqq\sup_{k\in\N}\normk{v_k}$.
The definition of the lifting operators and the supports of the terms in the splitting of~\eqref{eq:weak_convergence_5} imply that 
\begin{multline}\label{eq:weak_convergence_6}
 \int_{\Fki\setminus\Flpi}\jump{\Npw v_k}\cdot \left\{ \bvphi_k\bn \right\}+ \int_{\Fkb\setminus\Flp}\jump{\Npw_T v_k}\cdot\left\{ \bvphi_k\bn \right\} = 
\int_{\Om_{\ell}^{1-}} \br_{k,\ell}^{-}(\jump{\Npw v_k}):\bvphi_k
\\ = \int_{\Om_{\ell}^{1-}} \rk:\bvphi_k - \int_{\Om_{\ell}^{1-}\setminus\Om_{\ell}^{-}} \br_{k,\ell}^+(\jump{\Npw v_k}):\bvphi_k.
\end{multline}
Lemma~\ref{lem:hjvanishes} also shows that $\abs{\int_{\Om^{1-}_\ell\setminus\Om_{\ell}^{-}}\br_{k,\ell}^+(\jump{\Npw v_k}){:}\bvphi_k }\lesssim \abs{\Om^{1-}_{\ell}\setminus\Omm}^{\frac{1}{2}} M \norm{\bvphi}_{C(\overline{\Om};\R^{\dim\times\dim})}<\eps$ for all $k\geq \ell$ whenever $\ell$ is chosen to be sufficiently large.
We can also choose $\ell$ large enough such that $\abs{\int_{\Fki\setminus\Flpi}\jump{\Npw v_k}{\cdot}\{(\bvphi-\bvphi_k)\bn)\}}\lesssim M \norm{h_\ell\Npw \bvphi}_{\Om_{\ell}^{1-}}<\eps$ since $\norm{h_\ell \chi_{\Om_{\ell}^{1-}}}_{L^\infty(\Om)}\tends 0$ as $\ell\tends \infty$ by Lemma~\ref{lem:hjvanishes}.
Also, since $\bvphi$ is compactly supported in $\Omega$, we get $\abs{\int_{\Fkb\setminus\Flp}\jump{\Npw_T v_k}\cdot\left\{ \bvphi_k\bn \right\}}\lesssim M\norm{h_\ell \Npw \bvphi}_{\Om_{\ell}^{1-}}<\eps$ for all $k\geq \ell$ sufficiently large.
Furthermore, by weak convergence of $\rk\rightharpoonup \br$ in $\Ldd$ and by strong convergence of $\bvphi_k\chi_{\Om^{1-}_\ell}\tends \bvphi \chi_{\Omm}$ in $\Ldd$ as $k\geq \ell\tends \infty$, we can also choose $\ell$ large enough such that $\abs{ \int_{\Om^{1-}_\ell} \rk {:} \bvphi_k - \int_{\Om} \chi_{\Omm}\br{:}\bvphi  } < \eps$ for all $k\geq \ell$.
Thus, by addition and subtraction of the terms in \eqref{eq:weak_convergence_6}, we infer from the above inequalities that $\lvert \int_{\Fki\setminus \Flpi}\jump{\Npw v_k}{\cdot}(\bvphi\bn) - \int_\Om  \chi_{\Omm}\br{:}\bvphi \rvert < 4 \eps$ for all $k\geq \ell$, which bounds the last term on the right-hand side of~\eqref{eq:weak_convergence_7}.
Finally, strong convergence of $\Npw v_k|_K\tends \Npw v|_K$ (in any norm) for each element $K\in\Tp$, and the finiteness of the set of faces $\Flpi$, imply that the second term in the right-hand side of \eqref{eq:weak_convergence_7} also vanishes in the limit $k\tends \infty$, for any $\ell\in\N$.
Therefore, we conclude that the left-hand side of \eqref{eq:weak_convergence_7} vanishes in the limit $k\tends \infty$, which then gives \eqref{eq:weak_convergence_4} upon recalling that $\Dpw v_k \rightharpoonup \bm{M}$ in $\Ldd$. This completes the proof.
\Fqed\end{proof}

\begin{remark}
It is easy to construct examples of sequences of functions $\{v_k\}_{k\in\N}$ satisfying the conditions of Theorem~\ref{thm:weak_convergence} such that $\bm{r}$ is nonvanishing on~$\Omm$. This explains the appearance of the indicator functions $\chi_{\Omp}$ in the equation $\bm{r}\chi_{\Omp} = \bm{r}_{\infty}(\jump{\nabla v})$ in the statement of Theorem~\ref{thm:weak_convergence} and also the appearance of the indicator function $\chi_{\Omm}$ in~\eqref{eq:weak_convergence_4}.
\end{remark}

\section{The limit problem and proof of convergence}\label{sec:limit_problem_proof}
\subsection{The limit problem}\label{sec:asymptotic_consistency}
The convergence of the sequence of numerical solutions $\{u_k\}_{k\in \N}$ from~\eqref{eq:num_scheme} is shown by introducing a suitable notion of a limit problem on the space $\Vinfty^s$.
We start by extending the definition of the operator $\Fg$ in~\eqref{eq:Fg_def} to functions in $\HD$ by $\Fg[v] \coloneqq \inf_{\alpha\in\sA} \sup_{\beta\in\sB} \left[\gamma^{\ab}(a^{\ab}:\Dpw v - f^{\ab}) \right]$ i.e.\ we use the notion of Hessian $\Dpw v$ defined by~\eqref{eq:Hessian_notation} inside the nonlinear operator $\Fg$. 
The operator $\Fg$ is then a Lipschitz continuous mapping from $\HD$ to $L^2(\Om)$, and the inequalities~\eqref{eq:cordes_ineq1} and \eqref{eq:cordes_ineq2} extend to functions in the sum space $\HD+\Vk^s$ for each $k\in\N$.

Let the nonlinear form $A_\infty(\cdot;\cdot)\colon \Vinfty^s\times\Vinfty^s\tends \R$ be defined by
\begin{equation}\label{eq:Ainfty_def}
\begin{aligned}
A_\infty(v;w)\coloneqq \int_\Om \Fg[v]\Delta_\infty w + \theta \Sinfty(v,w) + J_\infty^{\sigma,\rho}(v,w) &&&\forall v,\,w \in \Vinfty^s,
\end{aligned}
\end{equation}
where the bilinear forms $\Sinfty\colon \Vinfty^s\times \Vinfty^s\tends \R$ and $J_\infty^{\sigma,\rho}\colon \Vinfty^s\times \Vinfty^s\tends \R$ are defined by
\begin{align}
\Sinfty(v,w)&\coloneqq \int_\Om \left[\Hinf v {:} \Hinf w - \Delta_\infty v \Delta_\infty v \right] \notag
\\ & \quad + \int_\Om \left[\Tr\br_\infty(\jump{\Npw v}) \Tr\br_\infty(\jump{\Npw w}) - \br_\infty(\jump{\Npw v}){:}\br_\infty(\jump{\Npw w})  \right],\\
J_\infty^{\sigma,\rho}(v,w) &\coloneqq \int_{\Fpi} \sigma h_+^{-1} \jump{\Npw v}\cdot\jump{\Npw w} + \int_{\Fp} \rho h_+^{-3}\jump{v} \jump{w},\label{eq:limit_jumps}
\end{align}
for all functions $v$ and $w\in\Vinfty^s$, where it is recalled that the lifting operators $\br_\infty$, lifted Hessian $\Hinf$ and Laplacian $\Delta_\infty$ are defined in~\eqref{eq:H_infty_def}. 
The definition of $\Sinfty(\cdot;\cdot)$ is motivated by the identity~\eqref{eq:lifted_Bk_form} in Lemma~\ref{lem:lifted_Bk_form}, and this will be used later in the analysis.
We emphasize that the parameter $\theta$ in \eqref{eq:Ainfty_def} and the penalty parameters~$\sigma$ and~$\rho$ appearing in~\eqref{eq:limit_jumps} are the same as the ones used in the numerical scheme in Section~\ref{sec:num_schemes}.
Using the bounds on the lifting operators in~\eqref{eq:infty_lifting_boundedness} and~\eqref{eq:infty_lifting_boundedness_2} and the extension of~\eqref{eq:cordes_ineq2} to functions in $\HD$, see above, it is then straightforward to show that the nonlinear form $A_\infty(\cdot;\cdot)$ is Lipschitz continuous on $\Vinfty^s$, i.e.\
\begin{equation}\label{eq:Ainfty_lipschitz}
\begin{aligned}
\abs{A_\infty(z;w)-A_\infty(v;w) }\lesssim \norm{z-v}_{\HD}\norm{w}_{\HD} &&& \forall z,\, v,\,w \in \Vinfty^s.
\end{aligned}
\end{equation}

The following Lemma further motivates the above definitions by showing that the nonlinear forms $A_k$ are asymptotically consistent with the limit nonlinear form $A_\infty$ with respect to strongly convergent sequences in the first argument and weakly convergent (sub)sequences in the second argument. Recall that $\chi_{\Omp}$ denotes the indicator function of the set $\Omp$.

\begin{lemma}[Asymptotic consistency]\label{lem:asymptotic_consistency}
Let $\{w_{k_j}\}_{j\in\N}$ and $\{v_{k_j}\}_{j\in\N}$ be two sequences of functions, such that $w_{k_j},v_{k_j}\in V_{k_j}^s$ for each $j\in\N$, and such that $\sup_{j\in\N}\left[\norm{w_{k_j}}_{k_j}+\norm{v_{k_j}}_{k_j}\right]<\infty$.
Suppose that there exists a $v\in\Vinfty^s$ such that $\normk{v-v_{k_j}}\tends 0$ as $j\tends \infty$.
Suppose also that there exists a $w\in\Vinfty^s$ and a $\br\in\Ldd$ such that $\br \chi_{\Omp} = \br_\infty(\jump{\Npw w})$ a.e.\ in $\Om$, and such that $v_{k_j}\tends v$ in $L^2(\Om)$, $\Npw v_{k_j}\tends \Npw v$ in $\Ld$, $\bm{H}_{k_j}w_{k_j} \rightharpoonup \Hinf w$ and $\br_{k_j}(\jump{\Npw w_{k_j}})\rightharpoonup \br $ in $\Ldd$ as $j\tends \infty$. Then
\begin{equation}\label{eq:asymptotic_consistency}
\lim_{j\tends\infty} A_{k_j}(v_{k_j};w_{k_j}) = A_\infty(v;w).
\end{equation}
\end{lemma}
 \begin{proof}
  First, note that since the lifted Laplacian is defined as the trace of the lifted Hessians, its follows immediately that $\Delta_{k_j} w_{k_j}\rightharpoonup \Delta_\infty w$ in $L^2(\Om)$ as $j\tends \infty$.
 Therefore, considering the nonlinear term in $A_k(\cdot;\cdot)$, we use the strong convergence $\Fg[v_{k_j}]\tends \Fg[v]$ in $L^2(\Om)$ and the weak convergence of the lifted Laplacians to get $\int_\Om \Fg[v_{k_j}]\Delta_{k_j}w_{k_j} \tends \int_{\Om}\Fg[v]\Delta w$ as $j\tends \infty$.
 We next show convergence of the remaining terms in the nonlinear forms $A_k(\cdot;\cdot)$ as follows.

We now turn towards the term $S_{k_j}(v_{k_j},w_{k_j})$.
Lemma~\ref{lem:lifted_Hess_convergence} shows that $\bm{H}_{k_j}v_{k_j}\tends \Hinf v$ and that $\br_{k_j}(\jump{\Npw v_{k_j}})\tends \br_\infty(\jump{\Npw v})$ in $\Ldd$ as $k\tends \infty$. 
Therefore we infer that
\[
\begin{aligned}
\int_\Om \Hinf v:\Hinf w & =\lim_{j\tends\infty}\int_\Om \bm{H}_{k_j} v_{k_j}: \bm{H}_{k_j} w_{k_j},
\\  \int_\Om \Delta_\infty w \Delta_\infty v &= \lim_{j\tends\infty}\int_\Om \Delta_{k_j} v_{k_j}\Delta_{k_j}w_{k_j}.
  \end{aligned}
\]
Next, recall that $\br_\infty(\jump{\Npw v})$ vanishes on $\Omm$ for any $v\in\Vinfty^s$, and since the weak limit $\br$ of the sequence $\br_{k_j}(\jump{\Npw w_{k_j}})$ satisfies $\br|_{\Omp}=\br_{\infty}(\jump{\Npw w})$ by hypothesis, we obtain the identities $\int_\Om \br_\infty(\jump{\Npw v}):\br=\int_{\Om} \br_\infty(\jump{\Npw v}):\br_\infty(\jump{\Npw w})$ and $\int_\Om\Tr \br_\infty(\jump{\Npw v}) \Tr \br= \int_\Om\Tr \br_\infty(\jump{\Npw v}) \Tr \br_\infty(\jump{\Npw w})$.
Therefore, we conclude that 
\begin{align*}
\int_\Om \br_\infty(\jump{\Npw v}):\br_\infty(\jump{\Npw w}) &= \lim_{j\tends\infty}\int_\Om \br_{k_j}(\jump{\Npw v_{k_j}}):\br_{k_j}(\jump{\Npw w_{k_j}}), 
\\ \int_\Om\Tr \br_\infty(\jump{\Npw v}) \Tr \br_\infty(\jump{\Npw w})&=\lim_{j\tends\infty}\int_\Om \Tr\br_{k_j}(\jump{\Npw v_{k_j}}) \Tr\br_{k_j}(\jump{\Npw w_{k_j}}).
\end{align*}
Therefore, using Lemma~\ref{lem:lifted_Bk_form} and the above limits, we obtain
\begin{equation}\label{eq:asymptotic_consistency_2}
\lim_{j\tends\infty} S_{k_j}(v_{k_j},w_{k_j}) = S_\infty(v,w).
\end{equation}
It remains only to show the convergence of the jumps $J_{k_j}^{\sigma,\rho}(v_{k_j},w_{k_j})\tends J_\infty^{\sigma,\rho}(v,w)$ as $j\tends \infty$. It follows from the strong convergence of the sequence $v_{k_j}$ to $v$ that it is enough to consider the limits of $\int_{\cF_{k_j}^I} h_{k_j}^{-1} \jump{\Npw v}\cdot\jump{\Npw w_{k_j}}$ and $\int_{\cF_{k_j}}h_{k_j}^{-3}\jump{v}\jump{w_{k_j}}$.
Let $\eps>0$ be arbitrary; then Corollary~\ref{cor:jump_term_limits} and the finiteness of $\norm{v}_{\HD}$ implies that there is a $\ell\in\N$ such that $\int_{\Fpi\setminus \Flpi} h_+^{-1}\abs{\jump{\Npw v}}^2 < \eps$ and $\int_{\Fki\setminus\Fksi}h_+^{-1} \abs{\jump{\Npw v}}^2<\eps$ for all $k\geq \ell$, so that 
\begin{equation*}
\left\lvert \int_{\cF_{k_j}^I} h_{k_j}^{-1} \jump{\Npw v}\cdot\jump{\Npw w_{k_j}} - \int_{\Flpi} h_{+}^{-1}\jump{\Npw v}\cdot \jump{\Npw w_{k_j}} \right\rvert 
 = \left\lvert\int_{\cF_{k_j}^{I}\setminus\Flpi} h_{k_j}^{-1} \jump{\Npw v}\cdot \jump{\Npw w_{k_j}}\right\rvert \leq 2 M \eps ,
\end{equation*}
where $M\coloneqq\sup_{k\in\N}\normk{w_k}$, with the inequality obtained by using the Cauchy--Schwarz inequality and the disjoint partitioning $\Fki\setminus \Flpi =  (\Fki\setminus\Fksi)\cup(\Fksi\setminus \Flpi)$ for all $k\geq \ell$.
Note that $\int_{\Flpi} h_{+}^{-1}\jump{\Npw v}\cdot \jump{\Npw w_{k_j}}$ converges to $\int_{\Flpi} h_+^{-1} \jump{\Npw v}\cdot \jump{\Npw w} $ as $j\tends\infty$ for each $\ell\in\N$ since the convergence of the piecewise polynomials $\Npw w_{k_j} \tends \Npw w$ in $\Ld$ implies that $\jump{\Npw w_{k_j}}\tends \jump{\Npw w}$ for each never-refined face $F\in\Flpi$.
Passing first to the limit $j\tends \infty$ followed by $\ell\tends \infty$, we therefore obtain $ \int_{\cF_{k_j}^I} h_{k_j}^{-1} \jump{\Npw v}\cdot\jump{\Npw w_{k_j}} \tends \int_{\Fpi} h_+^{-1}\jump{\Npw v}\cdot\jump{\Npw w}$ as $j\tends \infty$.
A similar argument shows that $\int_{\cF_{k_j}}h_{k_j}^{-3}\jump{v}\jump{w_{k_j}} \tends \int_{\Fp} h_+^{-3}\jump{v}\jump{w}$ and thus $J_{k_j}^{\sigma,\rho}(v_{k_j},w_{k_j})\tends J_\infty^{\sigma,\rho}(v,w)$ as $j\tends\infty$, thereby completing the proof.
\Fqed\end{proof}

We are now able to prove that the nonlinear form $A_\infty(\cdot;\cdot)$ is strongly monotone with the same choices of penalty parameters $\rho$ and $\sigma$ used for the numerical scheme.
\begin{lemma}\label{lem:limit_strong_monotonicity}
The nonlinear forms $A_\infty(\cdot;\cdot)$ is strongly monotone on $\Vinfty^s$, and satisfies in particular
\begin{equation}\label{eq:limit_strong_monotonicity}
\begin{aligned}
\frac{1}{\cmon} \norm{w-v}_{\HD}^2 \leq A_\infty(w;w-v)-A_\infty(v;w-v) &&&\forall v,\,w \in \Vinfty^s,
\end{aligned}
\end{equation}
where the constant $\cmon>0$ is the same as in~\eqref{eq:strong_monotonicity}.
\end{lemma}
\begin{proof}
Theorem~\ref{thm:limit_space_characterization} show that for any $v$ and $w \in \Vinfty^s$, there exist sequences of functions $\{v_k\}_{k\in\N}$ and $\{w_k\}_{k\in\N}$ such that $v_k$, $w_k\in\Vk^s$ for all $k\in\N$, and such that $\normk{v-v_k}+\normk{w-w_k} \tends 0$ as $k\tends \infty$. 
Then, Lemma~\ref{lem:lifted_Hess_convergence} on the convergence of the lifting operators implies that the sequences of functions $\{v_k\}_{k\in\N}$, $\{w_k\}_{k\in\N}$ and $\{w_k-v_k\}_{k\in\N}$ satisfy the hypotheses of Lemma~\ref{lem:asymptotic_consistency}.
Therefore, we infer that
\begin{equation}\label{eq:limit_strong_monotonicity_1}
\begin{split}
\frac{1}{\cmon}\norm{w-v}_{\HD}^2&=\lim_{k\tends\infty} \frac{1}{\cmon}\normk{w_k-v_k}^2 
\\ &\leq \lim_{k\tends\infty}\left(A_k(w_k,w_k-v_k)-A_k(v_k,w_k-v_k) \right) 
\\ &= A_\infty(w;w-v)-A_\infty(v;w-v).
\end{split}
\end{equation}
where we have used Corollary~\ref{cor:jump_term_limits} for the first equality, followed by the strong monotonicity bound~\eqref{eq:strong_monotonicity}, and then an application of the asymptotic consistency shown by Lemma~\ref{lem:asymptotic_consistency}.
\Fqed\end{proof}

\paragraph{Limit problem.}
We recall that the nonlinear form $A_\infty \colon \Vinfty^s\times\Vinfty^s\tends \R$ defined in~\eqref{eq:Ainfty_def} is Lipschitz continuous, see~\eqref{eq:Ainfty_lipschitz}, and is furthermore strongly monotone as shown by Lemma~\ref{lem:limit_strong_monotonicity}. Recall also that $\Vinfty^s$ is a Hilbert space since it is a closed subspace of the Hilbert space $\HD$, see~Theorem~\ref{thm:completeness_H2}.
The Browder--Minty theorem can then be applied to deduce that there exists a unique solution $u_\infty \in \Vinfty^s$ of the limit problem
\begin{equation}\label{eq:u_infty_def}
\begin{aligned}
A_\infty(u_\infty;v) = 0 &&&\forall v \in\Vinfty^s.
\end{aligned}
\end{equation}
where it is noted that $u_\infty$ depends on $s$, but this is not reflected in the notation as there is no risk of confusion.

\subsection{Convergence of the numerical solutions and proof of Theorem~\ref{thm:main}}\label{sec:main_proof}
Our present goal is to show that the numerical approximations $u_k$ converge to $u_\infty$ and that $u_\infty=u $ the solution of~\eqref{eq:isaacs}, thereby proving~Theorem~\ref{thm:main}.
The following Lemma provides the first step by proving the convergence of the discrete solutions of the numerical scheme~\eqref{eq:num_scheme} to the solution of the limit problem~\eqref{eq:u_infty_def}, in the spirit of the analysis of Galerkin's method for strongly monotone operators.
\begin{lemma}[Convergence to $u_\infty$]\label{lem:uinfty_convergence}
The sequence of numerical solutions $u_k\in\Vk^s$ defined by~\eqref{eq:num_scheme} satisfies
\begin{equation}\label{eq:uinfty_converges}
\lim_{k\tends \infty}\normk{u_\infty- u_k} =0.
\end{equation}
\end{lemma}
\begin{proof}
Theorem~\ref{thm:limit_space_characterization} shows that there exists a sequence $\{v_k\}_{k\in\N}$ such that $v_k\in\Vk^s$ for each $k\in\N$ and such that $\normk{u_\infty-v_k}\tends 0$ as $k\tends \infty$. Lemma~\ref{lem:lifted_Hess_convergence} shows that $\bm{H}_k v_k\tends \bm{H}_\infty u_\infty $ and $\br_{k}(\jump{\Npw v_k})\tends \br_\infty(\jump{\Npw v})$ in $\Ldd$ as $k\tends \infty$.
Recall also that the sequence of numerical solutions is uniformly bounded, see~\eqref{eq:num_sol_bounded}, and thus Theorem~\ref{thm:weak_convergence} shows that there exists a $u_* \in \Vinfty^s$ and a $\br \in \Ldd$ such that $\br \chi_{\Omp}=\br_\infty(\jump{\Npw u_*})$ a.e.\ in $\Om$, and a subsequence $\{u_{k_j}\}_{j\in\N}$ such that $u_{k_j}\tends u_* $ in $L^2(\Om)$, $\Npw u_{k_j}\tends \Npw u_*$ in $\Ld$ and $\bm{H}_{k_j} u_{k_j}\rightharpoonup \bm{H}_\infty u_*$ and $\br_{k_j}(\jump{\Npw u_{k_j}})\rightharpoonup \br$ in $\Ldd$ as $j\tends \infty$. The sequences $\{v_{k_j}\}_{j\in\N}$ and $\{v_{k_j}-u_{k_j}\}_{j\in\N}$ therefore satisfy the hypotheses of Lemma~\ref{lem:asymptotic_consistency}.
Therefore, using the strong monotonicity of the nonlinear forms and asymptotic consistency, we get
\begin{equation*}
\begin{split}
\lim_{j\tends\infty}\frac{1}{\cmon}\norm{v_{k_j}-u_{k_j}}_{k_j}^2 & \leq \lim_{j\tends\infty} \left(A_{k_j}(v_{k_j};v_{k_j}-u_{k_j})-A_{k_j}(u_{k_j};v_{k_j}-u_{k_j}) \right)
\\ &= \lim_{j\tends \infty} A_{k_j}(v_{k_j};v_{k_j}-u_{k_j}) = A_\infty(u_\infty;u_\infty-u_*) =0,
\end{split}
\end{equation*}
where we have used the definition of numerical scheme~\eqref{eq:num_scheme}, then we have passed to the limit using~\eqref{eq:asymptotic_consistency} and finally we have used the definition of the limit problem~\eqref{eq:u_infty_def}. Therefore, the triangle inequality and the convergence of the $v_k$ to $u_\infty$ imply that $\norm{u_\infty-u_{k_j}}_{k_j}\tends 0 $ as $j\tends \infty$. Since $u_\infty \in \Vinfty^s$ is uniquely defined, the uniqueness of limits and a standard contradiction argument then imply that the whole sequence $u_k$ converges to $u_\infty$ and that \eqref{eq:uinfty_converges} holds.
\Fqed\end{proof}

The next Lemma proves that the maximum element-wise error estimator of the numerical approximations converges to zero in the limit as a consequence of the marking condition~\eqref{eq:max_marking}.
Recall that the elementwise estimators $\{\ek(u_k,K)\}_{K\in\Tk}$ are defined by~\eqref{eq:estimator_def}.

\begin{lemma}\label{lem:max_estimator_vanishes}
For any marking scheme that satisfies~\eqref{eq:max_marking}, we have
\begin{equation}\label{eq:max_estimator_vanishes}
\lim_{k\tends\infty} \max_{K\in\Tk}\ek(u_k,K) =0.
\end{equation}
\end{lemma}
\begin{proof}
The marking condition~\eqref{eq:max_marking}, the fact that any marked element is refined, and the Lipschitz continuity of $\Fg$ imply that
\begin{multline}\label{eq:max_estimator_vanishes_1}
\max_{K\in\Tk}\ek(u_k,K)^2 = \max_{K\in\Tk^-}[\ek(u_k,K)]^2  \lesssim \normk{u_\infty-u_k}^2 \\ + \max_{K\in\Tk^-}\left[\int_K \abs{\Fg[u_\infty]}^2 + \sum_{F\in\Fki;F\subset \p K} \int_F h_k^{-1}\abs{\jump{\Npw u_\infty}}^2 + \sum_{F\in\Fk; F\subset \p K}\int_{F} h_k^{-3} \abs{\jump{u_\infty}}^2 \right].
\end{multline}
Note that $\normk{u_\infty-u_k}\tends 0$ as $k\tends \infty$ as shown by Lemma~\ref{lem:uinfty_convergence}.
Lemma~\ref{lem:hjvanishes} shows that the elements of $\Tk^-$ have uniformly vanishing measures in the limit, and thus the square integrability of $\Fg[u_\infty]$ implies that $ \max_{K\in\Tk^-}\int_K \abs{\Fg[u_\infty]}^2 \tends 0$ as $k\tends \infty$.
Finally, for any $K\in\Tk^-$, the faces of $K$ belong to $\Fk\setminus \Fks$ and thus
 \begin{multline*}
  \max_{K\in\Tk^-}\left[\sum_{F\in\Fki;F\subset \p K} \int_F h_k^{-1}\abs{\jump{\Npw u_\infty}}^2+\sum_{F\in\Fk; F\subset \p K}\int_{F} h_k^{-3} \abs{\jump{u_\infty}}^2 \right]
  \\ \leq  \int_{\Fki\setminus\Fksi} h_k^{-1}\abs{\jump{\Npw u_\infty}}^2 + \int_{\Fk\setminus\Fks} h_k^{-3}\abs{\jump{u_\infty}}^2.
\end{multline*}
Using~\eqref{eq:jump_term_vanish}, we then deduce that all terms on the right-hand side of~\eqref{eq:max_estimator_vanishes_1} vanish in the limit as $k\tends \infty$, which implies~\eqref{eq:max_estimator_vanishes}. 
\Fqed\end{proof}

We are now ready to prove our the main result of this work.

\noindent\emph{Proof of Theorem~\ref{thm:main}.}
The proof consists of several steps.

\emph{Step~1.} We first show that the jump $\jump{u_\infty}_F$, respectively  $\jump{\Npw u_\infty}_F$, vanishes identically for all faces $F\in\Fp$, respectively $F\in\Fpi$, which will imply that $u_\infty \in H^2(\Om)\cap H^1_0(\Om)$.
Moreover we show that $\Fg[u_\infty]=0$ a.e. in $\Omp$. 
To do so, consider an arbitrary but fixed $K\in\Tp$; then, there exists an $\ell\in\N$ such that $K\in\cT_{k}^{1+}$ for all $k\geq \ell$. Note then that each face of $F$ of $K$ is in $\Fks$ and $h_k|_F=h_+|_F$ all $k\geq \ell$. So, for all $k\geq \ell$, the triangle inequality shows that
\begin{multline}\label{eq:u_infty_est_vanish}
\int_K \abs{\Fg[u_\infty]}^2 + \sum_{F\in\Fpi;F\subset \p K} \int_F h_+^{-1}\abs{\jump{\Npw u_\infty}}^2 + \sum_{F\in\Fp; F\subset \p K}\int_{F} h_+^{-3} \abs{\jump{u_\infty}}^2 \\
\lesssim \normk{u_\infty-u_k}^2 + [\ek(u_k,K)]^2 \leq \normk{u_\infty-u_k}^2 + \max_{K^\prime\in\Tk}[\ek(u_k,K^\prime)]^2.
\end{multline}
Then, Lemmas~\ref{lem:uinfty_convergence} and~\ref{lem:max_estimator_vanishes} imply that all of the terms in the right-hand side of~\eqref{eq:u_infty_est_vanish} above vanish in the limit as $k\tends\infty$, and thus the left-hand side, which is independent of $k$, vanishes identically.
Therefore, $\Fg[u_\infty]=0$ a.e.\ on $K$ and $\jump{\Npw u_\infty}_F=0$ for each interior face $F$ of $K$ and $\jump{u_\infty}_F=0$ for each face $F$ of $K$.
Recalling that $K\in\Tp$ was arbitrary, it follows that $\Fg[u_\infty]=0$ a.e.\ in~$\Omp$ since $\Omp$ is the countable union of all elements in $\Tp$.
Furthermore, since each face of $\Fp$ is a face of an element in $\Tp$, we also conclude that $\jump{u_\infty}_F=0$ for all faces $F\in\Fp$ and that $\jump{\Npw u_\infty}_F=0$ for all faces $F\in\Fpi$.
We then infer that $u_\infty \in H^2(\Om)\cap H^1_0(\Om)$ from the definition of the space $\HD$ in Definition~\ref{def:HD_def}, the forms of the first and second distributional derivatives of $u_\infty$ in \eqref{eq:distderiv_H100Tp} and \eqref{eq:distderiv_h2}, and from the characterization of $H^1_0(\Om)$ in~\cite[Theorem~5.29]{AdamsFournier03}.

\emph{Step~2.} The fact that $u_\infty \in H^2(\Om)\cap H^1_0(\Om)$ and Lemma~\ref{lem:uinfty_convergence} then imply that the jump seminorms of the numerical solutions vanish in the limit, i.e.\
\begin{equation}\label{eq:jump_numsol_vanish}
\lim_{k\tends \infty} \absJ{u_k} = \lim_{k\tends\infty} \absJ{u_k-u_\infty} \leq \lim_{k\tends\infty} \normk{u_k-u_\infty} =0,
\end{equation}
where it is recalled that the jump seminorm $\absJ{\cdot}$ is defined in~\eqref{eq:norm_def}.

\emph{Step~3.} We now prove that $u_\infty=u $ is the exact solution of~\eqref{eq:isaacs}. Let $z \coloneqq u_\infty-u$, and note that $z\in H^2(\Om)\cap H^1_0(\Om)$.
Since the mesh size vanishes uniformly in the limit on $\Omm$, c.f.~Lemma~\ref{lem:hjvanishes},
by using a similar quasi-interpolant as the one in the proof of Theorem~\ref{thm:limit_space_characterization}, we see  that there exists a $z_k \in \Vk^s$ such that $\normk{z_k}\lesssim \norm{z}_{H^2(\Om)}$ for all $k\in\N$, and such that $\norm{\nabla^m(z-z_k)}_{\Omm} \tends 0$ as $k \tends \infty$ for each $m\in\{0,1,2\}$ as a consequence of Lemma~\ref{lem:hjvanishes}.
Then, the strong monotonicity bound~\eqref{eq:continuous_strong_monotonicity} implies that
\begin{equation}
\norm{u_\infty-u}_{H^2(\Om)}^2 \lesssim \int_{\Om} (\Fg[u_\infty]-\Fg[u])\Delta z = \int_{\Om}\Fg[u_\infty] \Delta z.
\end{equation}
Then, by addition/subtraction of $\int_\Om \Fg[u_k]\Delta z$ and using $A_k(u_k;z_k)=0$ by~\eqref{eq:num_scheme}, we find that
\begin{equation}\label{eq:convergence_1}
\norm{u_\infty-u}_{H^2(\Om)}^2 \lesssim \int_{\Om} (\Fg[u_\infty]-\Fg[u_k])\Delta z + \int_\Om \Fg[u_k]\Delta(z-z_k)  - \theta S_k(u_k,z_k) - \Jpen(u_k,z_k).
\end{equation}
We now claim that each of the terms on the right-hand side of~\eqref{eq:convergence_1} vanish in the limit as $k\tends \infty$, which will then imply that $u_\infty = u$. The first term $\int_\Om (\Fg[u_\infty]-\Fg[u_k])\Delta z$ vanishes in the limit owing to the Lipschitz continuity of $\Fg$ and to the strong convergence $\normk{u_\infty-u}\tends 0$ as $k\tends \infty$.
Turning our attention towards the second term in the right-hand side of~\eqref{eq:convergence_1}, we find that
\begin{equation}\label{eq:convergence_2}
\begin{split}
\left\lvert\int_\Om \Fg[u_k]\Delta(z-z_k)\right\rvert & \leq \left\lvert \int_{\Omp}(\Fg[u_k]-\Fg[u_\infty])\Delta(z-z_k) \right\rvert + \left\lvert\int_{\Omm}\Fg[u_k]\Delta(z-z_k)\right\rvert
\\ & \lesssim  \normk{u_k-u_\infty}\norm{z}_{H^2(\Om)}+\normk{u_k}\norm{\Dpw(z-z_k)}_{\Omm},
\end{split}
\end{equation}
where in the first inequality we used the fact that $\Fg[u_\infty]=0$ a.e.\ in $\Omp$, and in the second inequality we have used the stability bound $\norm{\Delta(z-z_k)}_{\Omp}\lesssim\norm{z}_{H^2(\Om)}$. Therefore, we infer that $\left\lvert\int_\Om \Fg[u_k]\Delta(z-z_k)\right\rvert \tends 0$ as $k\tends \infty$ from the boundedness of the sequence of numerical solutions, see~\eqref{eq:num_sol_bounded}, from the limit $\normk{u_\infty-u_k}\tends 0$ and from the convergence $\norm{\nabla^m(z-z_k)}_{\Omm} \tends 0$ for all $m\in\{0,1,2\}$ as $k \tends \infty$.
For the last two remaining terms in~\eqref{eq:convergence_1}, we apply Theorem~\ref{thm:b_k_jump_bound} and deduce that
\begin{equation}\label{eq:convergence_3}
\left\lvert S_k(u_k,z_k)  \right\rvert + \left\lvert \Jpen(u_k,z_k)\right\rvert\lesssim C_{\sigma,\rho} \absJ{u_k}\absJ{z_k},
\end{equation}
where $C_{\sigma,\rho}$ is a constant depending only on $\sigma$ and $\rho$. We then use the convergence of the jump seminorms in~\eqref{eq:jump_numsol_vanish} and the boundedness $\absJ{z_k}\lesssim\normk{z_k}\lesssim \norm{z}_{H^2(\Om)}$ to conclude that $ S_k(u_k,z_k) \tends 0$ and $\Jpen(u_k,z_k)\tends 0$ as $k\tends \infty$. Thus we have established that all terms in the right-hand side of~\eqref{eq:convergence_1} vanish in the limit as $k\tends \infty$ and we deduce that $u_\infty=u$ is the exact solution of~\eqref{eq:isaacs}.

We then conclude that $\normk{u-u_k}=\normk{u_\infty-u_k}\tends 0$ as $k\tends \infty$, which proves the first statement in~\eqref{eq:main}. The convergence of the estimators $\ek(u_k)\tends 0$ as $k\tends 0$ then follows immediately from the global efficiency bound~\eqref{eq:global_efficiency}, thus completing the proof of \eqref{eq:main} and of Theorem~\ref{thm:main}.\qed

\end{document}